\documentclass{amsart}
\usepackage{amssymb
,amsthm
,amsmath
,amscd
,mathtools
}
\usepackage[all]{xy}
\usepackage[top=30truemm,bottom=30truemm,left=25truemm,right=25truemm]{geometry}

\allowdisplaybreaks

\newcommand{\C}{\mathbb{C}}
\newcommand{\R}{\mathbb{R}}
\newcommand{\Z}{\mathbb{Z}}
\renewcommand{\H}{\mathbb{H}}

\renewcommand{\1}{\mathbf{1}}
\newcommand{\WD}{\mathit{WD}}

\renewcommand{\AA}{\mathcal{A}}
\newcommand{\Sc}{\mathcal{S}}
\newcommand{\TT}{\mathcal{T}}
\newcommand{\UU}{\mathcal{U}}
\newcommand{\VV}{\mathcal{V}}

\newcommand{\oo}{\mathfrak{o}}
\newcommand{\p}{\mathfrak{p}}
\newcommand{\w}{\mathfrak{w}}

\newcommand{\GL}{\mathrm{GL}}
\newcommand{\SL}{\mathrm{SL}}
\newcommand{\Sp}{\mathrm{Sp}}
\newcommand{\Mp}{\mathrm{Mp}}
\newcommand{\Oo}{\mathrm{O}}
\newcommand{\SO}{\mathrm{SO}}
\newcommand{\U}{\mathrm{U}}
\newcommand{\St}{\mathrm{St}}
\newcommand{\Hom}{\mathrm{Hom}}
\newcommand{\Gal}{\mathrm{Gal}}
\newcommand{\Rep}{\mathrm{Rep}}
\newcommand{\Isom}{\mathrm{Isom}}
\newcommand{\Sym}{\mathrm{Sym}}
\newcommand{\Asai}{\mathrm{Asai}}

\newcommand{\Frob}{\mathrm{Frob}}
\newcommand{\MVW}{\mathrm{MVW}}
\newcommand{\disc}{\mathrm{disc}}
\newcommand{\temp}{\mathrm{temp}}
\newcommand{\Irr}{\mathrm{Irr}}
\newcommand{\Ind}{\mathrm{Ind}}
\newcommand{\down}{\mathrm{down}}
\newcommand{\up}{\mathrm{up}}
\newcommand{\tr}{\mathrm{tr}}
\newcommand{\id}{\mathrm{id}}

\newcommand{\resp}{resp.~}
\newcommand{\pair}[1]{\langle #1 \rangle}
\newcommand{\half}[1]{\frac{#1}{2}}
\newcommand{\cl}[1]{\widetilde{#1}}
\newcommand{\iif}{&\quad&\text{if }}
\newcommand{\other}{&\quad&\text{otherwise}}
\newcommand{\ep}{\varepsilon}
\newcommand{\lam}{\lambda}
\newcommand{\bchi}{{\boldsymbol \chi}}

\newtheorem{thm}{Theorem}[section]
\newtheorem{prop}[thm]{Proposition}
\newtheorem{lem}[thm]{Lemma}
\newtheorem{des}[thm]{Desideratum}
\newtheorem{rem}[thm]{Remark}
\newtheorem{cor}[thm]{Corollary}

\title{Local Theta correspondence of Tempered Representations and Langlands parameters}
\author{Hiraku Atobe \and Wee Teck Gan}
\date{}
\address{Department of mathematics, Kyoto University, Kitashirakawa-Oiwake-cho, Sakyo-ku, Kyoto, 606-8502, Japan}
\email{atobe@math.kyoto-u.ac.jp}
\address{Department of Mathematics, National University of Singapore, 10 Lower Kent Ridge Road, Singapore 119076}
\email{matgwt@nus.edu.sg}

\begin{document}
\maketitle

\begin{abstract}
In this paper, we give an explicit determination of the theta lifting
for symplectic-orthogonal and unitary dual pairs
over a nonarchimedean field $F$ of characteristic $0$.
We determine when theta lifts of tempered representations are nonzero, 
and determine the theta lifts in terms of the local Langlands correspondence.
\end{abstract}

\tableofcontents

\section{Introduction}\label{intro}
The theory of local theta correspondence was initiated by Roger Howe almost 40 years ago 
and has since been a major theme in representation theory and the theory of automorphic forms. 
In this paper, we shall address some basic questions concerning the local theta correspondence. 
Let us briefly recall the setup in broad strokes, 
leaving the precise exposition to the main body of the paper.  
\vskip 5pt

Let $F$ be a nonarchimedean local field of characteristic $0$ and 
let $E$ be $F$ itself or a quadratic field extension of $F$. 
Fix $\epsilon = \pm 1$ and set $\epsilon_0 = \epsilon$ if $E = F$ and 
$\epsilon_0 = 0$ if $E$ is a quadratic field. 
Consider a $-\epsilon$-Hermitian space $W_n$ over $E$ of dimension $n$ with 
associated isometry group $\U(W_n)$. 
Likewise, let $V_m$ be  an $\epsilon$-Hermitian space over $E$ of dimension $m$ with 
associated isometry group $\U(V_m)$. 
Then 
\[  
\U(W_n)  \times \U(V_m) \subset \Sp( \mathrm{Res}_{E/F}(W_n \otimes_E V_m))  
\]
forms a reductive dual pair in the above symplectic group. 
\vskip 5pt

After fixing some extra data, the dual pair  
$\U(W_n)  \times \U(V_m)$ has a Weil representation $\omega_{W_n,V_m}$. 
For an irreducible representation $\pi$ of $\U(W_n)$, 
the maximal $\pi$-isotypic quotient of $\omega_{W_n, V_m}$ has the form 
\[  
\pi \boxtimes \Theta_{W_n, V_m}(\pi)  
\]
for some smooth representation $ \Theta_{W_n, V_m}(\pi)$ of $\U(V_m)$ 
(known as the big theta lift of $\pi$).  
It was shown by Kudla that $ \Theta_{W_n, V_m}(\pi)$ has finite length (possibly zero).
The following basic result is known as the Howe duality conjecture 
(see \cite{Ho}, \cite{W1}, \cite{GT1} and \cite{GT2}):
\vskip 5pt

\begin{thm}
If $ \Theta_{W_n, V_m}(\pi)$ is nonzero, 
then it has a unique irreducible quotient $ \theta_{W_n, V_m}(\pi)$. 
\end{thm}
\vskip 5pt

We call  $\theta_{W_n, V_m}(\pi)$ the small theta lift of $\pi$ to $H(V_m)$ and shall interpret it to be $0$ if $ \Theta_{W_n, V_m}(\pi)$ is zero.  
After the above theorem, it is natural to consider the following two basic problems:
\vskip 5pt

\noindent{\bf Problem A:} Determine precisely when $\theta_{W_n, V_m}(\pi)$ is nonzero.
\vskip 10pt

\noindent{\bf Problem B:} Determine $ \theta_{W_n, V_m}(\pi)$ precisely when it is nonzero.
\vskip 10pt

\noindent In this paper, we shall address these two problems for tempered representations $\pi$.  
\vskip 10pt

To formulate answers to these two problems, especially Problem B, 
it is necessary to have some sort of classification of irreducible representations of the groups 
$\U(W_n)$ and $\U(V_m)$. 
Such a classification is provided by the local Langlands correspondence (LLC). 
The recent results of Arthur \cite{Ar}, Mok \cite{Mo}, 
Kaletha--M\'{i}nguez--Shin--White \cite{KMSW} and Gan--Savin \cite{GS} 
meant that the LLC is almost completely known for the groups considered in this paper. 
\vskip 5pt

The LLC classifies the irreducible representations $\pi$ of $\U(W_n)$ 
by their $L$-parameters $(\phi, \eta)$, where
\[
\phi \colon \WD_E \rightarrow  {}^L\U(W)    
\]
is a conjugate self-dual representation of the Weil--Deligne group 
$\WD_E = W_E \times \SL_2(\C)$ with a certain sign, 
and
\[
\eta \in \Irr(A_{\phi}) 
\]
is an irreducible character of the component group $A_{\phi}$ associated to $\phi$. 
We may think of $\phi$ as the last name of the representation $\pi$ and $\eta$ its first name.  
Thus we shall address Problems A and B in terms of the last names and first names of 
tempered representations.
\vskip 10pt

Before going on,  let us give a reformulation of Problem A.
Let $\VV = (V_m)$ be a Witt tower of $\epsilon$-Hermitian spaces over $E$ so that 
$V_{m+2}  = V_m  + \H$, where $\H$ is the hyperbolic plane. 
In particular, $m = \dim_E(V_m)$ is of a fixed parity. 
Then one has a Witt tower of local theta correspondence 
associated to the dual pair $\U(W_n)  \times \U(V_m)$.  
It is known by Kudla that the number 
\[
m_{\VV}(\pi)  = \min \{  m\ |\   \Theta_{V_m,W_n}(\pi) \not= 0 \}  
\]
is finite. 
Moreover, $\Theta_{V_m,W_n}(\pi) \not= 0$ for all $m \geq m_{\VV}(\pi)$. 
The number $m_{\VV}(\pi) $ is called the first occurrence index of $\pi$ in the Witt tower $\VV$.  
Addressing Problem A for $\pi$ is equivalent to determining the first occurrence index 
$m_{\VV}(\pi)$ of $\pi$ in every Witt tower $\VV$.
\vskip 5pt

For this purpose, the so-called conservation relation reduces our workload by half. 
More precisely, given any Witt tower $\VV$, there is a companion Witt tower $\VV' = (V'_m)$. 
We shall denote the two Witt towers by $(V_m^+)$ and $(V_m^-)$ 
and denote the first occurrence indices of $\pi$ by $m^{\pm}(\pi)$ accordingly.
The conservation relation, shown by Kudla--Rallis \cite{KR} and Sun--Zhu \cite{SZ}, says that
\[  
m^+(\pi) + m^-(\pi)  = 2 \cdot( n+ \epsilon_0 +1). 
\]
This shows that 
\[  
m^{\down}(\pi) = \min\{m^+(\pi),  m^-(\pi)\}  \leq n + \epsilon_0  +1 
\]
and
\[  
m^{\up}(\pi) =  \max\{m^+(\pi),  m^-(\pi)\}  \geq n + \epsilon_0  +1. 
\]
\vskip 5pt

To address Problems A and B, we need to determine:
\vskip 5pt
\begin{itemize}
\item the value of $m^{\down}(\pi)$ and which of $m^{\pm}(\pi)$ it is equal to;
\item the $L$-parameter $(\theta^{\pm}_m(\phi), \theta^{\pm}_m(\eta))$ of 
$\theta_{V^{\pm}_m,W_n}(\pi)$ if it is nonzero;
\end{itemize}
in terms of the $L$-parameter $(\phi, \eta)$ of $\pi$.
\vskip 5pt

Let us describe our results in the special case of discrete series representations 
when $\U(W) \times \U(V)  = \Mp_{2n}  \times \Oo_{2m+1}$. 
More precisely, let  $W_{2n}$ be the $2n$-dimensional symplectic space and 
$V^{\pm}_{2m+1}$ be the two $(2m+1)$-dimensional quadratic spaces of discriminant $1$, 
with $V_{2m+1}^+$ the split quadratic space. 
Let $\pi$ be an irreducible (genuine) discrete series representation of $\Mp(W_{2n})$, 
with $L$-parameter $(\phi, \eta)$. 
Thus
\[  
\phi =  \bigoplus_{i=1}^r  \phi_i  
\]
is a direct sum of distinct irreducible symplectic representations of the Weil--Deligne group 
$\WD_F = W_F \times \SL_2(\C)$ of $F$ and $\eta$ is a character of the component group
\[  
A_{\phi}  = \bigoplus_{i=1}^r  \Z/2\Z a_i,  
\]
which is a $\Z/2\Z$-vector space with a canonical basis $\{a_i \}$ 
indexed by the summands $\phi_i$ of $\phi$. 
Let $z_{\phi}$ denote the element $\sum_{i=1}^r a_i \in A_{\phi}$.
On the other hand, since $\Oo(V_{2m+1}^{\pm})  \cong \SO(V_{2m+1}^{\pm}) \times \Z/2\Z$, 
an irreducible representation of $\Oo(V_{2m+1}^{\pm})$ is parametrized by $(\phi', \eta', \nu')$ 
where
\vskip 5pt
\begin{itemize}
\item $\phi'$ is a symplectic representation of $\WD_F$;
\item $\eta'$ is an irreducible character of the component group $A_{\phi'}$; 
\item $\nu'  = \pm 1$ is a sign, with $\nu'=1$ corresponding to the trivial character of $\Z/2\Z$.
\end{itemize}
\vskip 5pt

Now  we consider the theta liftings of $\pi$ 
to the two Witt towers $\VV^{\pm}$. 
The conservation relation says that 
\[  
m^{\down}(\pi) + m^{\up}(\pi)  = 4n+4, 
\]
so that
\[  
m^{\down}(\pi) \leq 2n+1 \quad \text{and} \quad m^{\up}(\pi) \geq 2n+3. 
\]
Our main results in this case are summarized in the following three theorems:
\vskip 5pt

\begin{thm}
\begin{enumerate}
\item
$m^{\down}(\pi)  = m^{\epsilon}(\pi)$ if and only if $ \epsilon =  \eta( z_{\phi})$. 
We call $\VV^{  \eta( z_{\phi})}$ the going-down tower, 
and $\VV^{  -\eta( z_{\phi})} $ the going-up tower.
\item
Consider the set $\TT$ containing $0$ and all even integers $l >0$ 
satisfying the following conditions:
\begin{itemize}
\item (chain condition) $\phi$ contains $S_2 + S_4 +\dots+S_l$, 
where $S_k$ denotes the (unique) $k$-dimensional irreducible representation of $\SL_2(\C)$;
\vskip 5pt

\item (initial condition) if $e_k$ denotes the basis element of $A_{\phi}$ associated to $S_k$, 
then  $\eta(e_2) = 1$; 
\vskip 5pt

\item (alternating condition) $\eta(e_i)  = - \eta(e_{i+2})$ for even $2 \leq i \leq l-2$.  
\end{itemize}
\end{enumerate}
Let
\[  
l(\pi) = \max \, \TT.  
\]
Then
\[  
m^{\down}(\pi) =  2n+1 - l(\pi) \quad \text{and} \quad   m^{\up}(\pi) = 2n+3  + l(\pi). 
\]
\end{thm}
\vskip 5pt

\noindent In particular, the above theorem addresses Problem A. 
\vskip 5pt

\begin{thm}
Consider the going-down tower $\VV^{\eta(z_{\phi})}$. 
For each $V_{2m+1}$ in this Witt tower, with $2m+1 \geq m^{down}(\pi) = 2n+1- l(\pi)$, 
consider the theta lift $\theta_{W_{2n}, V_{2m+1}}(\pi)$ 
and let its $L$-parameter be given by 
$(\theta_{2m+1}(\phi),  \theta_{2m+1}(\eta), \nu_{2m+1}(\phi, \eta))$.  
\begin{enumerate}
\item
One has:
\[ 
\nu_{2m+1}(\phi, \eta) =   \eta(z_{\phi})  \cdot \epsilon(1/2, \phi). 
\]
\item
If $m^{\down}(\pi) \leq \dim V_{2m+1}  < 2n+1$, then 
\[  
\theta_{2m+1}(\phi)  = \phi  -  S_{2n-2m}.
\]
Hence $\theta_{2m+1}(\phi)$ is a discrete series parameter and there is a natural injection 
$A_{\theta_{2m+1}(\phi)} \hookrightarrow A_{\phi}$. 
For the basis element $a_i$ of $A_{\theta_{2m+1}(\phi)}$ 
associated to an irreducible summand $\phi_i$, 
one has
\begin{align*}
\theta_{2m+1}(\eta) (a_i)  /  \eta(a_i)  
&=  
\epsilon(1/2, \phi_i \otimes S_{2(n-m)-1}) \cdot \epsilon(1/2, \phi_i)
\\&=
\left\{\begin{aligned}
&-1	\iif \text{$\phi_i = S_{2k}$ for some $1 \leq k \leq n-m-1$}, \\
&1	\other.
\end{aligned}\right.
\end{align*}
\item
If $m = n$, then 
\[
\theta_{2m+1}(\phi)  = \phi
\quad\text{and}\quad 
\theta_{2m+1}(\eta) = \eta.
\]
Hence $\theta_{2m+1}(\phi)$ is a discrete series parameter.
\item
If $m > n$, then $\theta_{2m+1}(\pi)$ is non-tempered and 
is the unique Langlands quotient of the standard module
\[     
\times_{i=1}^{m-n}  |\cdot|^{m-n+  \frac{1}{2}-i}  \rtimes \theta_{2n+1}(\pi).    
\]
In particular, 
\[  
\theta_{2m+1}(\phi)  =   
\phi \oplus 
\left( \bigoplus_{i=1}^{m-n} |\cdot|^{m-n +\frac{1}{2} - i}  \oplus |\cdot|^{-(m-n  + \frac{1}{2} -i)} \right),
\]
so that there is a natural identification $A_{\theta_{2m+1}(\phi)}  \cong A_{\theta_{2n+1}(\phi)}$, 
and
\[  
\theta_{2m+1}(\eta)  =  \theta_{2n+1}(\eta). 
\]
\end{enumerate}
\end{thm}
\vskip 5pt

\begin{thm}
Consider the going-up tower $\VV^{-\eta(z_{\phi})}$. 
For each $V_{2m+1}$ in this Witt tower,  
consider the theta lift $\theta_{W_{2n}, V_{2m+1}}(\pi)$ 
and let its $L$-parameter be given by
$(\theta_{2m+1}(\phi),  \theta_{2m+1}(\eta),  \nu_{2m+1}(\phi, \eta))$.  
\begin{enumerate}
\item
One has:
\[ 
\nu_{2m+1}(\phi, \eta) =   \eta(z_{\phi})  \cdot \epsilon(1/2, \phi). 
\]
\item
If $\dim V_{2m+1} = m^{\up}(\pi)$, then $\theta_{2m+1}(\pi)$ is a tempered representation with
\[  
\theta_{2m+1}(\phi)  = \phi  +  S_{l(\pi) +2}, 
\]
so that there is a natural inclusion 
\[  
A_{\phi}  \hookrightarrow A_{\theta_{2m+1}(\phi)}. 
\]
For the basis element $a_i$ of $A_{\theta_{2m+1}(\phi)}$ 
associated to an irreducible summand $\phi_i$, one has
\begin{align*}
\theta_{2m+1}(\eta)(a_i)  /  \eta(a_i)  
&=
\epsilon(1/2, \phi_i \otimes S_{l(\pi)+1}) \cdot \epsilon(1/2, \phi_i)
\\&=
\left\{\begin{aligned}
&-1	\iif \text{$\phi_i = S_{2k}$ for some $1 \leq k \leq l(\pi)/2$},\\
&1	\other.
\end{aligned}\right.
\end{align*}
\item
If $\dim V_{2m+1}  >  m^{\up}(\pi)$ (so that $m-n-1 - l(\pi) > 0$), 
then $\theta_{2m+1}(\pi)$ is non-tempered and 
is the unique Langlands quotient of the standard module
\[    
\times_{i=1}^{m-n-1-l(\pi)/2} |\cdot|^{m-n+ \frac{1}{2} - i}  \rtimes \theta_{m^{\up}(\pi)}(\pi).  
\]
In particular, 
\[  
\theta_{2m+1}(\phi)  =   \phi \oplus  S_{l(\pi)+2}  \oplus 
\left( \bigoplus_{i=1}^{m-n-1 - l(\pi)/2}   |\cdot|^{m-n +\frac{1}{2} - i}  
\oplus |\cdot|^{-(m-n +\frac{1}{2} -i)} \right), 
\]
so that there is a natural identification $A_{\theta_{2m+1}(\phi)}  \cong A_{\theta_{m^{\up}(\pi)}(\phi)}$ 
and
\[  
\theta_{2m+1}(\eta)  =  \theta_{m^{\up}(\pi)}(\eta).  
\]
\end{enumerate}
\end{thm}
\vskip 5pt

\noindent Taken together, the above two theorems give  precise determination of the theta lifts of any discrete series representation $\pi$ of $\Mp(W_{2n})$. In the case of tempered $\pi$, the results are in the same spirit, though slightly more involved to state.
\vskip 5pt

We note that Problems A and B have been extensively studied by Mui\'c (\cite{Mu1}--\cite{Mu4}) 
and M{\oe}glin (\cite{Moe1}, \cite{Moe2}), at least for the symplectic-orthogonal dual pairs. 
Their work uses the M{\oe}glin--Tadi\'c classification of discrete series representations of 
classical groups in terms of supercuspidal representations. 
At that point, the M{\oe}glin--Tadi\'c classification was conditional, 
and it may be viewed as a preliminary form of the LLC. 
As such, the formulation of the answers to Problems A and B in the various papers of Mui\'c 
may seem somewhat complicated, as are the proofs. 
The formulation of our main results and their proofs are neater and more transparent. 
There are several reasons for this:
\vskip 5pt

\begin{itemize}
\item 
the LLC affords a more efficient language to describe the answers;
\vskip 5pt

\item 
the theory of local theta correspondence is in a more mature state today 
than at the time of Mui\'c's work. 
For example, the conservation relation is now known and we do exploit it to simplify life.  
\vskip 5pt

\item 
we make use of a wider spectrum of tools than Mui\'c. 
For example, we use results of Gan--Ichino \cite{GI1} on the behaviour of the standard gamma factors and
Plancherel measures in the local theta correspondence, 
as well as results of Gan--Takeda \cite{GT1} and Gan--Savin \cite{GS}. 
In the proofs of some of these results, the doubling see-saw diagram plays a crucial role.
In addition, Problems A and B in the almost equal rank case were resolved in 
\cite{GI2} for the unitary case and \cite{At} for symplectic-orthogonal case 
by the local intertwining relation given by Arthur \cite{Ar}.
Mui\'c, on the other hand, mainly made use of the computation of Jacquet modules and 
Kudla's filtration.
\end{itemize}
\vskip 5pt

However, the main innovation of this paper is the exploitation of the local Gross--Prasad conjecture
(GP), which is now established, in addressing Problems A and B. 
Recall that the GP conjecture comes in two flavours:  
the Bessel case and the Fourier--Jacobi case. 
For tempered representations, the Bessel case was proved by Waldpsurger (\cite{W2}--\cite{W5}) 
for special orthogonal groups, and Beuzart-Plessis (\cite{BP1}--\cite{BP3}) for unitary groups. 
In \cite{GI2} and \cite{At}, the Fourier--Jacobi case (for tempered representations) was 
deduced from the Bessel case by using the theta correspondence in the almost equal rank case. 
In particular, in the almost equal rank case, Problems A and B were fully addressed in 
\cite{GI2} for unitary dual pairs, \cite{At} and \cite{AG} for symplectic-orthogonal dual pairs, 
and \cite{GS} for metaplectic-orthogonal dual pairs, 
and these allow one to deduce the Fourier--Jacobi case of the GP conjecture from the Bessel case.  
In this paper, with the GP conjecture in hand, 
we turn the table around and use it to understand the theta correspondence for general dual pairs.  
 \vskip 5pt

Let us give a brief summary of the contents of this paper. 
After describing some background material on theta correspondence and the LLC 
in Sections 2 and 3, our main results are given in Section 4.    
In order not to overburden the reader with too much background material, 
we have placed the more precise description of LLC in Appendix A and B.
The local Gross--Prasad conjecture and Prasad's conjectures  
(which resolve Problems A and B  for almost equal rank dual pairs)
are placed in Appendix C and D, respectively. 
Note that in a prequel to this paper \cite{AG}, 
we have discussed the LLC for full orthogonal groups and
established the GP conjecture for full orthogonal groups.
Finally the proofs of the main results are given in Sections 5 and 6. 
\vskip 5pt

\subsection*{Acknowledgments}
This project was initiated when the first author visited National University of Singapore in March 2015. 
He would like to thank NUS for the hospitality.
He is also grateful to Professor Gordan Savin and Professor Atsushi Ichino for their helpful comments.
The project was completed when  the second author visited Kyoto University in December 2015
as Distinguished Visiting Project Professor (Japan Gateway: Kyoto University Top Global Program) 
of Center for the Promotion of Interdisciplinary Education and Research. 
The second author would like to thank Kyoto University for its generous support.
The first author is supported by JSPS KAKENHI Grant Number 26-1322.
The second author is partially supported by 
a Singapore government MOE Tier 2 grant R-146-000-175-112.

\section{Local theta correspondence}
In this section, we fix some notations.
\vskip 5 pt

\subsection{Fields}
Let $F$ be a nonarchimedean local field of characteristic $0$ and residue characteristic $p$.
Let $\oo_F$ be the ring of integers of $F$, 
$\p_F$ be the maximal ideal of $\oo_F$, 
$\varpi_F$ be a uniformizer of $\oo_F$,
and $q_F$ be the cardinality of $\oo_F/\p_F$.
The absolute value $|\cdot|_F$ on $F$ is normalized by $|\varpi_F|_F=q_F^{-1}$.
We fix a non-trivial additive character $\psi$ of $F$.
\vskip 5pt

Let $E$ be either $F$ itself or a quadratic extension of $F$,
and $\omega_{E/F}$ be the quadratic character of $F^\times$ corresponding to $E$ via 
the local class field theory.
We denote the generator of $\Gal(E/F)$ by $c$.
We define a non-trivial additive character $\psi_E$ of $E$ by $\psi_E=\psi \circ \tr_{E/F}$.
If $E\not=F$, we fix an element $\delta \in E^\times$ such that $\tr_{E/F}(\delta)=0$, 
and set 
\[
\psi^E_a(x)=\psi(\half{a}\tr_{E/F}(\delta x))
\]
for $x\in E$ and $a \in F^\times$.
If $a=1$, we simply write $\psi^E=\psi^E_1$.
One should not confuse $\psi_E$ with $\psi^E$.
If $E=F$, we set
\[
\psi_a(x)=\psi(ax)
\]
for $x\in F$ and $a \in F^\times$.
\vskip 5 pt

\subsection{Spaces}\label{spaces}
Fix $\epsilon=\pm1$ in $E^\times$.
Let 
\begin{align*}
W_n&=\text{a $-\epsilon$-Hermitian space over $E$ of dimension $n$ over $E$},\\
V_m&=\text{an $\epsilon$-Hermitian space over $E$ of dimension $m$ over $E$}.
\end{align*}
We set 
\[
l=n-m+\epsilon_0
\quad\text{with}\quad
\epsilon_0=\left\{
\begin{aligned}
&\epsilon \iif E=F,\\
&0 \iif E\not= F,
\end{aligned}
\right.
\]
and
\[
\kappa=\left\{
\begin{aligned}
&1	\iif \text{$l$ is odd}, \\
&2	\iif \text{$l$ is even}.
\end{aligned}
\right.
\]
We define the discriminant $\disc(V_m)$ and $\disc(W_n)$ as in \cite[\S 2.2]{GI1}.
Note that
\[
\disc(V_m) \in 
\left\{
\begin{aligned}
&F^\times/F^{\times2}	\iif E=F,\\
&F^\times/N_{E/F}(E^\times)	\iif \text{$E\not=F$ and $\epsilon=+1$},\\
&\delta^m \cdot F^\times/N_{E/F}(E^\times)	\iif \text{$E\not=F$ and $\epsilon=-1$}.
\end{aligned}
\right.
\]
\vskip 5 pt

\subsection{Groups}\label{group}
We will consider the isometry groups associated to the pair $(V_m,W_n)$ of 
$\pm\epsilon$-Hermitian spaces. 
More precisely, we set:
\[
G(W_n)=\left\{
\begin{aligned}
&\text{the metaplectic group $\Mp(W_n)$},\iif \text{$E=F$, $\epsilon=+1$ and $m$ is odd},\\
&\text{the isometry group of $W_n$}, \other.
\end{aligned}
\right.
\]
We define $H(V_m)$ similarly by switching the roles of $W_n$ and $V_m$.
\vskip 5pt

For a vector space $X$ over $E$, 
we denote the general linear group of $X$ by $\GL(X)$.
Let $\det_X=\det_{\GL(X)}$ be the determinant on $\GL(X)$.
\vskip 5 pt

\subsection{Representations}
Let $G$ be a $p$-adic group.
We denote the category of smooth representations of $G$ by $\Rep(G)$.
Let $\Irr(G)$ be 
the set of equivalence classes of irreducible smooth (genuine) representations of $G$.
We also denote by $\Irr_\temp(G)$ (\resp $\Irr_\disc(G)$) 
the subset of $\Irr(G)$ of classes of irreducible tempered representations
(\resp discrete series representations).
\vskip 5pt

For a parabolic subgroup $P=MN$ of $G$, let $\delta_P$ be the modulus character of $P$.
For $(\pi_0,\VV_0) \in \Rep(M)$, we define the normalized induction $\Ind_P^G(\pi_0)$ by 
the space of smooth functions $f \colon G \rightarrow \VV_0$ such that
\[
f(mng)=\delta_P(m)^{\half{1}} \cdot \pi_0(m)f(g) 
\quad \text{for $m \in M$, $n\in N$ and $g \in G$}.
\]
The group $G$ acts on $\Ind_P^G(\pi_0)$ by right translation.
For $(\pi,\VV) \in \Rep(G)$, we define the normalized Jacquet module $R_P(\pi)$ by
$R_P(\pi)=\VV/\VV(N)$, where $\VV(N)$ is the subspace generated by 
$\pi(n)v-v$ for $n\in N$ and $v\in \VV$.
Note that $\VV(N)$ is an $M$-subrepresentation of $\VV$.
The group $M$ acts on $R_P(\pi)$ by
\[
m\cdot (v \bmod \VV(N)) = \delta_P(m)^{-\half{1}} \cdot \pi(m)v \bmod \VV(N)
\]
for $m \in M$ and $v \in \VV$.
\vskip 5pt

We have the normalized induction functor
\[
\Ind_P^G \colon \Rep(M) \rightarrow \Rep(G)
\]
and the normalized Jacquet functor
\[
R_P \colon \Rep(G) \rightarrow \Rep(M).
\]
Let $\overline{P}= M\overline{N}$ be the opposite parabolic subgroup to $P$.
Then there exist two Frobenius' reciprocity formulas:
\[
\Hom_G(\pi,\Ind_P^G(\pi_0)) \cong \Hom_M(R_P(\pi),\pi_0)
\quad\text{(standard Frobenius reciprocity)}
\]
and
\[
\Hom_G(\Ind_P^G(\pi_0),\pi) \cong \Hom_M(\pi_0, R_{\overline{P}}(\pi))
\quad\text{(Bernstein's Frobenius reciprocity)}.
\]
\vskip 5 pt

\subsection{Parabolic inductions}
We shall use Tadi\'{c}'s notation for induced representations.
Let $W_n$ be a $-\epsilon$-Hermitian space, and $G(W_n)$ as in \S \ref{group}.
If $X_t$ is a $t$-dimensional isotropic subspace of $W_n$, we decompose
\[
W = X_t \oplus W_{n-2t} \oplus X_t^*,
\]
where $X_t^*$ a $t$-dimensional isotropic subspace of $W_n$
such that $X_t \oplus X_t^*$ is non-degenerate, and
$W_{n-2t}$ is the orthogonal complement of $X_t \oplus X_t^*$ in $W_n$.
We denote by $P(X_t) = L(X_t) \cdot U(X_t)$
the maximal parabolic subgroup stabilizing $X_t$, 
where $L(X_t) = \GL(X_t) \times G(W_{n-2t})$
is the Levi subgroup of $P(X_t)$ stabilizing $X_t^*$.
If $\tau \in \GL(X_t)$ and $\pi_0 \in \Irr(G(W_{n-2t}))$, 
we write
\[
\tau \rtimes \pi_0 \coloneqq \Ind_{P(X_t)}^{G(W_n)}(\tau \otimes \pi_0).
\]
More generally, a standard parabolic subgroup $P$ of $G(W)$ has the Levi factor of the form 
$\GL_{n_1}(E) \times \dots \times \GL_{n_r}(E) \times G(W_{n_0})$, and we set
\[
\tau_1 \times \dots \times \tau_r \rtimes \pi_0 \coloneqq
\Ind_P^{G(W_n)}(\tau_1 \otimes \dots \tau_r \otimes \pi_0),
\]
where $\tau_i$ is a representation of $\GL_{n_i}(E)$ and 
$\pi_0$ is a representation of $G(W_{n_0})$. 
When $G(W_n) = \Mp(W_n)$ is a metaplectic group, 
we will follow the convention of \cite[\S 2.2--2.5]{GS} for the normalized parabolic induction.
\vskip 5 pt

\subsection{Galois conjugate}
Recall that $c$ denotes the generator of $\Gal(E/F)$.
Let $X$ be a vector space over $E$ of dimension $t$.
Choose a basis $\{x_j\}$ of $X$, and we set
\[
i \colon \GL_t(E) \rightarrow \GL(X),\quad
g\mapsto [(x_1,\dots,x_t)\mapsto (x_1,\dots,x_t)g].
\]
For a representation $\tau$ of $\GL(X)$, 
we define the $c$-conjugate ${}^c\tau$ of $\tau$ by
\[
{}^c\tau(h)\coloneqq \tau(i \circ c \circ i^{-1}(h))
\]
for $h \in \GL(X)$.
Let $\{x'_j\}$ be another basis of $X$ and 
we denote by $i' \colon \GL_t(E) \rightarrow \GL(X)$ the corresponding map.
If $A \in \GL_t(E)$ satisfies
\[
(x'_1,\dots,x'_t)=(x_1,\dots,x_t) \cdot A, 
\]
then we have $i'(g)=i(A \cdot g \cdot A^{-1})$, and so that
\[
i' \circ c \circ i'^{-1}(h)=i(A \cdot {}^c A^{-1}) \cdot i \circ c \circ i^{-1}(h) \cdot i(A \cdot {}^c A^{-1})^{-1}
\]
for $h \in \GL(X)$.
This shows that the equivalence class of ${}^c\tau$ 
is independent of the choice of a basis of $X$.
\vskip 5 pt

\subsection{$\MVW$ functor}
Let $\delta$ be an $F$-linear automorphism on $W_n$ such that 
$\delta G(W_n) \delta^{-1} = G(W_n)$. 
For a representation $\pi$ of $G(W_n)$, 
we denote by $\pi^\delta$ the representation of $G(W_n)$ 
defined by conjugation, i.e., $\pi^\delta(g) = \pi( \delta g \delta^{-1})$.
The following proposition is in Chapter 4.$\mathrm{II}$.1 in \cite{MVW}.
\begin{prop}\label{mvw}
Let $\pi$ be an irreducible admissible representation of $G(W_n)$ and 
$\pi^\vee$ be the contragredient of $\pi$. 
Let $\delta$ be an $E$-conjugate linear automorphism on $W_n$ such that 
\[
\pair{\delta x, \delta y} = \pair{y,x}
\]
for $x, y \in W_n$.
Here, $\pair{-,-}$ denotes the Hermitian pairing of $W_n$.
Then, $\pi^\delta \cong \pi^\vee$.
\end{prop}

Fix $\delta$ as in Proposition \ref{mvw}.
We define a functor 
\[
\MVW \colon \Rep(G(W_n)) \rightarrow \Rep(G(W_n))
\]
by $\pi^\MVW = \pi^\delta$.
Note that $\MVW$ is independent of the choice of $\delta$.
By the definition and Proposition \ref{mvw}, we see that
\begin{itemize}
\item
$\MVW$ is an involution, i.e., $(\pi^\MVW)^\MVW \cong \pi$;
\item
$\MVW$ is a covariant functor;
\item
$\Ind_{P(X_t)}^{G(W_n)}(\tau \otimes \pi_0)^\MVW
\cong \Ind_{P(X_t)}^{G(W_n)}({}^c\tau \otimes \pi_0^\MVW)$
for $\tau \in \Irr(\GL(X_t))$ and $\pi_0 \in \Rep(G(W_{n-2t}))$;
\item
if $\pi$ is irreducible, then $\pi^\MVW \cong \pi^\vee$.
\end{itemize}
\vskip 5pt

We will use $\MVW$ in the following form.
\begin{lem}\label{sub-quot}
Let $P$ be a standard parabolic subgroup of $G(W_n)$ with
the Levi factor of the form $\GL_{n_1}(E) \times \dots \times \GL_{n_r}(E) \times G(W_{n_0})$.
Then for $\tau_i \in \Irr(\GL_{n_i}(E))$, $\pi_0 \in \Irr(G(W_{n_0}))$ and $\pi \in \Irr(G(W_{n}))$, 
the following are equivalent:
\begin{enumerate}
\item
$\pi$ is a subrepresentation of $\tau_1\times \dots \times \tau_r \rtimes \pi_0$;
\item
$\pi$ is a quotient of 
${}^c\tau_1^\vee\times \dots \times {}^c\tau_r^\vee \rtimes \pi_0$.
\end{enumerate}
\end{lem}
\begin{proof}
Use both the contragredient functor and the $\MVW$ functor.
\end{proof}
\vskip 5 pt

\subsection{Weil representations}
Let $(V,W)=(V_m,W_n)$ be as in \S \ref{spaces}.
We consider the Weil representation of the pair $G(W) \times H(V)$. 
We fix a pair of characters $\bchi=(\chi_{V_m},\chi_{W_n})$ of $E^\times$
as in \cite[\S 3.2]{GI1}.
When there is no fear of confusion,
$\chi_{V_m}$ and $\chi_{W_n}$ are simply denoted by $\chi_V$ and $\chi_W$, respectively.
Note that ${}^c \chi_V^{-1}=\chi_V$ and ${}^c \chi_W ^{-1} =\chi_W$.
Moreover, if $V_m$ (\resp $W_n$) is a symplectic space, then
$\chi_V=\1$ (\resp $\chi_W=\1$).
These data and $\psi$ give a splitting 
$G(W) \times H(V) \rightarrow \Mp(W \otimes V)$ of the dual pair. 
More precisely, see \cite{Ku2}, \cite{HKS} and \cite[\S 3.3]{GI1}.
Pulling back the Weil representation of $\Mp(W \otimes V)$ to $G(W) \times H(V)$ via this splitting, 
we obtain the associated Weil representation $\omega_{V,W,\bchi,\psi}$ of $G(W) \times H(V)$. 
We simply write $\omega_{V,W}$ for the Weil representation.
\vskip 5 pt

\subsection{Theta correspondence}
Let $\omega_{V,W}$ be the Weil representation of $G(W) \times H(V)$.
For $\pi\in\Irr(G(W))$, the maximal $\pi$-isotypic quotient of $\omega_{V,W}$ is of the form
\[
\pi \boxtimes \Theta_{V,W}(\pi),
\]
where $\Theta_{V,W}(\pi)$ is a smooth representation of $H(V)$.
We emphasize that  $\Theta_{V,W}(\pi)$ depends on $\bchi$ and $\psi$ also.
It was shown by Kudla \cite{Ku1} that $\Theta_{V,W}(\pi)$ is either zero or of finite length.

The following result is proven by Waldspurger \cite{W1} when $p\not=2$ and
by \cite{GT1} and \cite{GT2} in general.

\begin{thm}[Howe duality conjecture]\label{howe}
If $\Theta_{V,W}(\pi)$ is nonzero, then
$\Theta_{V,W}(\pi)$ has a unique irreducible quotient $\theta_{V,W}(\pi)$.
\end{thm}
\vskip 5 pt

\subsection{First occurrence and tower property}\label{FT}
Fix $\epsilon=\pm1$.
Let $W_n$ be a $-\epsilon$-Hermitian space as in \S \ref{spaces}.
For an anisotropic $\epsilon$-Hermitian space $V_{m_0}$ and $r\geq0$, 
we put
\[
V_{m_0+2r}=V_{m_0} \oplus \H^r,
\]
where $\H$ is the hyperbolic plane.
The collection
\[
\VV=\{V_{m_0+2r}\ |\ r\geq0\}
\]
is called a Witt tower of spaces.
Note that $\disc(V_m)$ and the parity of $\dim(V_m)$ 
depend only on the Witt tower $\VV$ to which $V_m$ belongs.
One can consider a tower of the theta correspondence associated to 
reductive dual pairs $\{(G(W_n),H(V_{m}))\ | \ V_m \in \VV\}$. 
For $\pi\in\Irr(G(W_n))$, we have the representation $\Theta_{V_{m},W_n}(\pi)$ of $H(V_{m})$. 
The number
\[
m_\VV(\pi) = \min\{m\ |\ \Theta_{V_{m},W_n}(\pi)\not=0\}
\]
is finite and
is called the first occurrence index of $\pi$ for the Witt tower $\VV$, 
and the representation $\theta_{V_{m_\VV(\pi)},W_n}(\pi)$ is called 
the first occurrence of $\pi$ for this Witt tower.
\vskip 5pt

The following proposition is often referred to as the tower property of theta correspondence
(see {\cite{Ku1}}).
\begin{prop}\label{tower}
Let $m_\VV(\pi)$ be the first occurrence index of $\pi$ for the Witt tower $\VV=\{V_{m}\}$.
Then we have $\Theta_{V_{m},W_n}(\pi)\not=0$ for any $m \geq m_\VV(\pi)$.
\end{prop}
\vskip 5pt

If $E\not=F$ or $\epsilon=+1$,
for a given Witt tower $\VV=\{V_m\}$, 
there exists another Witt tower $\VV'=\{V'_{m'}\}$
such that
\begin{itemize}
\item
$\dim(V_m)\equiv \dim(V'_{m'})\bmod 2$;
\item
$\disc(V_{m})=\disc(V'_{m'})$ if $E=F$ and $\epsilon =+1$.
\end{itemize}
We call $\VV'$ the companion Witt tower of $\VV$.
Also, by a companion space of $V_m$, we mean $V_m$ or $V_m'$.
For each $\pi \in \Irr(G(W_n))$, 
we may consider two first occurrence indices $m_\VV(\pi)$ and $m_{\VV'}(\pi)$.
Let $\VV^+=\{V^+_{m}\}$ be the Witt tower whose anisotropic space is
\[
\left\{
\begin{aligned}
&0	\iif \text{$E\not=F$ and $m$ is even},\\
&(E,1)	\iif \text{$E\not=F$, $m$ is odd and $\epsilon=+1$},\\
&(E,\delta)	\iif \text{$E\not=F$, $m$ is odd and $\epsilon=-1$},\\
&0	\iif \text{$E=F$, $m$ is even and $\disc(V_m)=1$},\\
&(F(\sqrt{d}), \tr_{F(\sqrt{d})/F})	
\iif \text{$E=F$, $m$ is even and $d \coloneqq \disc(V_m)\not=1$ in $F^\times/F^{\times2}$},\\
&(F, 2\disc(V_m))	\iif \text{$E=F$ and $m$ is odd}.
\end{aligned}
\right.
\]
Here, we consider $V_m$ as a vector space equipped with a suitable Hermitian pairing.
For example, by $(F, 2\disc(V_m))$, we mean the one dimensional space equipped with
the bilinear form
\[
(x,y) \mapsto 2d xy, 
\]
where $d \in F^\times$ satisfies $d \bmod F^{\times2} = \disc(V_m)$ in $F^\times / F^{\times2}$.
Note that this space has discriminant $\disc(V_m)$.
We denote the other Witt tower by $\{V^-_{m}\}$.
Then for each $\pi \in \Irr(G(W_n))$, 
we have two first occurrence indices $m^\pm(\pi) \coloneqq m_{\VV^\pm}(\pi)$.
\vskip 5pt

On the other hand, 
if $E=F$ and $\epsilon=-1$, 
then there is only a single tower of symplectic spaces $\VV=\{V_{m}\}$.
In this case, a companion space of $V_m$ is just $V_m$.
However, since $\pi$ is a representation of the orthogonal group $G(W_n)=\Oo(W_n)$, 
we may consider its twist $\pi \otimes \det$. 
Thus we have the two towers of theta lifts
\[
\Theta_{V_{m},W_{n}}(\pi)
\quad\text{and}\quad
\Theta_{V_{m},W_{n}}(\pi \otimes \det).
\]
Hence we may define two first occurrence indices for each $\pi \in \Irr(G(W_n))$.
When $n$ is odd, we define $m^\pm(\pi)$ by
\[
m^\pm(\pi) \coloneqq \min\{m\ |\ 
\text{$\Theta_{V_m,W_n}(\pi') \not=0$ with $\pi' \in \{\pi, \pi\otimes \det\}$
such that $\pi'(-\1_{W_n})=\pm \id$}
\}.
\]
When $n$ is even, we define $m^\pm(\pi)$ by
\begin{align*}
m^+(\pi) \coloneqq \min\Big\{ 
\min\{m\ |\ \Theta_{V_m,W_n}(\pi) \not=0\},\ \min\{m\ |\ \Theta_{V_m,W_n}(\pi \otimes \det) \not=0\}
\Big\},\\
m^-(\pi) \coloneqq \max\Big\{ 
\min\{m\ |\ \Theta_{V_m,W_n}(\pi) \not=0\},\ \min\{m\ |\ \Theta_{V_m,W_n}(\pi \otimes \det) \not=0\}
\Big\}.
\end{align*}
\vskip 5pt

In any case, for each $\pi \in \Irr(G(W_n))$, 
we have two first occurrence indices $m^\pm(\pi)$.
We put
\begin{align*}
m^{\up}(\pi)=\max\{m^+(\pi), m^-(\pi)\}
\quad\text{and}\quad
m^{\down}(\pi)=\min\{m^+(\pi), m^-(\pi)\}.
\end{align*}
The following proposition is often referred to as the conservation relation (see \cite{SZ}).
\begin{prop}\label{cons}
For any $\pi\in\Irr(G(W_{n}))$, 
we have 
\[
m^\up(\pi)+m^\down(\pi)=2n+2+2\epsilon_0.
\]
\end{prop}
\vskip 5pt

This relation shows that
\[
m^\up(\pi) \geq n+1+\epsilon_0
\quad\text{and}\quad
m^\down(\pi) \leq n+1+\epsilon_0.
\]
If we put
\[
l=n-m^\down(\pi)+\epsilon_0,
\]
then we have $l \geq -1$.

\section{Parametrization of irreducible representations}\label{LLC}
In this section, we explain the local Langlands correspondence (LLC) quickly.
More precisely, see Appendix \ref{app LLC}.
\vskip 5pt

Let $\WD_E=W_E \times \SL_2(\C)$ be the Weil--Deligne group of $E$.
We define $\Phi(H(V_m))$, which is a set of equivalence classes of representations of $\WD_E$, 
in the various cases as follows:
\[
\left\{
\begin{aligned}
\Phi(\Oo(V_m))&=\{ \phi \colon \WD_F \rightarrow \Sp(m-1,\C)\}/\cong,\iif \text{$m$ is odd},\\
\Phi(\Sp(V_m))&=\{ \phi \colon \WD_F \rightarrow \SO(m+1,\C)\}/\cong,\\
\Phi(\Oo(V_m))&=\{ \phi \colon \WD_F \rightarrow \Oo(m,\C)\ |\ \det(\phi)=\chi_{V}\}/\cong,
\iif \text{$m$ is even},\\
\Phi(\Mp(V_m))&=\{ \phi \colon \WD_F \rightarrow \Sp(m,\C)\}/\cong.
\end{aligned}
\right.
\]
For the unitary group $\U(m)$, we define $\Phi(\U(m))$ to be 
the set of equivalence classes of conjugate self-dual representations of 
$\WD_E$ with sign $(-1)^{m-1}$.
For the notion of conjugate self-dual representations, 
see Appendix \ref{repWD}.
\vskip 5pt

We say that $\phi \in \Phi(H(V_m))$ is tempered if $\phi(W_E)$ is bounded.
We denote by $\Phi_\temp(H(V_m))$ the subset of equivalence classes of tempered $\phi$.
For $\phi \in \Phi(H(V_m))$, we denote by $L(s,\phi)$, $\ep(s,\phi,\psi')$, and $\gamma(s,\phi,\psi')$
the $L$-, $\ep$-, and $\gamma$-factors associated to $\phi$, respectively.
Here, $\psi'$ is a non-trivial additive character of $E$.
The root number $\ep(1/2,\phi,\psi')$ is also denoted by $\ep(\phi)$ or $\ep(\phi,\psi')$.
\vskip 5pt

For an irreducible representation $\phi_0$ of $\WD_E$, 
we denote the multiplicity of $\phi_0$ in $\phi$ by $m_\phi(\phi_0)$.
We can decompose 
\[
\phi=m_1\phi_1+\dots+m_r\phi_r+\phi'+{}^c\phi'^\vee,
\]
where $\phi_1,\dots, \phi_r$ are distinct irreducible representations
of $\WD_E$ of the same type as $\phi$, $m_i=m_\phi(\phi_i)$, 
and $\phi'$ is a sum of irreducible representations of $\WD_E$
which are not of the same type as $\phi$.
We define the component group $A_\phi$ by
\[
A_\phi=\bigoplus_{i=1}^{r}(\Z/2\Z) a_i \cong (\Z/2\Z)^r.
\]
Namely, $A_\phi$ is a free $\Z/2\Z$-module of rank $r$ and $\{a_1, \dots, a_r\}$
is a basis of $A_\phi$ with $a_i$ associated to $\phi_i$.
For $a=a_{i_1}+\dots+a_{i_k} \in A_\phi$ with $1\leq i_1 < \dots < i_k \leq r$, 
we put
\[
\phi^{a}=\phi_{i_1} \oplus \dots \oplus \phi_{i_k}.
\]
Also, we set
\[
z_\phi \coloneqq \sum_{i=1}^r m_\phi(\phi_i)\cdot a_i 
=\sum_{i=1}^r m_i \cdot a_i \in A_\phi.
\]
We call $z_\phi$ the central element in $A_\phi$.
There is a homomorphism
\begin{align*}
\det \colon A_\phi \rightarrow \Z/2\Z, \quad
\sum_{i=1}^r \ep_i a_i \mapsto \sum_{i=1}^r \ep_i \cdot \dim(\phi_i) \pmod 2,
\end{align*}
where $\ep_i\in\{0,1\} = \Z/2\Z$.
\vskip 5pt

The LLC classifies $\Irr(H(V_m))$ as follows:
\begin{des}\label{desLLC}
\begin{enumerate}
\item
There exists a partition
\[
\bigsqcup_{V_m^\bullet}\Irr(H(V_m^\bullet)) = \bigsqcup_{\phi \in \Phi(H(V_m))}\Pi_\phi,
\]
where $V_m^\bullet$ runs over all companion spaces of $V_m$.
We call $\Pi_\phi$ the $L$-packet of $\phi$.
\item
$\pi \in \Irr(H(V_m^\bullet))$ is tempered
if and only if
$\pi$ belongs to $\Pi_\phi$ for tempered $\phi$.
\item
There exists a map
\[
\iota \colon \Pi_\phi \rightarrow \widehat{A_\phi},
\]
which satisfies certain character identities. 
Here, we denote by $\widehat{A_\phi}$ the Pontryagin dual of $A_\phi$.
\item
The map $\iota$ is surjective unless $H(V_m) = \Sp(V_m)$ is a symplectic group.
In this case, 
the image of $\iota$ is given by
\[
\{\eta \in \widehat{A_\phi}\ |\ \eta(z_\phi)=1\}.
\]
\item
The map $\iota$ is injective unless $H(V_m)=\Oo(V_m)$ is an odd orthogonal group (i.e., $m$ is odd).
In this case, each fiber of this map is of the form
\[
\{\pi,\ \pi \otimes \det\}.
\]
Hence the map
\[
\Pi_\phi \rightarrow \widehat{A_\phi} \times \{\pm1\},\ 
\pi \mapsto (\iota(\pi), \omega_\pi(-1))
\]
is bijective, where $\omega_\pi$ is the central character of $\pi$.
\item
Suppose that $V_m^-$ exists.
Then $\pi \in \Pi_\phi$ is a representation of $H(V_m^-)$
if and only if $\iota(\pi)(z_\phi)=-1$.
\end{enumerate}
\end{des}
\vskip 5pt

Therefore, unless $H(V_m)=\Oo(V_m)$ is an odd orthogonal group, 
$\pi \in \Irr(H(V_m))$ is parametrized by $(\phi,\eta)$,
where $\phi \in \Phi(H(V_m))$ and $\eta \in \widehat{A_\phi}$.
If $H(V_m)=\Oo(V_m)$ is an odd orthogonal group, 
$\pi \in \Irr(H(V_m))$ is parametrized by the triple $(\phi, \eta, \nu)$,
where $\phi \in \Phi(H(V_m))$, $\eta \in \widehat{A_\phi}$ and $\nu \in \{\pm1\}$.
The pair $(\phi,\eta)$ is called the $L$-parameter for $\pi$.
We also call $\phi$ and $\eta$ the last name and the first name of $\pi$, respectively.

\begin{rem}
The map $\iota \colon \Pi_\phi \rightarrow \widehat{A_\phi}$ may not be canonical.
To specify $\iota$, we need to choose a Whittaker datum for $H(V_m)$.
More precisely, see Remark \ref{non-can} below.
\end{rem}
\vskip 5pt

Suppose that $H(V_m)=\Oo(V_m)$ is an even orthogonal group (i.e., $m$ is even).
Then the following are equivalent:
\begin{itemize}
\item
$\phi \in \Phi(\Oo(V_m))$ contains an irreducible orthogonal representation of $\WD_F$
with odd dimensional;
\item
some $\pi \in \Pi_\phi$ satisfies that $\pi \not\cong \pi \otimes \det$;
\item
any $\pi \in \Pi_\phi$ satisfies that $\pi \not\cong \pi \otimes \det$.
\end{itemize}

\section{Main results}
The purpose of this paper is to describe theta lifts of tempered representations 
in terms of the local Langlands correspondence.
In this section, we state the main results over $3$ theorems.
Though we formulate the main results as $3$ theorems, 
these are proven together (in \S \ref{pf}).
\vskip 5pt

We denote by $S_r$ 
the unique irreducible algebraic representation of $\SL_2(\C)$ with dimension $r$.
The first main theorem gives an answer to Problem A in \S \ref{intro} for tempered representations.
\begin{thm}\label{main1}
Let $(V_m,W_n)$ and $\kappa\in\{1,2\}$ be as in \S \ref{spaces}, 
and $\pi \in \Irr_\temp(G(W_n))$ with $L$-parameter $(\phi,\eta)$.
\begin{enumerate}
\item
Consider the set $\TT$ containing $\kappa-2$ and all integers $l >0$ with $l \equiv \kappa \bmod 2$
satisfying the following conditions:
\begin{itemize}
\item (chain condition) $\phi$ contains $\chi_V S_{r}$ for $r=\kappa, \kappa+2, \dots, l$;
\item (odd-ness condition)
the multiplicity $m_\phi(\chi_V S_{r})$ is odd for $r=\kappa, \kappa+2, \dots, l-2$;
\item (initial condition) 
if $\kappa=2$, then
\[
\eta(e_2)=
\left\{
\begin{aligned}
&\epsilon \cdot \delta(\chi_V=\1)
\iif\text{$E=F$ and $m \not\equiv n \bmod 2$},\\
&-1
\iif\text{$E\not=F$ and $m \equiv n \bmod 2$};
\end{aligned}
\right.
\]
\item (alternating condition) 
$\eta(e_{r})=-\eta(e_{r+2})$ for $r=\kappa, \kappa+2, \dots, l-2$.
\end{itemize}
Here, $e_r$ is the element in $A_\phi$ corresponding to $\chi_V S_r$, 
and for a character $\chi$, we put
\[
\delta(\chi=\1)=\left\{
\begin{aligned}
&+1	\iif \chi=\1,\\
&-1	\other.
\end{aligned}
\right.
\]
Let
\[  
l(\pi) = \max \, \TT.  
\]
Then
\[  
m^{\down}(\pi) =  n + \epsilon_0 - l(\pi) 
\quad \text{and} \quad   
m^{\up}(\pi) = n+2+\epsilon_0 + l(\pi). 
\]
\vskip 5pt

\item
If $l(\pi)=-1$, then $m^\up(\pi)=m^\down(\pi)$.
Suppose that $l(\pi)\geq0$.
Then $\phi$ contains $\chi_V$ if $\kappa=1$.
Moreover, $m^\down(\pi)=m^\alpha(\pi)$ if and only if 
\[
\alpha=\left\{
\begin{aligned}
&\eta(z_\phi+e_1)
\iif \kappa=1,\\
&\eta(z_\phi)\cdot \ep(\phi)\cdot \ep(\phi \otimes \chi_V) \cdot \chi_V(-1)^{\half{n}}
\iif \text{$E=F$, $m\not\equiv n \bmod 2$ and $\epsilon=+1$},\\
&\eta(z_\phi)\cdot \ep(\phi)
\iif \text{$E=F$, $m\not\equiv n \bmod 2$ and $\epsilon=-1$},\\
&\eta(z_\phi) \cdot \ep(\phi \otimes \chi_V^{-1}, \psi_2^E)
\iif \text{$E\not=F$ and $m\equiv n \bmod 2$}.
\end{aligned}
\right.
\]
\end{enumerate}
We call $\VV^\down\coloneqq\VV^{\alpha}$ (\resp $\VV^\up\coloneqq\VV^{-\alpha}$)
the going-down tower (\resp the going-up tower)
with respect to $\pi$.
\end{thm}

\begin{rem}
Recall that when $(G(W_n),H(V_m)) = (\Oo(W_n), \Sp(V_{m}))$ with even $n$, 
by the definition, $m^\down(\pi) = m^+(\pi)$ for each $\pi \in \Irr(\Oo(W_n))$ (see \S \ref{FT}).
In this case, (2) asserts that 
if $\pi \in \Irr(\Oo(W_n))$ satisfies that 
$\Theta_{V_m,W_n}(\pi) \not=0$ and $\Theta_{V_m,W_n}(\pi \otimes \det) =0$ for 
some $m \leq n$, then the $L$-parameter $(\phi,\eta)$ of $\pi$ satisfies that
$\phi \supset \1$ and $\eta(z_\phi + e_1) =1$.
This follows from Prasad's conjecture (Theorem \ref{PA} below).
\end{rem}
\vskip 5pt

The proof of Theorem \ref{main1} is given in \S \ref{pf}.
We give an indication for the relevant result.
To prove (1), it is enough to show the following two statements:
\begin{itemize}
\item
If $\Theta_{V_m^\bullet,W_n}(\pi) \not=0$, then $l\coloneqq n-m+\epsilon_0 \in \TT$.
\item
$n-m^\down(\pi)+\epsilon_0+2 \not\in \TT$.
\end{itemize}
For the first assertion, 
(chain condition) and (odd-ness condition) follow from Corollary \ref{nus}, and
(initial condition) and (alternating condition) follow from Proposition \ref{alt}.
The second assertion follows from Corollary \ref{temp2}, Proposition \ref{non-alt}
and Prasad's conjecture (Theorem \ref{PA}).
The assertion (2) follows from Prasad's conjecture (Theorem \ref{PA}) together with
a comparison of central elements $z_\phi$ (Proposition \ref{center1})
unless $E=F$, $m\not\equiv n \bmod 2$ and $\epsilon=-1$.
In this case, we compare the central character of $\pi \in \Irr(\Oo(W_n))$
with the central element $z_\phi$ (Proposition \ref{sign2}). 
\vskip 5pt

The second and third main theorems describe the $L$-parameter for $\theta_{V_m,W_n}(\pi)$.
\begin{thm}\label{main2}
Let $(V_m,W_n)$ and $\kappa\in\{1,2\}$ be as in \S \ref{spaces}, 
and $\pi \in \Irr_\temp(G(W_n))$ with $L$-parameter $(\phi,\eta)$.
Assume that $V_m$ belongs to the going-down tower $\VV^\down$,
$m\geq m^\down(\pi)$ and $m\equiv m^\down(\pi) \bmod 2$.
Put $m_1=n+\epsilon_0+2-\kappa$.
Let $(\theta_m(\phi),\theta_m(\eta))$ be the $L$-parameter for $\theta_{V_m,W_n}(\pi)$.
\begin{enumerate}
\item
If $m^\down(\pi) \leq m < m_1$, then
\[
\theta_m(\phi)=(\phi \otimes \chi_V^{-1}\chi_W)-\chi_W S_l,
\]
where $l=n-m+\epsilon_0>0$.
In particular, there is a canonical injection $A_{\theta_m(\phi)} \hookrightarrow A_{\phi}$.
If $l=1$, then we have $\eta | A_{\theta_{m}(\phi)} = \theta_{m}(\eta)$.
If $l>1$, then $\theta_m(\eta)(a)/\eta(a)$ is equal to
\[
\left\{
\begin{aligned}
&\ep(\phi^a\chi_{V}^{-1} \otimes S_{l-1}) 
\cdot \ep(\phi^a) \cdot \chi_{V}(-1)^{\half{1}\dim(\phi^a)}
\iif \text{$E=F$, $\epsilon = +1$ and $m$ is odd},\\
&\ep(\phi^a\chi_{V}^{-1} \otimes S_{l-1}) 
\cdot \ep(\phi^a\chi_{W}) \cdot \chi_{W}(-1)^{\half{1}\dim(\phi^a)}
\iif \text{$E=F$, $\epsilon = -1$ and $n$ is odd},\\
&\ep(\phi^a\chi_{V}^{-1} \otimes S_{l-1}) \cdot \det(\phi^a \chi_V^{-1})(-1)^{\half{l-1}}
\iif \text{$E=F$ and $m$, $n$ are even},\\
&\ep(\phi^a\chi_{V}^{-1} \otimes S_{l-1},\psi^E_{2})
\iif \text{$E\not=F$},\\
\end{aligned}
\right.
\]
for any $a \in A_{\theta_m(\phi)} \subset A_{\phi}$.
\item
If $m=m_1$ and $\kappa=1$, then
\[
\theta_{m_1}(\phi)=(\phi \otimes \chi_V^{-1}\chi_W) \oplus \chi_W.
\]
In particular, there is a canonical injection $A_{\phi} \hookrightarrow A_{\theta_{m_1}(\phi)}$.
Moreover, we have $\theta_{m_1}(\eta) | A_{\phi} =\eta$.
\item
If $m=m_1$ and $\kappa=2$, then
\[
\theta_{m_1}(\phi)=\phi \otimes \chi_V^{-1}\chi_W.
\]
In particular, there is a canonical identification $A_{\phi} = A_{\theta_{m_1}(\phi)}$. 
Moreover, $\theta_m(\eta)(a)/\eta(a)$ is equal to
\[
\left\{
\begin{aligned}
&\ep(\phi^a) \cdot \ep(\phi^a \otimes \chi_V^{-1}\chi_W) \cdot 
(\chi_V^{-1}\chi_W)(-1)^{\half{1}\dim(\phi^a)}
\iif \text{$E=F$},\\
&\ep(\phi^a \otimes \chi_V^{-1}, \psi^E_2)
\iif \text{$E\not=F$}
\end{aligned}
\right.
\]
for any $a \in A_{\theta_{m_1}(\phi)} = A_{\phi}$.
\item
If $m>m_1$, then $\theta_m(\phi)$ is equal to
\[
\theta_{m_1}(\phi) \oplus 
\left(\bigoplus_{i=1}^{(m-m_1)/2}\left(\chi_W|\cdot|_E^{\half{m-n-\epsilon_0+1}-i} \oplus
\chi_W|\cdot|_E^{-\half{m-n-\epsilon_0+1}+i}\right)\right).
\]
In particular, there is a canonical identification $A_{\theta_m(\phi)} = A_{\theta_{m_1}(\phi)}$.
Moreover, we have $\theta_{m}(\eta) | A_{\theta_{m_1}(\phi)} =\theta_{m_1}(\eta)$.
\item
If $(G(W_n),H(V_m))=(\Mp(W_n),\Oo(V_m))$ with odd $m$, then
$\theta_{V_m,W_n}(\pi)$ is parametrized by $(\theta_m(\phi),\theta_m(\eta),\nu_m(\phi,\eta))$ with
\[
\nu_m(\phi,\eta)=
\eta_\pi(z_{\phi}) \cdot \ep(\phi) \cdot \chi_V(-1)^{\half{n}}.
\]
\end{enumerate}
\end{thm}

\begin{rem}
In Theorem \ref{main1} (2), we note that 
\[
[A_{\theta_{m_1}(\phi)} : A_\phi] = 
\left\{
\begin{aligned}
&1	\iif	\text{$\phi$ contains $\chi_V$},\\
&2	\iif	\text{$\phi$ does not contain $\chi_V$}.
\end{aligned}
\right.
\]
If $\phi$ does not contain $\chi_V$, 
then $m^+(\pi) = m^-(\pi) = m_1$ for any $\pi \in \Pi_\phi$ by Theorem \ref{main1} (1), 
and $z_{\theta_{m_1}(\phi)}$ is not contained in $A_\phi$.
The value $\theta_{m_1}(\eta)(z_{\theta_{m_1}(\phi)})$ is determined by
Desideratum \ref{desLLC} (4) or (6).
\end{rem}

The assertion (1) will be shown in \S \ref{corres char}.
The assertions (2) and (3) are the (almost) equal rank cases
(Theorems \ref{Mp}, \ref{PE} and \ref{PA}).
The assertion (4) follows from \cite[Proposition 3.2]{GT1}
(see Proposition \ref{L-quot} below).
The assertion (5) is Proposition \ref{sign}.
\vskip 5 pt

\begin{thm}\label{main3}
Let $(V_m,W_n)$ and $\kappa\in\{1,2\}$ be as in \S \ref{spaces}, 
and $\pi \in \Irr_\temp(G(W_n))$ with $L$-parameter $(\phi,\eta)$.
Assume that $V_m$ belongs to the going-up tower $\VV^\up$
and 
\[
m \geq m^\up(\pi)\geq n+\epsilon_0+2
\quad \text{i.e.,}\quad
l(\pi)\geq0.
\]
Let $(\theta_m(\phi),\theta_m(\eta))$ be the $L$-parameter for 
$\theta_{V_m,W_n}(\pi)$.
We put $l=m-n-\epsilon_0-2 \geq 0$.
\begin{enumerate}
\item
Suppose that $m=m^\up(\pi)$ so that $l=l(\pi)$.
If $l=0$ or $m_\phi(\chi_V S_l)$ is odd, then
\[
\theta_{m}(\phi)=(\phi \otimes \chi_V^{-1}\chi_W) \oplus \chi_W S_{l+2},
\]
so that $\theta_{V_m,W_n}(\pi)$ is tempered.
In particular, there is a canonical injection $A_{\phi} \hookrightarrow A_{\theta_{m}(\phi)}$.
Moreover, $\theta_{m}(\eta)(a)/\eta(a)$ is equal to
\[
\left\{
\begin{aligned}
&\ep(\phi^a\chi_{V}^{-1} \otimes S_{l+1}) 
\cdot \ep(\phi^a) \cdot \chi_{V}(-1)^{\half{1}\dim(\phi^a)}
\iif \text{$E=F$, $\epsilon = +1$ and $m$ is odd},\\
&\ep(\phi^a\chi_{V}^{-1} \otimes S_{l+1}) 
\cdot \ep(\phi^a\chi_{W}) \cdot \chi_{W}(-1)^{\half{1}\dim(\phi^a)}
\iif \text{$E=F$, $\epsilon = -1$ and $n$ is odd},\\
&\ep(\phi^a\chi_{V}^{-1} \otimes S_{l+1}) \cdot \det(\phi^a\chi_V^{-1})(-1)^{\half{l+1}}
\iif \text{$E=F$ and $m$, $n$ are even},\\
&\ep(\phi^a\chi_{V}^{-1} \otimes S_{l+1},\psi^E_{2})
\iif \text{$E\not=F$},\\
\end{aligned}
\right.
\]
for $a \in A_\phi \subset A_{\theta_{m}(\phi)}$.
\item
Suppose that $m=m^\up(\pi)$ so that $l=l(\pi)$.
If $l>0$ and $m_\phi(\chi_V S_l)=2h>0$, then
\[
\theta_{m^\up}(\phi)=
\Big((\phi \otimes \chi_V^{-1}\chi_W)-\chi_WS_l \Big) \oplus 
\Big( \chi_W S_{l+1} \otimes (|\cdot|_E^{\half{1}}+|\cdot|_E^{-\half{1}}) \Big),
\]
so that $\theta_{V_m,W_n}(\pi)$ is not tempered.
In this case,  
$\pi \subset \chi_V \St_l \times \dots \times \chi_V \St_l \rtimes \pi_0$, 
where $\pi_0 \in \Irr_\temp(G(W_{n-2lh}))$ has the $L$-parameter $(\phi_0,\eta_0)$
given by $\phi_0=\phi-(\chi_V \St_l)^{\oplus 2h}$ and $\eta_0=\eta|A_{\phi_0}$.
Then 
\[
m_0\coloneqq m^\up(\pi_0)=m-2lh-2 
\quad\text{and}\quad
\theta_{m_0}(\phi_0)=(\phi \otimes \chi_V^{-1}\chi_W)-(\chi_WS_l)^{\oplus (2h-1)}.
\]
In particular, there is a canonical identification 
$A_{\theta_{m_0}(\phi_0)} = A_{\theta_{m}(\phi)}$.
Moreover, we have 
$\theta_{m}(\eta) | A_{\theta_{m_0}(\phi_0)} =\theta_{m_0}(\eta_0)$.
\item
Suppose that $m>m_1\coloneqq m^\up(\pi)$.
Then $\theta_m(\phi)$ is equal to
\[
\theta_{m_1}(\phi)\oplus
\left(\bigoplus_{i=1}^{(m-m_1)/2}\left(\chi_W|\cdot|_E^{\half{m-n-\epsilon_0+1}-i}
\oplus \chi_W|\cdot|_E^{-\half{m-n-\epsilon_0+1}+i}\right)\right).
\]
In particular, there is a canonical identification $A_{\theta_m(\phi)} = A_{\theta_{m_1}(\phi)}$.
Moreover, we have $\theta_{m}(\eta) | A_{\theta_{m_1}(\phi)} =\theta_{m_1}(\eta)$.
\item
If $(G(W_n),H(V_m))=(\Mp(W_n),\Oo(V_m))$ with odd $m$, then
$\theta_{V_m,W_n}(\pi)$ is parametrized by $(\theta_m(\phi),\theta_m(\eta),\nu_m(\phi,\eta))$ with
\[
\nu_m(\phi,\eta)=
\eta_\pi(z_{\phi}) \cdot \ep(\phi) \cdot \chi_V(-1)^{\half{n}}.
\]
\end{enumerate}
\end{thm}

\begin{rem}
Note that in (1), $A_{\phi}$ can have index $2$ in $A_{\theta_{m}(\phi)}$.
In this case, we see that 
\[
A_{\theta_{m}(\phi)} =A_\phi \oplus (\Z/2\Z)e_{l(\pi)+2}.
\] 
By Theorem \ref{main1} (1), we have $\theta_{m}(\eta)(e_{l(\pi)+2}) = - \theta_{m}(\eta)(e_{l(\pi)})$.
Together with this equation, 
we see that (1) describes $\theta_{m^\up}(\eta)$ completely.
\end{rem}
\vskip 5pt

Under the assumption of Theorem \ref{main3} (1), 
we will show that $\theta_{V_m,W_n}(\pi)$ is tempered (Corollary \ref{temp2}).
If we knew the temperedness of $\theta_{V_m,W_n}(\pi)$, 
we obtain Theorem \ref{main3} (1) by applying Theorem \ref{main2} (1)
to $\theta_{V_m,W_n}(\pi)$.
The assertions (2) and (3) will be shown in \S \ref{first} and \S \ref{higher}, respectively.
The assertion (4) is Propositions \ref{sign}.
\vskip 5pt

The twisted epsilon factors appearing in Theorems \ref{main2} and \ref{main3} 
can be computed by using the following lemma.
\begin{lem}
Let $l \geq 3$ be an integer, $\chi_V$ be a character of $E^\times$ and
$\phi$ be a representation of $\WD_E$ such that
$\phi \chi_V^{-1}$ is (conjugate) self-dual with sign $(-1)^{l-1}$.
\begin{enumerate}
\item
If $E=F$ and $l$ is even, then
\[
\ep(\phi\chi_V^{-1} \otimes S_{l-1})
= (-1)^{m_\phi(\chi_V S_{l-2}) + \dots + m_\phi(\chi_V S_{2})} \cdot \ep(\phi\chi_V^{-1}).
\]
\item
If $E=F$ and $l$ is odd, then
\[
\ep(\phi\chi_V^{-1} \otimes S_{l-1}) \cdot \det(\phi\chi_V^{-1})(-1)^{\half{l-1}}
= (-1)^{m_\phi(\chi_V S_{l-2}) + \dots + m_\phi(\chi_V S_{1})}.
\]
\item
If $E\not=F$ and $l$ is even, then
\[
\ep(\phi\chi_V^{-1} \otimes S_{l-1}, \psi_{2}^E)
= (-1)^{m_\phi(\chi_V S_{l-2}) + \dots + m_\phi(\chi_V S_{2})} \cdot \ep(\phi\chi_V^{-1}, \psi_{2}^E).
\]
\item
If $E\not=F$ and $l$ is odd, then
\[
\ep(\phi\chi_V^{-1} \otimes S_{l-1}, \psi_{2}^E) 
= (-1)^{m_\phi(\chi_V S_{l-2}) + \dots + m_\phi(\chi_V S_{1})}.
\]
\end{enumerate}
\end{lem}
\begin{proof}
This follows from Lemma \ref{crit}.
\end{proof}

\section{Irreducibility and temperedness of theta lifts}
In this section, we recall Kudla's filtration of the normalized Jacquet module of Weil representations,
and prove the irreducibility and temperedness of theta lifts.
\vskip 5pt

\subsection{Kudla's filtration and irreducibility of big theta lifts}
Let $(V_m,W_n)$ be a pair of spaces as in \S \ref{spaces}.
We denote the anisotropic space in the Witt tower $\VV=\{V_{m}\}$ by $V_{\mathrm{an}}$.
Decompose 
\[
W_n=X_k + W_{n-2k} + X_k^*
\quad\text{and}\quad
V_m=Y_a + V_{m-2a} + Y_a^*,
\]
where $X_k, X_k^*$ (\resp $Y_a, Y_a^*$) are $k$-dimensional (\resp $a$-dimensional)
isotropic subspaces of $W_n$ (\resp $V_m$) such that $X_k + X_k^*$ (\resp $Y_a + Y_a^*$)
is non-degenerate, and
$W_{n-2k}$ (\resp $V_{m-2a}$) is the orthogonal complement of 
$X_k + X_k^*$ (\resp $Y_a + Y_a^*$) in $W_n$ (\resp $V_m$).
Let $P(X_k)$ (\resp $Q(Y_a)$) be the maximal parabolic subgroup of 
$G(W_n)$ (\resp $H(V_m)$) stabilizing $X_k$ (\resp $Y_a$).
We denote the normalized Jacquet functor with respect to 
a parabolic subgroup $P$ by $R_P$.
\vskip 5pt

The following lemma is called Kudla's filtration.
\begin{lem}[{\cite{Ku1}}]\label{kudla}
The normalized Jacquet module $R_{P(X_k)}(\omega_{V_m,W_n})$ 
has an equivariant filtration
\[
R_{P(X_k)}(\omega_{V_m,W_n})=R^0 \supset R^1 \supset \dots \supset R^k \supset R^{k+1}=0,
\]
whose successive quotient $J^a=R^a/R^{a+1}$ is described as follows:
\[
J^a=\Ind_{P(X_{k-a},X_k) \times G(W_{n-2k}) \times Q(Y_a)}
^{\GL(X_k) \times G(W_{n-2k}) \times H(V_{m})}
(\chi_V|{\det}_{X_{k-a}}|_E^{\lam_{k-a}} \otimes \Sc(\Isom(Y_a,X'_a)) \otimes 
\omega_{V_{m-2a},W_{n-2k}}),
\]
where
\begin{itemize}
\item
$\lam_{k-a}=(m-n+k-a-\epsilon_0)/2$;
\item
$X_k=X_{k-a}+X'_a$ with $\dim(X_{k-a})=k-a$ and $\dim(X'_a)=a$, and
$P(X_{k-a},X_k)$ is the maximal parabolic subgroup of $\GL(X_k)$ stabilizing $X_{k-a}$;
\item
$\Isom(Y_a,X'_a)$ is the set of invertible $E$-conjugate linear maps from $Y_a$ to $X'_a$ and 
$\Sc(\Isom(Y_a,X'_a))$ is the space of locally constant compactly supported functions on 
$\Isom(Y_a,X'_a)$;
\item
$\GL(X'_a) \times \GL(Y_a)$ acts on $\Sc(\Isom(Y_a,X'_a))$ as
$((g,h) \cdot f)(x)=\chi_V(\det(g))\chi_W(\det(h))f(g^{-1}\cdot x \cdot h)$
for $(g,h) \in \GL(X'_a) \times \GL(Y_a)$, $f \in \Sc(\Isom(Y_a,X'_a))$ and $x \in \Isom(Y_a,X'_a)$.
\end{itemize}
If $m-2a < \dim(V_{\mathrm{an}})$, we interpret $R^a$ and $J^a$ to be $0$.
\end{lem}

For a representation $\UU$ of a totally disconnected locally compact group $G$,
we denote by $\UU_\infty$ the smooth part of $\UU$, i.e.,
the $G$-submodule of smooth vectors in $\UU$.
Note that for $\pi \in \Irr(G(W_n))$, we have an isomorphism
\[
\Hom_{G(W_{n})}(\omega_{V_{m},W_{n}}, \pi)_\infty \cong \Theta_{V_m,W_n}(\pi)^\vee
\]
as representations of $H(V_m)$.
\vskip 5pt

The following proposition is useful.
\begin{prop}\label{Ja}
Assume that $l=n-m+\epsilon_0>0$ and $k>0$.
Let $\pi_0$ be an admissible representation of $G(W_{n-2k})$, and 
$\tau$ be an irreducible essentially discrete series representation of $\GL(X_k)$. 
Then $\Hom_{\GL(X_k) \times G(W_{n-2k})}(J^a, \chi_V {}^c\tau^\vee \otimes \pi_0)_\infty$
is isomorphic to 
\[
\left\{
\begin{aligned}
&\Ind_{Q(Y_k)}^{H(V_m)}\left(\chi_W^{-1} \tau^\vee \otimes
\Hom_{G(W_{n-2k})}(\omega_{V_{m-2k},W_{n-2k}}, \pi_0)_\infty\right)
\iif a=k,\\
&\Ind_{Q(Y_{k-1})}^{H(V_m)}\left(\chi_W^{-1} \St_{k-1}|{\det}_{Y_{k-1}}|_E^{\half{k-l+1}} \otimes
\Hom_{G(W_{n-2k})}(\omega_{V_{m-2k+2},W_{n-2k}}, \pi_0)_\infty\right)
\iif \text{$a=k-1$ and $\tau=\St_{k}|{\det}_{X_k}|_E^{\half{l-k}}$},\\
&0	\other
\end{aligned}
\right.
\]
as representations of $H(V_m)$.
\end{prop}
\begin{proof}
We put $\tau'={}^c\tau^\vee$.
For $a=k$,
it is easy to see that
\[
\Hom_{\GL(X_k) \times G(W_{n-2k})}(J^k, \chi_V \tau' \otimes \pi_0)_\infty
\cong
\Ind_{Q(Y_k)}^{H(V_m)}\left(\chi_W^{-1} \tau^\vee \otimes
\Hom_{G(W_{n-2k})}(\omega_{V_{m-2k},W_{n-2k}}, \pi_0)_\infty\right)
\]
(c.f., \cite[p. 1674--1676]{GS}).
\vskip 5pt

Next, we assume that $a<k$.
By Bernstein's Frobenius reciprocity, we have
\begin{align*}
&\Hom_{\GL(X_k) \times G(W_{n-2k})}(J^a, \chi_V \tau' \otimes \pi_0)\\
&\cong
\Hom_{\GL(X_{k-a}) \times \GL(X'_a) \times G(W_{n-2k})}
\left(\chi_V|{\det}_{X_{k-a}}|_E^{\lam_{k-a}} \otimes \Sc(\Isom(Y_a,X'_a)) \otimes 
\omega_{V_{m-2a},W_{n-2k}}, R_{\overline{P(X_{k-a},X_k)}}(\chi_V\tau') \otimes \pi_0\right),
\end{align*}
where $\overline{P(X_{k-a},X_k)}$ is the parabolic subgroup of $\GL(X_k)$ 
opposite to $P(X_{k-a},X_k)$.
By \cite[Proposition 9.5]{Z}, the normalized Jacquet module 
$R_{\overline{P(X_{k-a},X_k)}}(\chi_V\tau')$ is given by
\[
R_{\overline{P(X_{k-a},X_k)}}(\chi_V\tau') \cong
\chi_V \tau_1 |{\det}_{X_{k-a}}|_E^{e_1} \otimes \chi_V \tau_2 |{\det}_{X'_{a}}|_E^{e_2},
\]
where $\tau_1$ (\resp $\tau_2$) is an irreducible (unitary) discrete series representation
of $\GL(X_{k-a})$ (\resp $\GL(X'_a)$),
and $e_1, e_2 \in \R$ such that 
\[
e_1 < e_2
\quad\text{and}\quad
e_1 \cdot (k-a) + e_2 \cdot a=0.
\]
Since $\GL(X_{k-a})$ acts on 
$\chi_V|{\det}_{X_{k-a}}|_E^{\lam_{k-a}} \otimes \Sc(\Isom(Y_a,X'_a)) \otimes 
\omega_{V_{m-2a},W_{n-2k}}$ by the character $\chi_V|{\det}_{X_{k-a}}|_E^{\lam_{k-a}}$, 
if $\Hom_{\GL(X_k) \times G(W_{n-2k})}(J^a, \chi_V \tau' \otimes \pi_0)\not=0$, 
then we must have $k-a=1$.
Moreover, by \cite[p. 105]{Z}, we must have $\tau'=\St_k|{\det}_{X_k}|_E^e$ for some $e \in \R$.
Then we have
\[
e_1=e-\half{k-1}
\quad\text{and}\quad
e_2=e+\half{1}.
\]
We must have $e_1=\lam_1$ so that $e=(k-l)/2$.
In this case, we have
\begin{align*}
&\Hom_{\GL(X_k) \times G(W_{n-2k})}(J^{k-1}, \chi_V \tau' \otimes \pi_0)_\infty\\
&\cong
\Hom_{\GL(X'_{k-1}) \times G(W_{n-2k})}
(\Sc(\Isom(Y_{k-1},X'_{k-1})) \otimes 
\omega_{V_{m-2k+2},W_{n-2k}}, \chi_V\St_{k-1}|{\det}_{X'_{k-1}}|_E^{\half{k-l+1}} \otimes \pi_0
)_\infty,\\
&\cong
\Ind_{Q(Y_{k-1})}^{H(V_m)}\left(\chi_W^{-1} \St_{k-1}|{\det}_{X'_{k-1}}|_E^{\half{k-l+1}} \otimes
\Hom_{G(W_{n-2k})}(\omega_{V_{m-2k+2},W_{n-2k}}, \pi_0)_\infty\right)
\end{align*}
(c.f., \cite[p. 1674--1676]{GS}).
Hence the proposition.
\end{proof}

\begin{cor}\label{pi0}
We put $n_0=n-2k$ and $m_0=m-2k$.
Let $\pi \in \Irr(G(W_n))$, $\pi_0\in \Irr(G(W_{n_0}))$ and 
$\tau$ be an irreducible essentially discrete series representation of $\GL(X_k)$. 
Assume that 
\begin{itemize}
\item
$l=n-m+\epsilon_0>0$;
\item
$\tau \not\cong \St_k|{\det_{X_k}}|_E^{\half{l-k}}$;
\item
$\Ind_{P(X_k)}^{G(W_n)}(\chi_V \tau \otimes \pi_0) \twoheadrightarrow \pi$.
\end{itemize}
Then we have 
\[
\Ind_{Q(Y_k)}^{H(V_m)}(\chi_W \tau \otimes \Theta_{V_{m_0},W_{n_0}}(\pi_0))
\twoheadrightarrow \Theta_{V_m,W_n}(\pi).
\]
\end{cor}
\begin{proof}
By Lemma \ref{sub-quot}, we have
$\pi \hookrightarrow \Ind_{P(X_k)}^{G(W_n)}(\chi_V {}^c\tau^\vee \otimes \pi_0)$.
Hence we have
\begin{align*}
\Theta_{V_m,W_n}(\pi)^\vee
&\cong \Hom_{G(W_n)}(\omega_{V_m,W_n}, \pi)_\infty\\
&\hookrightarrow \Hom_{G(W_n)}(\omega_{V_m,W_n}, 
\Ind_{P(X_k)}^{G(W_n)}(\chi_V {}^c\tau^\vee \otimes \pi_0))_\infty\\
&\cong \Hom_{\GL(X_k) \times G(W_{n_0})}(R_{P(X_k)}(\omega_{V_m,W_n}), 
\chi_V {}^c\tau^\vee \otimes \pi_0)_\infty.
\end{align*}
Since $\tau \not\cong \St_k|{\det}_{X_k}|_E^{\half{l-k}}$, 
by Proposition \ref{Ja}, we have
\begin{align*}
\Hom_{\GL(X_k) \times G(W_{n_0})}(R_{P(X_k)}(\omega_{V_m,W_n}), 
\chi_V {}^c\tau^\vee \otimes \pi_0)_\infty
&\hookrightarrow \Hom_{\GL(X_k) \times G(W_{n_0})}(J^k, \chi_V {}^c\tau^\vee \otimes \pi_0)_\infty\\
&\cong \left(
\Ind_{Q(Y_k)}^{H(V_m)}(\chi_W \tau \otimes \Theta_{V_{m_0},W_{n_0}}(\pi_0))
\right)^\vee.
\end{align*}
Taking the contragredient functor, we get the corollary.
\end{proof}

Corollary \ref{pi0} implies an irreducibility condition of big theta lifts.
\begin{prop}\label{irred}
Let $\pi \in \Irr(G(W_{n}))$ whose last name is $\phi \in \Phi(G(W_n))$.
Assume that 
\begin{itemize}
\item
$\pi$ is tempered;
\item
$\Theta_{V_m,W_n}(\pi)\not=0$ for $l=n-m+\epsilon_0>0$;
\item
$\phi$ contains $\chi_V S_l$ with multiplicity one.
\end{itemize}
Then $\Theta_{V_m,W_n}(\pi)$ is irreducible and tempered.
\end{prop}
\begin{proof}
We prove this corollary by induction on $n$.
If $\pi$ is a discrete series representation, 
then by a similar argument to \cite[Proposition C.1]{GI1}, 
we see that
all irreducible subquotients of $\Theta_{V_m,W_n}(\pi)$ are 
discrete series representations.
Hence $\Theta_{V_m,W_n}(\pi)$ is a direct sum of irreducible discrete series representations, 
and so that $\Theta_{V_m,W_n}(\pi)$ is irreducible by the Howe duality conjecture 
(Theorem \ref{howe}).
\vskip 5pt

Suppose that $\pi$ is not a discrete series representation.
Then there exist $\tau \in \Irr_\disc(\GL(X_k))$ and 
$\pi_0 \in \Irr_\temp(G(W_{n_0}))$ with $n_0=n-2k$ such that
$\Ind_{P(X_k)}^{G(W_n)}(\chi_V \tau \otimes \pi_0) \twoheadrightarrow \pi$.
By our assumption, $\tau \not\cong \St_l$.
Also, $\tau \not \cong \St_{k}|{\det}_{X_k}|_E^{\half{l-k}}$ since $\tau$ is discrete series.
Hence we can apply Corollary \ref{pi0} to $\pi$.
We have 
\[
\Ind_{Q(Y_k)}^{H(V_m)}(\chi_W \tau \otimes \Theta_{V_{m_0},W_{n_0}}(\pi_0))
\twoheadrightarrow \Theta_{V_m,W_n}(\pi).
\]
By the induction hypothesis, we see that 
$\Theta_{V_{m_0},W_{n_0}}(\pi_0)$ is irreducible and tempered.
Hence so is $\Theta_{V_{m},W_{n}}(\pi)$.
\end{proof}

\subsection{Temperedness of theta lifts 1}
First, we prove the following proposition.
\begin{prop}\label{temp}
Let $\pi\in\Irr(G(W_{n}))$ be such that $\Theta_{V_m,W_n}(\pi)\not=0$.
Assume one of the following:
\begin{enumerate}
\item
$\pi$ is tempered and $m \leq n+1+\epsilon_0$;
\item
$\pi$ is a discrete series representation and 
$\Theta_{V_m,W_n}(\pi)$ is the first lift to the going-up tower $\VV^\up$ so that
$m=m^\up(\pi)$.
\end{enumerate}
Then all irreducible subquotients of $\Theta_{V_{m},W_{n}}(\pi)$ are tempered.
\end{prop}
\begin{proof}
The first case is similar to \cite[Proposition C.1]{GI1}.
Hence we consider the second case.
So we assume that $\pi$ is a discrete series representation and $m=m^\up(\pi)$.
\vskip 5pt

Fix an $H(V_{m})$-invariant filtration of $\Theta_{V_{m},W_{n}}(\pi)$:
\[
\Theta_{V_{m},W_{n}}(\pi)
=\Sigma_0 \supset \Sigma_1 \supset \dots \supset \Sigma_c \supset \Sigma_{c+1}=0
\]
such that 
\[
\Pi_i\coloneqq \Sigma_i/\Sigma_{i+1}
\]
is irreducible for any $i$.
Suppose that $\Pi_k$ is non-tempered.
We may assume that $\Pi_i$ is tempered for $i=0,\dots,k-1$.
Then there exists a maximal parabolic subgroup $Q$ of $H(V_{m})$
with Levi component $L_Q=\GL_t(E)\times H(V_{m_0})$ such that
\[
\Pi_k \hookrightarrow \Ind_Q^{H(V_{m})}(\tau|\det|_E^{-s_0} \otimes \sigma_0),
\]
where $\tau\in\Irr_\disc(\GL_t(E))$, $s_0>0$ and $\sigma_0\in H(V_{m_0})$.
By a similar argument to \cite[Proposition C.1]{GI1}, 
we have a nonzero $H(V_{m})$-map
\[
\Theta_{V_{m},W_{n}}(\pi) \rightarrow
\Ind_Q^{H(V_{m})}(\tau|\det|_E^{-s_0} \otimes \sigma_0).
\]
Hence we have
\begin{align*}
\pi^\vee & \hookrightarrow 
\Hom_{H(V_{m})}(\omega_{V_{m},W_{n}}, 
\Ind_Q^{H(V_{m})}(\tau|\det|_E^{-s_0} \otimes \sigma_0))\\
& \cong \Hom_{\GL_t(E)\times H(V_{m_0})}
(R_Q(\omega_{V_{m},W_{n}}), \tau|\det|_E^{-s_0} \otimes \sigma_0),
\end{align*}
where $R_Q$ denotes the normalized Jacquet functor with respect to $Q$.
The last $\Hom$ space has been studied precisely in the proof of \cite[Proposition 3.1]{GT1}.
According to (the proof of) this proposition, 
one of the following must occur:
\begin{enumerate}
\renewcommand{\labelenumi}{(\alph{enumi})}
\item
$\Theta_{V_{m-2},W_{n}}(\pi)\not=0$;
\item
$\Ind_{P(X_a)}^{G(W_{n})}(\chi_V\St_a \otimes \pi_0)\twoheadrightarrow \pi$
for some $a$ and $\pi_0\in\Irr_\temp(G(W_{n_0}))$.
\end{enumerate}
However, (a) can not occur since $\Theta_{V_{m},W_{n}}(\pi)$ is the first occurrence.
Also, (b) contradicts that $\pi$ is a discrete series representation.
This completes the proof.
\end{proof}

We also need the following proposition in \cite{GT1}:
\begin{prop}[{\cite[Proposition 3.2]{GT1}}]\label{L-quot}
Let $\pi \in \Irr(G(W_n))$.
Assume that $l=n-m+\epsilon_0 \leq 0$ and
$\theta_{V_m,W_n}(\pi)$ is nonzero and tempered.
We put $V_{m+2r}=V_m\oplus\H^r$ for $r\geq0$.
Then $\theta_{V_{m+2r},W_n}(\pi)$ is the unique irreducible quotient of
the standard module
\[
\chi_W|\cdot|_E^{\half{2r-1-l}} \times
\chi_W|\cdot|_E^{\half{2r-3-l}} \times \dots \times
\chi_W|\cdot|_E^{\half{1-l}} \rtimes \theta_{V_m,W_n}(\pi).
\]
\end{prop}
\vskip 5pt

This proposition implies Theorem \ref{main2} (4).
In fact, \cite[Proposition 3.2]{GT1} can be applied to more general situation
as we shall show in Proposition \ref{GT32} below.
Theorem \ref{main3} (3) is proven by showing that we can apply \cite[Proposition 3.2]{GT1}
to $\theta_{V_{m^\up(\pi)},W_n}(\pi)$, which may be non-tempered, for $\pi \in \Irr_\temp(G(W_n))$.
\vskip 5pt

Also, Proposition \ref{L-quot} implies the following.
\begin{cor}\label{nontemp}
Let $\pi \in \Irr(G(W_n))$.
Assume that $l=n-m+\epsilon_0 < -1$ and
$\theta_{V_m,W_n}(\pi)$ is nonzero and tempered.
Let $V_{m_0}$ be the space which belongs to the same Witt tower as $V_m$, and
$l_0=n-m_0+\epsilon_0=0$ or $-1$.
Then $\Theta_{V_{m_0},W_n}(\pi)=0$.
\end{cor}
\begin{proof}
If $\theta_{V_{m_0},W_n}(\pi)$ were nonzero, it must be tempered by Proposition \ref{temp},
and so that $\theta_{V_{m},W_n}(\pi)$ is non-tempered by Proposition \ref{L-quot}.
It contradicts the temperedness of $\theta_{V_m,W_n}(\pi)$.
\end{proof}

\section{Proof of main theorems}\label{pf}
In this section, we prove Theorems \ref{main1}, \ref{main2} and \ref{main3}.
\vskip 5pt

\subsection{Correspondence of last names}
First, we study the correspondence of last names.
\vskip 5pt

\begin{prop}\label{nu}
Let $\pi\in\Irr_\temp(G(W_{n}))$ with $L$-parameter $(\phi,\eta)$.
Assume that $\Theta_{V_m,W_n}(\pi)\not=0$ with $l=n-m+\epsilon_0>0$.
Then $\phi$ contains $\chi_V S_{l}$.
\end{prop}
\begin{proof}
Consider the standard gamma factors (see Appendix \ref{std}).
By Proposition \ref{pole} and Desideratum \ref{des} (\ref{G-hyp}), 
the gamma factor
\[
\gamma(s,\phi\otimes\chi_V^{-1},\psi_E)=\ep(s,\phi\otimes\chi_V^{-1},\psi_E)
\frac{L(1-s,\phi^\vee\otimes\chi_V^{-1})}{L(s,\phi\otimes\chi_V^{-1})}
\]
has a pole at $s=\frac{l+1}{2}$.
This implies that 
$L(1-s,\phi^\vee\otimes\chi_V^{-1})$ has a pole at $s=\frac{l+1}{2}$.
We decompose
\[
\phi=\bigoplus_{i\geq1}\phi_i \otimes S_i,
\]
where $\phi_i$ is a tempered representation of $W_E$.
Then we have
\[
L(1-s,\phi^\vee \otimes \chi_V^{-1})=\prod_{i\geq 1}L(1-s+\frac{i-1}{2},\phi_i^\vee\otimes\chi_V^{-1}).
\]
Since $\phi_i$ is tempered, 
only $L(1-s+\frac{l-1}{2},\phi_{l}\otimes\chi_V^{-1})$ can have a pole at $s=\frac{l+1}{2}$.
Moreover, if it has a pole, then $\phi_{l}\otimes\chi_V^{-1}$ must contain the trivial representation.
Hence the proposition.
\end{proof}

\begin{cor}\label{nus}
Let $\pi\in\Irr_\temp(G(W_{n}))$ with $L$-parameter $(\phi,\eta)$.
Assume that $\Theta_{V_m,W_n}(\pi)\not=0$ with $l=n-m+\epsilon_0>0$.
Define $\kappa \in \{1,2\}$ by $\kappa \equiv l \bmod 2$.
Then $\phi$ contains $\chi_V S_{r}$ for $r=\kappa, \kappa+2, \dots, l$.
Moreover, the multiplicity $m_\phi(\chi_V S_{r})$ is odd for $r=\kappa, \kappa+2, \dots, l-2$.
\end{cor}
\begin{proof}
By Proposition \ref{nu} and Proposition \ref{tower}, 
we see that $\phi$ contains $\chi_V S_{r}$ for $r=\kappa, \kappa+2, \dots, l$.
\vskip 5pt

By an induction on $n$,
we prove that $m_\phi(\chi_V S_r)$ is odd for any $r=\kappa+2i$ with $0 \leq i < (l-\kappa)/2$.
We may assume that $m_\phi(\chi_V S_r) \geq 2$.
Then we can write
\[
\phi=\chi_V S_r \oplus \phi_0 \oplus \chi_V S_r
\]
for some $\phi_0 \in \Phi_\temp(G(W_{n_0}))$ with $n_0=n-2r$.
We can find $\pi_0 \in \Irr_\temp(G(W_{n_0}))$
such that there is a surjection
$\Ind_{P(X_r)}^{G(W_n)}(\chi_V \St_{r} \otimes \pi_0) \twoheadrightarrow \pi$.
Then the $L$-parameter of $\pi_0$ is given by $(\phi_0,\eta|A_{\phi_0})$.
Since $r < l$, by Corollary \ref{pi0}, we have a surjection
$\Ind_{Q(Y_r)}^{H(V_m)}(\chi_W \St_{r} \otimes \Theta_{V_{m_0},W_{n_0}}(\pi_0))
\twoheadrightarrow \Theta_{V_m,W_n}(\pi)$ with $m_0=m-2r$.
In particular, $\Theta_{V_{m_0},W_{n_0}}(\pi_0)$ is nonzero.
Since $n_0-m_0+\epsilon_0=l$, by the induction hypothesis, 
we see that $m_{\phi_0}(\chi_V S_{r})$ is odd.
Therefore $m_\phi(\chi_V S_{r})=m_{\phi_0}(\chi_V S_{r})+2$ is also odd.
\end{proof}

Corollary \ref{nus} gives the (chain condition) and the (odd-ness condition) in Theorem \ref{main1} (1).
Note that it is possible that $m_\phi(\chi_V S_l)$ is even as we shall see later.
The parity of $m_\phi(\chi_V S_l)$ determines the temperedness of
the first occurrence $\theta_{V'_{m^\up(\pi)},W_{n}}(\pi)$
to the going-up tower (Corollary \ref{temp2}).
\vskip 5pt

Next, we determine the last name of theta lifts in a special case for the going-down tower.
\begin{thm}\label{param}
Let $\pi \in \Irr_\temp(G(W_n))$ whose last name is $\phi\in \Phi_\temp(G(W_n))$.
Assume that $\Theta_{V_m,W_n}(\pi)\not=0$ with $l=n-m+\epsilon_0>0$.
Put
\[
\theta_{V_m,W_n}(\phi)=(\phi\otimes\chi_V^{-1}\chi_W)-\chi_W S_{l}.
\]
Then $\theta_{V_m,W_n}(\phi)\in \Phi(H(V_{m}))$ and it is the last name of 
$\theta_{V_m,W_n}(\pi)$.
\end{thm}
\begin{proof}
Let $\phi_{\theta(\pi)}$ be the last name of $\theta_{V_m,W_n}(\pi)$.
Consider the Plancherel measure (see Appendix \ref{Pm}).
By Theorem \ref{Pmeas}, we have
\[
\mu_\psi(\tau_s\chi_W \otimes \theta_{V_m,W_n}(\pi))
=\mu_\psi(\tau_s\chi_V \otimes \pi)
\gamma(s-\frac{l-1}{2},\tau,\psi_E)^{-1}\gamma(-s-\frac{l-1}{2},\tau^\vee,\psi_E^{-1})^{-1}
\]
for any supercuspidal representation $\tau$ of $\GL_k(E)$.
Using Desideratum \ref{des} (\ref{P-hyp}) and Lemma \ref{gamma},
for any irreducible representation $\phi_\tau$ of $W_E$, we have
\begin{align*}
&\gamma(s,\phi_\tau\chi_W\otimes\phi_{\theta(\pi)}^\vee,\psi_E)
\gamma(-s,(\phi_\tau\chi_W)^\vee\otimes\phi_{\theta(\pi)},\psi_E^{-1})\\
&=\frac{\gamma(s,\phi_\tau\chi_V\otimes\phi^\vee,\psi_E)
\gamma(-s,(\phi_\tau\chi_V)^\vee\otimes\phi,\psi_E^{-1})}
{\gamma(s-\frac{l-1}{2},\phi_\tau,\psi_E)\gamma(-s-\frac{l-1}{2},\phi_\tau^\vee,\psi_E^{-1})}\\
&=\gamma(s,\phi_\tau\chi_W \otimes \theta_{V_m,W_n}(\phi)^\vee,\psi_E)
\gamma(-s,(\phi_\tau\chi_W)^\vee \otimes \theta_{V_m,W_n}(\phi),\psi_E^{-1}).
\end{align*}
By Proposition \ref{temp} (1), 
we see that $\phi_{\theta(\pi)}$ is tempered.
Hence by Lemma \ref{converse}, we have
\[
\phi_{\theta(\pi)}=\theta_{V_m,W_n}(\phi),
\]
as desired.
In particular, we have $\theta_{V_m,W_n}(\phi)\in \Phi(H(V_{m}))$.
\end{proof}

\subsection{Correspondence of first names}\label{corres char}
In this subsection, we compare the first name of $\theta_{V_m,W_n}(\pi)$ with the one of $\pi$.
To do this, we need the following lemma.
\begin{lem}\label{lemma L-packet}
Let $\pi\in\Irr(G(W_n))$.
Assume that $\Theta_{V_m,W_n}(\pi)\not=0$ and
all irreducible subquotients of $\Theta_{V_m,W_n}(\pi)$ are tempered.
Then all irreducible subquotients of $\Theta_{V_m,W_n}(\pi)$
belong to the same $L$-packet.
\end{lem}
\begin{proof}
This follows from \cite[Lemma A.1]{GI2}, \cite[Lemma B.2, Proposition B.3]{GI1}
and \cite[Lemma A.6]{GI2}.
\end{proof}
\vskip 5pt

In the following theorem, to avoid a confusion, 
we denote the characters associated to $V_m$ and $W_n$ by
$\chi_{V_m}$ and $\chi_{W_n}$, respectively.
\begin{thm}\label{main2.1}
Let $\pi \in \Irr_\temp(G(W_n))$ with $L$-parameter $(\phi, \eta)$.
Assume that $\Theta_{V_m,W_n}(\pi) \not=0$ with $l = n-m+\epsilon_0 > 1$.
Let $(\theta(\phi), \theta(\eta))$ be the $L$-parameter for $\theta_{V_m,W_n}(\pi) \in \Irr(H(V_m))$.
Then we have
\[
\theta(\eta)(a)/\eta(a)
=\left\{
\begin{aligned}
&\ep(\phi^a\chi_{V_m}^{-1} \otimes S_{l-1}) \cdot \ep(\phi^a) \cdot \chi_{V_m}(-1)^{\half{1}\dim(\phi^a)}
\iif \text{$E=F$, $\epsilon = +1$ and $m$ is odd},\\
&\ep(\phi^a\chi_{V_m}^{-1} \otimes S_{l-1}) \cdot \ep(\phi^a\chi_{W_n}) 
\cdot \chi_{W_n}(-1)^{\half{1}\dim(\phi^a)}
\iif \text{$E=F$, $\epsilon = -1$ and $n$ is odd},\\
&\ep(\phi^a\chi_{V_m}^{-1} \otimes S_{l-1}) \cdot \det(\phi^a\chi_{V_m}^{-1})(-1)^{\half{l-1}} \cdot \nu^{\det(a)}
\iif \text{$E=F$ and $m$, $n$ are even},\\
&\ep(\phi^a\chi_{V_m}^{-1} \otimes S_{l-1},\psi^E_{2})
\iif \text{$E\not=F$},\\
\end{aligned}
\right.
\]
where the constant $\nu \in \{\pm1\}$ is given by
\[
\nu = (-1)^{\half{l-1}} \cdot \eta(e_1+e_l).
\]
\end{thm}
\begin{proof}
If $E\not=F$, we choose a character $\chi$ of $E^\times$ such that 
$\chi | F^\times =\omega_{E/F}$.
We shall treat the cases $\epsilon = +1$ and $\epsilon = -1$ separately.
\vskip 5pt

Suppose that $\epsilon = +1$.
Put
\[
\omega=\left\{
\begin{aligned}
&\omega_{\psi}		\iif \text{$E=F$},\\ 
&\omega_{\psi,\chi}	\iif E\not=F.
\end{aligned}
\right.
\]
Let $L$ be the Hermitian space of dimension $1$ such that
\[
\disc(L)=
\left\{
\begin{aligned}
&(-1)^{m+1}	\iif E=F,\\
&(-1)^{m}	\iif E\not=F.
\end{aligned}
\right.
\]
Put $V_{m+1}=V_m \oplus L$.
If $E\not=F$, we set $\chi_L = \chi^{(-1)^{m}}$ and $\chi_{V_{m+1}}=\chi_{V_m}\chi_L$.
We denote by $(G'(W_n),H(V_{m+1}))$ the pair of groups associated to $(V_{m+1},W_n)$
defined in \S \ref{group}.
By Lemma \ref{Lemma 12.5}, we can find $\pi' \in \Irr_\temp(G'(W_{n}))$ such that
\[
\left\{
\begin{aligned}
&\Hom_{G'(W_{n})}(\pi \otimes \pi', \omega)\not=0
\iif \text{$E=F$ and $m\equiv 0 \bmod 2$, or $E\not=F$ and $m\equiv 1 \bmod 2$},\\
&\Hom_{G'(W_{n})}(\pi \otimes \pi', \overline{\omega})\not=0
\other
\end{aligned}
\right.
\]
so that
\[
\left\{
\begin{aligned}
&\Hom_{G'(W_{n})}(\pi \otimes \overline{\omega}, \pi'^\vee)\not=0
\iif \text{$E=F$ and $m\equiv 0 \bmod 2$, or $E\not=F$ and $m\equiv 1 \bmod 2$},\\
&\Hom_{G'(W_{n})}(\pi \otimes \omega, \pi'^\vee)\not=0
\other.
\end{aligned}
\right.
\]
We put $\sigma=\theta_{V_m,W_n}(\pi) \in \Irr_\temp(H(V_m))$.
Since $\pi \cong \theta_{W_n,V_m}(\sigma)$, we have
\[
\Hom_{G'(W_{n})}(\Theta_{W_n,V_m}(\sigma) \otimes \overline{\omega}, \pi'^\vee) \supset 
\Hom_{G'(W_{n})}(\pi \otimes \overline{\omega}, \pi'^\vee)\not=0
\]
or
\[
\Hom_{G'(W_{n})}(\Theta_{W_n,V_m}(\sigma) \otimes \omega, \pi'^\vee) \supset 
\Hom_{G'(W_{n})}(\pi \otimes \omega, \pi'^\vee)\not=0.
\]
The see-saw diagram
\[
\xymatrix{
G(W_{n})\times G(W_{n}) \ar@{-}[dr] \ar@{-}[d] & H(V_{m+1}) \ar@{-}[d]\\
G'(W_{n}) \ar@{-}[ur] & H(V_{m})\times H(L)
}
\]
implies that 
\[
\Hom_{H(V_m)}(\Theta_{V_{m+1},W_n}(\pi'^\vee), \sigma)\not=0.
\]
Hence $\Theta_{V_{m+1},W_{n}}(\pi'^\vee)$ has an irreducible subquotient $\sigma'$ such that
\[
\Hom_{H(V_m)}(\sigma', \sigma)\not=0
\quad\text{so that}\quad
\Hom_{H(V_m)}(\sigma^\vee \otimes \sigma',\C)\not=0.
\]
Since $\sigma^\vee$ and $\sigma'$ are tempered, they are unitary, so that
$\overline{\sigma^\vee} \cong \sigma$ and $\overline{\sigma'} \cong \sigma'^\vee$.
Hence we have
\[
\Hom_{H(V_m)}(\sigma \otimes \sigma'^\vee, \C)\not=0.
\]
By the GP conjectures 
(Theorems \ref{GGP-O} -- \ref{GGP-SH} and Corollary \ref{GGP-HS}), we have
\begin{align*}
\eta(a)&=
\left\{
\begin{aligned}
&\ep(\phi^a \otimes \phi_{\pi'} \chi_{-1}) \cdot \ep(\phi^a)
\cdot \chi_{-1}(-1)^{\half{1}\dim(\phi^a)}
\iif \text{$E=F$ and $m$ is odd},\\
&\ep(\phi^a \otimes \phi_{\pi'}) \cdot \ep(\phi \otimes \phi_{\pi'})^{\det(a)} \cdot 
\det(\phi^a)(-1)^{\half{1}\dim(\phi_{\pi'})}
\iif \text{$E=F$ and $m$ is even},\\
&\ep(\phi^a \otimes \phi_{\pi'} \otimes \chi^{-1}, \psi^E_{2})
\iif \text{$E\not=F$ and $m$ is odd},\\
&\omega_{E/F}(-1)^{\dim(\phi^a)} \cdot \ep(\phi^a \otimes \phi_{\pi'} \otimes \chi, \psi^E_{2})
\iif \text{$E\not=F$ and $m$ is even},\\
\end{aligned}
\right.\\
\theta(\eta)(a)&=
\left\{
\begin{aligned}
&\ep(\theta(\phi)^a \otimes \phi_{\sigma'^\vee}) 
\cdot \det(\phi_{\sigma'^\vee})(-1)^{\half{1}\dim(\theta(\phi)^a)}
\iif \text{$E=F$ and $m$ is odd},\\
&\ep(\theta(\phi)^a \otimes \phi_{\sigma'^\vee}) 
\cdot \det(\theta(\phi)^a)(-1)^{\half{1}\dim(\phi_{\sigma'^\vee})}
\cdot \nu(\sigma'^\vee)^{\det(a)}
\iif \text{$E=F$ and $m$ is even},\\
&\omega_{E/F}(-1)^{(m+1)\dim(\theta(\phi)^a)} 
\cdot \ep(\theta(\phi)^a \otimes \phi_{\sigma'^\vee}, \psi^E_{2})
\iif E\not=F
\end{aligned}
\right.
\end{align*}
for $a \in A_{\theta(\phi)} \subset A_{\phi}$.
Here, 
\begin{itemize}
\item
$\phi_{\pi'}$ and $\phi_{\sigma'^\vee}$ are the last names of $\pi'$ and $\sigma'^\vee$, respectively;
\item
$\nu(\sigma'^\vee) \in \{\pm1\}$ is the central value of $\sigma'^\vee$, 
i.e., $\sigma'^\vee(-\1_{V_{m+1}}) = \nu(\sigma'^\vee) \cdot \id$.
\end{itemize}
By Theorem \ref{param}, Lemma \ref{lemma L-packet}, Proposition \ref{contra} and 
Theorem \ref{Mp}, we have
\[
\theta(\phi)=(\phi \otimes \chi_{V_m}^{-1} - S_l) \otimes \chi_{W_n}
\quad\text{so that}\quad
\theta(\phi)^a = \phi^a \otimes \chi_{V_m}^{-1}\chi_{W_n}
\]
and
\[
\phi_{\sigma'^\vee}=
\left\{
\begin{aligned}
&(\phi_{\pi'} \otimes \chi_{V_{m+1}} - S_{l-1})\otimes \chi_{W_{n}}^{-1}
\iif \text{$E\not=F$ or $m$ is odd},\\
&(\phi_{\pi'} \otimes \chi_{V_{m+1}}\chi_{-1}- S_{l-1})\otimes \chi_{W_{n}}^{-1}
\iif \text{$E=F$ and $m$ is even}.
\end{aligned}
\right.
\]
Therefore we have
\[
\theta(\eta)(a)/\eta(a)
=\left\{
\begin{aligned}
&\ep(\phi^a\chi_{V_m}^{-1} \otimes S_{l-1}) 
\cdot \ep(\phi^a)	\cdot \chi_{V_m}(-1)^{\half{1}\dim(\phi^a)}
\iif \text{$E=F$ and $m$ is odd},\\
&\ep(\phi^a\chi_{V_m}^{-1} \otimes S_{l-1}) 
\cdot \det(\phi^a\chi_{V_m}^{-1})(-1)^{\half{l-1}} \cdot \nu^{\det(a)}
\iif \text{$E=F$ and $m$ is even},\\
&\ep(\phi^a\chi_{V_m}^{-1} \otimes S_{l-1},\psi^E_{2})
\iif \text{$E\not=F$},\\
\end{aligned}
\right.
\]
for some constant $\nu \in \{\pm1\}$.
\vskip 5pt

We shall determine this constant $\nu \in \{\pm1\}$.
So we assume that $E=F$ and $m$ is even, hence
$G(W_n) = \Sp(W_n)$ and $H(V_m) = \Oo(V_m)$.
Since $\sigma=\theta_{V_m,W_n}(\pi) \in \Irr_\temp(\Oo(V_m))$ satisfies that
$\pi \cong \theta_{W_n,V_m}(\sigma)$ is nonzero and tempered, by Corollary \ref{nontemp}, 
we have $\Theta_{W_m,V_m}(\sigma) = 0$.
By Prasad conjecture (Theorem \ref{PA}), we have $\theta(\eta)(z_{\theta(\phi)} + e_1) = -1$.
Since $z_{\theta(\phi)} + e_1 = z_\phi + e_1 + e_l$ in $A_\phi$, 
we have $\eta(z_{\theta(\phi)} + e_1) = \eta(e_1+e_l)$.
On the other hand, if $a = z_\phi + e_1 + e_l$, we have
\begin{align*}
\ep(\phi^a\chi_{V_m}^{-1} \otimes S_{l-1}) \cdot 
\det(\phi^a\chi_{V_m}^{-1})(-1)^{\half{l-1}} \cdot \nu^{\det(a)}
&=
\ep(\phi\chi_{V_m}^{-1} \otimes S_{l-1}) \cdot \ep((S_1 \oplus S_l) \otimes S_{l-1}) 
\cdot \chi_{V_{m}}(-1)^{\half{l-1}}\cdot \nu.
\end{align*}
We have $\ep((S_1 \oplus S_l) \otimes S_{l-1}) = -(-1)^{l-1} = -1$.
Also, since $\det(\phi\chi_{V_m}^{-1}) = \chi_{V_m}$, 
by Lemma \ref{crit} and (odd-ness condition) proved in Corollary \ref{nus}, we know 
\[
\ep(\phi\chi_{V_m}^{-1} \otimes S_{l-1}) \cdot \chi_{V_m}(-1)^{\half{l-1}} 
=(-1)^{m_\phi(\chi_{V_m}S_{l-2}) + m_{\phi}(\chi_{V_m}S_{l-4}) + \dots + m_{\phi}(\chi_{V_m}S_1)}
= (-1)^{\half{l-1}}.
\]
Hence we have $\nu = (-1)^{\half{l-1}} \cdot \eta(e_1+e_l)$, as desired.
\vskip 5pt

Now suppose that $\epsilon = -1$.
Then $n\geq m-\epsilon_0+2$. 
If $n \leq 2-\epsilon_0$, then $n=2-\epsilon_0$ and $m=0$.
In this case, the only representation of $G(W_n)$ which
participates in the theta correspondence with $H(V_0)$
is the trivial representation, so that
we have nothing to prove.
In the other cases, 
there is a line $L$ in $W_{n}$ such that 
\[
\disc(L) = 
\left\{
\begin{aligned}
&(-1)^{n} \iif E=F,\\
&(-1)^{n-1} \iif E \not=F.
\end{aligned}
\right.
\]
Let $W_{n-1}$ be the orthogonal complement of $L$ in $W_n$.
If $E\not=F$, we set $\chi_L=\chi^{(-1)^{n-1}}$ and $\chi_{W_{n-1}}=\chi_{W_n}\chi^{(-1)^n}$.
By Lemma \ref{Lemma 12.5}, we can find $\pi' \in \Irr_\temp(G(W_{n-1}))$ such that
\[
\Hom_{G(W_{n-1})}(\pi \otimes \pi', \C)\not=0
\quad\text{so that}\quad
\Hom_{G(W_{n-1})}(\pi, \pi'^\vee)\not=0.
\]
We put $\sigma=\theta_{V_m,W_n}(\pi) \in \Irr_\temp(H(V_m))$.
Since $\pi \cong \theta_{W_n,V_m}(\sigma)$, we have
\[
\Hom_{G(W_{n-1})}(\Theta_{W_n,V_m}(\sigma), \pi'^\vee) \supset 
\Hom_{G(W_{n-1})}(\pi, \pi'^\vee)\not=0.
\]
The see-saw diagram
\[
\xymatrix{
G(W_{n}) \ar@{-}[dr] \ar@{-}[d] & H(V_{m})\times H(V_{m}) \ar@{-}[d]\\
G(W_{n-1})\times G(L) \ar@{-}[ur] & H(V_{m})
}
\]
implies that
\[
\left\{
\begin{aligned}
&\Hom_{H(V_m)}(\Theta_{V_m,W_{n-1}}(\pi'^\vee) \otimes \omega, \sigma)\not=0
\iif \text{$E=F$ and $n \equiv 0 \bmod 2$, or $E\not=F$ and $n \equiv 1 \bmod 2$},\\
&\Hom_{H(V_m)}(\Theta_{V_m,W_{n-1}}(\pi'^\vee) \otimes \overline{\omega}, \sigma)\not=0
\other,
\end{aligned}
\right.
\]
where we put
\[
\omega=\left\{
\begin{aligned}
&\omega_{\psi}		\iif \text{$E=F$},\\ 
&\omega_{\psi,\chi}	\iif E\not=F.
\end{aligned}
\right.
\]
Hence $\Theta_{V_m,W_{n-1}}(\pi'^\vee)$ has an irreducible subquotient $\sigma'$ such that
\[
\left\{
\begin{aligned}
&\Hom_{H(V_m)}(\sigma' \otimes \omega, \sigma)\not=0
\iif \text{$E=F$ and $n \equiv 0 \bmod 2$, or $E\not=F$ and $n \equiv 1 \bmod 2$},\\
&\Hom_{H(V_m)}(\sigma' \otimes \overline{\omega}, \sigma)\not=0
\other,
\end{aligned}
\right.
\]
so that
\[
\left\{
\begin{aligned}
&\Hom_{H(V_m)}(\sigma \otimes \sigma'^\vee, \omega)\not=0
\iif \text{$E=F$ and $n \equiv 0 \bmod 2$, or $E\not=F$ and $n \equiv 1 \bmod 2$},\\
&\Hom_{H(V_m)}(\sigma \otimes \sigma'^\vee, \overline{\omega})\not=0
\other.
\end{aligned}
\right.
\]
Here we use the fact that $\sigma$, $\sigma'$ and $\omega$ are unitary.
By the GP conjectures 
(Theorems \ref{GGP-O} -- \ref{GGP-SH} and Corollary \ref{GGP-HS}), we have
\begin{align*}
\eta(a)&=
\left\{
\begin{aligned}
&\ep(\phi^a \otimes \phi_{\pi'}) 
\cdot \det(\phi_{\pi'})(-1)^{\half{1}\dim(\phi^a)}
\iif \text{$E=F$ and $n$ is odd},\\
&\ep(\phi^a \otimes \phi_{\pi'}) 
\cdot \det(\phi^a)(-1)^{\half{1}\dim(\phi_{\pi'})}
\cdot \nu(\pi')^{\det(a)}
\iif \text{$E=F$ and $n$ is even},\\
&\omega_{E/F}(-1)^{(n-1)\dim(\phi^a)} \cdot \ep(\phi^a \otimes \phi_{\pi'}, \psi^E_{2})
\iif E\not=F,
\end{aligned}
\right.\\
\theta(\eta)(a)&=
\left\{
\begin{aligned}
&\ep(\theta(\phi)^a \otimes \phi_{\sigma'^\vee} \chi_{-1}) \cdot \ep(\theta(\phi)^a)
\cdot \chi_{-1}(-1)^{\half{1}\dim(\theta(\phi)^a)}
\iif \text{$E=F$ and $n$ is odd},\\
&\ep(\theta(\phi)^a \otimes \phi_{\sigma'^\vee}) 
\cdot \ep(\theta(\phi) \otimes \phi_{\sigma'^\vee})^{\det(a)}
\cdot \det(\theta(\phi)^a)(-1)^{\half{1}\dim(\phi_{\sigma'^\vee})}
\iif \text{$E=F$ and $n$ is even},\\
&\ep(\theta(\phi)^a \otimes \phi_{\sigma'^\vee} \otimes \chi^{-1}, \psi^E_{2})
\iif \text{$E\not=F$ and $n$ is odd},\\
&\omega_{E/F}(-1)^{\dim(\theta(\phi)^a)} \cdot
\ep(\theta(\phi)^a \otimes \phi_{\sigma'^\vee} \otimes \chi, \psi^E_{2})
\iif \text{$E\not=F$ and $n$ is even}
\end{aligned}
\right.
\end{align*}
for $a \in A_{\theta(\phi)} \subset A_{\phi}$.
Here, 
\begin{itemize}
\item
$\phi_{\pi'}$ and $\phi_{\sigma'^\vee}$ are the last names for $\pi'$ and $\sigma'^\vee$,
respectively;
\item
$\nu(\pi') \in \{\pm1\}$ is the central value of $\pi'$, 
i.e., $\pi'(-\1_{V_{m+1}}) = \nu(\pi') \cdot \id$.
\end{itemize}
By Theorem \ref{param}, Lemma \ref{lemma L-packet}, Proposition \ref{contra}
and Theorem \ref{Mp}, we have
\[
\theta(\phi)=(\phi \otimes \chi_{V_m}^{-1} - S_l) \otimes \chi_{W_n}
\quad\text{so that}\quad
\theta(\phi)^a = \phi^a \otimes \chi_{V_m}^{-1}\chi_{W_n}
\]
and
\[
\phi_{\sigma'^\vee}=
\left\{
\begin{aligned}
&(\phi_{\pi'} \otimes \chi_{V_m} - S_{l-1})\otimes \chi_{W_{n-1}}^{-1}
\iif \text{$E\not=F$ or $n$ is odd},\\
&(\phi_{\pi'} \otimes \chi_{V_m}- S_{l-1})\otimes \chi_{W_{n-1}}^{-1}\chi_{-1}
\iif \text{$E=F$ and $n$ is even}.
\end{aligned}
\right.
\]
Therefore we have
\[
\theta(\eta)(a)/\eta(a)
=\left\{
\begin{aligned}
&\ep(\phi^a\chi_{V_m}^{-1} \otimes S_{l-1}) 
\cdot \ep(\phi^a\chi_{V_m}^{-1}\chi_{W_n}) \cdot \chi_{W_n}(-1)^{\half{1}\dim(\phi^a)}
\iif \text{$E=F$ and $n$ is odd},\\
&\ep(\phi^a\chi_{V_m}^{-1} \otimes S_{l-1}) \cdot \det(\phi^a\chi_{V_m}^{-1})(-1)^{\half{l-1}} 
\cdot \nu^{\det(a)}
\iif \text{$E=F$ and $n$ is even},\\
&\ep(\phi^a\chi_{V_m}^{-1} \otimes S_{l-1},\psi^E_{2})
\iif \text{$E\not=F$},\\
\end{aligned}
\right.
\]
for some constant $\nu \in \{\pm1\}$.
\vskip 5pt

We shall determine this constant $\nu \in \{\pm1\}$.
So we assume that $E=F$ and $n$ is even, hence
$G(W_n) = \Oo(W_n)$ and $H(V_m) = \Sp(V_m)$.
Note that $\theta(\eta)(z_{\theta(\phi)}) = 1$.
Also, by Prasad conjecture (Theorem \ref{PA}), we have $\eta(z_{\phi} + e_1) = 1$.
Since $z_{\theta(\phi)} = z_\phi + e_l$ in $A_\phi$, 
we have $\eta(z_{\theta(\phi)}) = \eta(e_1+e_l)$.
On the other hand, if $a = z_\phi + e_l$, we have
\begin{align*}
\ep(\phi^a\chi_{V_m}^{-1} \otimes S_{l-1}) \cdot \det(\phi^a\chi_{V_m}^{-1})(-1)^{\half{l-1}} 
\cdot \nu^{\det(a)}
&=
\ep(\phi\chi_{V_m}^{-1} \otimes S_{l-1}) \cdot \ep(S_l \otimes S_{l-1}) 
\cdot \chi_{W_n}(-1)^{\half{l-1}} \cdot \nu.
\end{align*}
We have $\ep(S_l \otimes S_{l-1}) = (-1)^{l-1}=1$.
Also, by Lemma \ref{crit} and (odd-ness condition) proved in Corollary \ref{nus}, we have
\[
\ep(\phi\chi_{V_m}^{-1} \otimes S_{l-1}) \cdot \chi_{W_n}(-1)^{\half{l-1}}  
(-1)^{m_\phi(\chi_{V_m}S_{l-2}) + m_{\phi}(\chi_{V_m}S_{l-4}) + \dots + m_{\phi}(\chi_{V_m}S_1)}
= (-1)^{\half{l-1}}.
\]
Hence we have 
$\nu = (-1)^{\half{l-1}} \cdot \eta(e_1+e_l)$, 
as desired.
This completes the proof.
\end{proof}

\begin{rem}
Suppose that $E=F$ and $m$, $n$ are even.
After Proposition \ref{alt}, which shows the (odd-ness condition), 
we will obtain $\eta(e_1+e_l) = (-1)^{\half{l-1}}$ so that $\nu=1$.
By using Theorems \ref{main2} (5) and \ref{main3} (4), which are proven in Proposition \ref{sign}, 
we can obtain $\nu=1$ directly.
\end{rem}

\subsection{Comparison of central elements}
Let $(V_m,W_n)$ and $l=n-m+\epsilon_0$ be as in \S \ref{spaces}.
Let $\pi \in \Irr_\temp(G(W_n))$. 
Assume that $l \geq 2$ and 
$\sigma=\theta_{V_m,W_n}(\pi)\not=0$ so that $\sigma \in \Irr(H(V_m))$.
We denote the $L$-parameters for $\pi$ and $\sigma$ by $(\phi_\pi,\eta_\pi)$ and
$(\phi_\sigma,\eta_\sigma)$, respectively.
In this subsection, we compare ``$\eta_\pi(z_{\phi_\pi})$'' with ``$\eta_\sigma(z_{\phi_\sigma})$''.
\vskip 5pt

Let $\phi \in \Phi(G(W_n))$ (\resp $\phi' \in \Phi(H(V_m))$).
If $l=n-m+\epsilon_0$ is odd and $\phi$ contains $\chi_V$ (\resp $\phi'$ contains $\chi_W$), 
then we denote by $e_1$ the element in $A_\phi$ (\resp $A_{\phi'}$)
corresponding to $\chi_V$ (\resp $\chi_W$), i.e., $\phi^{e_1}=\chi_V$ (\resp $\phi'^{e_1} = \chi_W$).

\begin{prop}\label{center1}
Let $\pi \in \Irr_\temp(G(W_n))$ such that 
$\sigma=\theta_{V_m,W_n}(\pi) \in \Irr(H(V_m))$ is nonzero.
Assume that $l=n-m+\epsilon_0 \geq 2$.
We denote the $L$-parameters for $\pi$ and $\sigma$ by $(\phi_\pi,\eta_\pi)$ and
$(\phi_\sigma,\eta_\sigma)$, respectively.
Then we have the following:
\begin{enumerate}
\item
If $l$ is odd, then $\phi_\pi \supset \chi_V$ and $\phi_\sigma \supset \chi_W$.
\item
If $E=F$, $m \not\equiv n \bmod2$ and $\epsilon=+1$, then
\[
\eta_\pi(z_{\phi_\pi})=\eta_\sigma(z_{\phi_\sigma}) \cdot 
\ep(\phi_\pi) \cdot \ep(\phi_\pi \otimes \chi_V) \cdot \chi_V(-1)^{\half{n}}.
\]
\item
If $E=F$, $m \not\equiv n \bmod 2$ and $\epsilon=-1$, then
\[
\eta_\pi(z_{\phi_\pi})=-\eta_\sigma(z_{\phi_\sigma}) \cdot \delta(\chi_W=\1) \cdot
\ep(\phi_\pi) \cdot \ep(\phi_\pi \otimes \chi_W) \cdot \chi_W(-1)^{\half{n-1}}.
\]
\item
If $E=F$, $m \equiv n \bmod2$ and $\epsilon = +1$, then 
$\eta_\pi(z_{\phi_\pi}+e_1) = \eta_\sigma(z_{\phi_\sigma})$.
\item
If $E=F$, $m \equiv n \bmod2$ and $\epsilon = -1$, then 
$\eta_\pi(z_{\phi_\pi}) = -\eta_\sigma(z_{\phi_\sigma} + e_1)$.
\item
If $E\not=F$ and $l$ is even, then
\[
\eta_\sigma(z_{\phi_\sigma}) = 
\ep(\phi_\pi \otimes \chi_V^{-1}, \psi^E_2) \cdot \eta_\pi(z_{\phi_\pi}).
\]
\item
If $E\not=F$ and $l$ is odd, then $\eta_\pi(e_1)=-\eta_\sigma(e_1)$ and
$\eta_\pi(z_{\phi_\pi}+e_1)=\eta_\sigma(z_{\phi_\sigma})$.
\end{enumerate}
\end{prop}
\begin{proof}
(1) follows from Corollary \ref{nus} and Theorem \ref{param}.
\vskip 5pt

The proofs of (2)--(5) are similar.
So we prove (3) only.
By the assumption, $G(W_n)=\Oo(W_n)$ is an odd orthogonal group and 
$H(V_m)=\Mp(V_m)$ is a metaplectic group.
By Theorem \ref{Mp-O}, there is unique $W_{m+1}^\bullet$ such that 
$\pi'=\theta_{W^\bullet_{m+1},V_m}(\sigma)$ is nonzero.
Let $(\phi_{\pi'},\eta_{\pi'})$ be the $L$-parameter for $\pi'$.
Note that $\theta_{W_n,V_m}(\sigma)=\pi$ is tempered and $m-n-\epsilon_0<-1$.
By applying Corollary \ref{nontemp} to $\sigma \in \Irr(\Mp(V_m))$ and $\Oo(W_n)$,
we have $\Theta_{W_{m+1},V_m}(\sigma)=0$, 
where $W_{m+1}$ is the space which belongs to the same Witt tower as $W_n$.
This implies that $W_{m+1}^\bullet \not= W_{m+1}$.
Hence we have
\begin{align*}
\eta_\pi(z_{\phi_\pi}) = -\eta_{\pi'}(z_{\phi_{\pi'}}).
\end{align*}
On the other hand, by Theorem \ref{Mp}, we have
\[
\eta_{\pi'}(z_{\phi_{\pi'}}) = \eta_\sigma(z_{\phi_\sigma}) \cdot 
\ep(\phi_\sigma) \cdot \ep(\phi_\sigma \otimes \chi_W) \cdot \chi_W(-1)^{\half{m}}.
\]
Since $\phi_\sigma = (\phi_\pi \otimes \chi_V^{-1} \chi_W) \oplus \chi_W S_l$ 
by Theorem \ref{param}, using Lemma \ref{delta}, we have
\begin{align*}
\eta_{\pi'}(z_{\phi_{\pi'}}) 
&=\eta_\sigma(z_{\phi_\sigma}) \cdot \delta(\chi_W=\1) \cdot
\ep(\phi_\pi) \cdot \ep(\phi_\pi \otimes \chi_W) \cdot \chi_W(-1)^{\half{n-1}}.
\end{align*}
Hence we obtain (3).
\vskip 5pt

Using Theorems \ref{Mp}, \ref{PE}, \ref{PA}, Proposition \ref{tower}, and Corollary \ref{nontemp},
the proofs of the other cases are similar to that of the above case.
\end{proof}

If $l\in\{-1,0,1\}$, then we see that a similar assertion holds by using Theorem \ref{Mp} and
Prasad's conjectures (Theorems \ref{PE} and \ref{PA}).
This implies Theorem \ref{main1} (2) unless $E=F$, $m \not\equiv n \bmod 2$ and $\epsilon=-1$.
In this case, $G(W_n) = \Oo(W_n)$ is an odd orthogonal group, and
the first occurrence indices $m^\pm(\pi)$ are determined by 
the central character of $\pi \in \Irr(\Oo(W_n))$.
Hence the remaining issue of Theorem \ref{main1} (2) is a relation between 
the central character of $\pi$ and theta lifts $\Theta_{V_m,W_n}(\pi)$.
It will be treated in \S \ref{sec.central} (Proposition \ref{sign2}).
\vskip 5pt

\subsection{Character conditions}
In this subsection, we derive the (initial condition) and the (alternating condition) 
in Theorem \ref{main1} (1).

\begin{prop}\label{alt}
Let $\pi \in \Irr_\temp(G(W_n))$ with $L$-parameter $(\phi,\eta)$.
Assume that $\Theta_{V_m,W_n}(\pi)\not=0$ and $l=n-m+\epsilon_0 \geq 2$.
Define $\kappa \in \{1,2\}$ by $\kappa \equiv l \bmod 2$.
Let $e_{i}$ be the element in $A_\phi$ corresponding to $\chi_V S_{i} \subset \phi$.
Then we have
\[
\eta(e_{\kappa+2i+2})=-\eta(e_{\kappa+2i}) 
\]
for $0 \leq i < (l-\kappa)/2$.
Moreover, if $\kappa=2$, then
\[
\eta(e_2)=
\left\{
\begin{aligned}
&\epsilon \cdot \delta(\chi_V=\1)
\iif\text{$E=F$ and $m \not\equiv n \bmod 2$},\\
&-1
\iif\text{$E\not=F$ and $m \equiv n \bmod 2$}.
\end{aligned}
\right.
\]
\end{prop}
\begin{proof}
Let $(\theta(\phi),\theta(\eta))$ be the $L$-parameter for $\theta_{V_m,W_n}(\pi)$.
Note that
\[
z_\phi = z_{\theta(\phi)}+e_{l}.
\]
By applying Proposition \ref{center1} and Theorem \ref{main2.1} to 
$a=z_{\theta(\phi)} \in A_{\theta(\phi)} \subset A_\phi$, we have
\begin{align*}
&\eta(e_l)=\frac{\eta(z_\phi)}{\theta(\eta)(z_{\theta(\phi)})} 
\cdot \frac{\theta(\eta)(z_{\theta(\phi)})}{\eta(z_{\theta(\phi)})}\\
&=\left\{
\begin{aligned}
&\delta(\chi_{V}=\1)\cdot 
\ep(\phi\chi_{V}^{-1} \otimes S_{l-1}) \cdot \ep(\phi\chi_V^{-1})
\iif\text{$E=F$, $\epsilon = +1$ and $m$ is odd},\\
&-\ep(\phi \otimes S_{l-1}) \cdot \ep(\phi)
\iif\text{$E=F$, $\epsilon = -1$ and $n$ is odd},\\
&\eta(e_1) \cdot \ep(\phi\chi_{V}^{-1} \otimes S_{l-1}) \cdot \chi_{V}(-1)^{\half{l-1}}
\iif\text{$E=F$, $\epsilon=+1$ and $m\equiv n \equiv 0\bmod 2$},\\
&-\ep(\phi \chi_{V}^{-1},\psi^E_2) \cdot \ep(\phi\chi_{V}^{-1} \otimes S_{l-1},\psi^E_{2})	
\iif\text{$E\not=F$ and $m \equiv n \bmod 2$},\\
&\eta(e_1) \cdot \ep(\phi\chi_{V}^{-1} \otimes S_{l-1},\psi^E_{2})	
\iif\text{$E\not=F$ and $m \not\equiv n \bmod 2$}.
\end{aligned}
\right.
\end{align*}
If $E=F$, $\epsilon=-1$ and $m\equiv n \equiv 0\bmod 2$, 
applying Proposition \ref{center1} and Theorem \ref{main2.1} to 
$a=z_{\theta(\phi)}+e_1 = z_\phi + e_1 + e_l \in A_{\theta(\phi)} \subset A_\phi$.
we obtain
\begin{align*}
&\eta(e_l)=\frac{\eta(z_\phi+e_1)}{\theta(\eta)(z_{\theta(\phi)}+e_1)} 
\cdot \frac{\theta(\eta)(z_{\theta(\phi)}+e_1)}{\eta(z_{\theta(\phi)}+e_1)}\\
&=\eta(e_1) \cdot \epsilon(\phi \chi_V^{-1} \otimes S_{l-1}) \cdot \chi_W(-1)^{\half{l-1}}.
\end{align*}
\vskip 5pt

By the tower property (Proposition \ref{tower}), 
a similar equation for $\eta(e_{l-2i})$ holds for $i=0,1,\dots,(l-\kappa)/2$.
In particular, if $\kappa=2$, then we have
\[
\eta(e_2)=
\left\{
\begin{aligned}
&\epsilon \cdot \delta(\chi_V=\1)
\iif\text{$E=F$ and $m \not\equiv n \bmod 2$},\\
&-1
\iif\text{$E\not=F$ and $m \equiv n \bmod 2$}.
\end{aligned}
\right.
\]
Moreover, we have
\begin{align*}
&\frac{\eta(e_{\kappa+2i+2})}{\eta(e_{\kappa+2i})}
=\frac{\ep(\phi\chi_{V}^{-1} \otimes S_{\kappa+2i+1},\psi^E_{2})}
{\ep(\phi\chi_{V}^{-1} \otimes S_{\kappa+2i-1},\psi^E_{2})}
\times
\left\{
\begin{aligned}
&\det(\phi\chi_{V}^{-1})(-1)	\iif\text{$E=F$},\\
&1	\iif\text{$E\not=F$}
\end{aligned}
\right.
\end{align*}
for $0 \leq i <(l-\kappa)/2$.
By Lemma \ref{crit} and (odd-ness condition) in Theorem \ref{main1} (1), 
this is equal to $(-1)^{m_\phi(\chi_{V}S_{\kappa+2i})}=-1$.
\end{proof}
This is the (initial condition) and the (alternating condition) in Theorem \ref{main1} (1).
In particular, we have
\[
n-m^\down(\pi)+\epsilon_0 \in \TT
\]
for any $\pi \in \Irr_\temp(G(W_n))$.
\vskip 5pt

Also, we have
\[
\eta(e_1+e_l) = (-1)^{\half{l-1}}.
\]
Theorem \ref{main2.1} together with this equation implies that 
Theorem \ref{main2} (1).

\begin{rem}
We may apply the result shown above (i.e., Theorem \ref{main2} (1))
to the going-up tower sometimes.
Under the notation and assumption of Theorem \ref{main3} (1), 
we will show that $\theta_{V_m,W_n}(\pi)$ is tempered (Corollary \ref{temp2}).
If we knew the temperedness of $\theta_{V_m,W_n}(\pi)$, 
Theorem \ref{main2} (1) implies Theorem \ref{main3} (1).
\end{rem}
\vskip 5pt

The following proposition says that $l(\pi)=\max\,\TT=n-m^\down(\pi)+\epsilon_0$
in a special case.
\begin{prop}\label{non-alt}
Let $\pi \in \Irr_\temp(G(W_n))$ with $L$-parameter $(\phi,\eta)$.
Assume that
\begin{itemize}
\item
$l=n-m^\down(\pi)+\epsilon_0 \geq 0$;
\item
$\phi$ contains $\chi_{V}S_{l-2i}$ for $i=-1,0,\dots, (l-\kappa)/2$;
\item
the first occurrence 
$\sigma^\up=\theta_{V'_{m^\up(\pi)},W_n}(\pi)$ to the going-up tower $\VV^\up$ is tempered.
\end{itemize}
Then we have:
\begin{enumerate}
\item
If $l=0$, then the (initial condition) in Theorem \ref{main1} (1) does not hold.
Namely, 
\[
\eta(e_2)=
\left\{
\begin{aligned}
&-\epsilon \cdot \delta(\chi_V=\1)
\iif\text{$E=F$ and $m \not\equiv n \bmod 2$},\\
&+1
\iif\text{$E\not=F$ and $m \equiv n \bmod 2$}.
\end{aligned}
\right.
\]
\item
If $l>0$, then $m_\phi(\chi_{V}S_l)$ is odd, and the (alternating condition) in Theorem \ref{main1} (1)
does not hold.
Namely,
\[
\eta(e_{l+2}+e_{l})=-(-1)^{m_\phi(\chi_{V}S_l)}=+1.
\]
\end{enumerate}
\end{prop}
\begin{proof}
First, we prove (2).
Let $(\phi_\sigma,\eta_\sigma)$ be the $L$-parameter for $\sigma^\up$.
Note that $\sigma^\up$ is tempered by the assumption, and 
$m^\up(\pi) - n - \epsilon_0 = l+2 \geq 2$ by the conservation relation (Proposition \ref{cons}). 
By applying Theorem \ref{param}, Corollary \ref{nus} and Proposition \ref{alt} to $\sigma^\up$,
we have
\[
\phi_\sigma=(\phi \otimes \chi_{V}^{-1}\chi_{W}) \oplus S_{l+2}\chi_{W},
\]
and we see that
$m_{\phi_\sigma}(\chi_{W}S_l)=m_\phi(\chi_{V}S_l)$ is odd, and
$\eta_\sigma(e_{l+2}+e_{l})=-1$.
Therefore it is enough to show $\eta(e_{l+2}+e_{l})/\eta_\sigma(e_{l+2}+e_{l})=-1$.
It follows from Theorem \ref{main2} (1).
Hence we have (2).
The proof of (1) is similar.
\end{proof}

By Proposition \ref{temp}, if $\pi$ is a discrete series representation, 
then the first occurrence $\sigma^\up=\theta_{V'_{m^\up(\pi)},W_n}(\pi)$ is tempered.
Hence by Proposition \ref{non-alt}, 
we see that
\[
n-m^\down(\pi)+\epsilon_0+2=l+2 \not\in \TT
\]
if $\pi$ is a discrete series representation.
This completes the proof of Theorem \ref{main1} (1) for discrete series representations.

\subsection{Temperedness of theta lifts 2}
In this subsection, we discuss whether 
the first occurrence 
$\sigma=\theta_{V'_{m^\up(\pi)},W_n}(\pi)$ to the going-up tower $\VV^\up$ is tempered or not.
\vskip 5pt

Let $\pi \in \Irr_\temp(G(W_n))$ with $L$-parameter $(\phi,\eta)$.
Assume that $l=n-m^\down(\pi)+\epsilon_0 \geq 0$.
Define $\kappa \in \{1,2\}$ by $\kappa \equiv l \bmod 2$.
Then by Corollary \ref{nus}, we know that
$\phi$ contains $\chi_V S_{\kappa+2i}$ for $0 \leq i \leq (l-\kappa)/2$, 
and $m_\phi(\chi_V S_{\kappa+2i})$ is odd for $0 \leq i < (l-\kappa)/2$.
Note that $m^\up(\pi) - n - \epsilon_0 = l+2 \geq 2$. 
\vskip 5pt

Decompose $\phi=\phi' \oplus \phi_0 \oplus {}^c\phi'^\vee$ with $\phi_0 \in \Phi_\disc(G(W_{n_0}))$.
Assume that 
\[
\chi_V\tau_1 \times \dots \times \chi_V \tau_r \rtimes \pi_0 \twoheadrightarrow \pi
\]
for some $\tau_i \in \Irr_\disc(\GL_{k_i}(E))$ and $\pi_0 \in \Irr_\disc(G(W_{n_0}))$ with
$n_0=n-2\sum_{i=1}^r k_i$, so that
the $L$-parameter of $\pi_0$ is given by $(\phi_0,\eta|A_{\phi_0})$.
If $m \geq n+\epsilon_0$, 
then by a similar argument to \cite[Proposition C.4]{GI1}, 
we have
\[
\chi_W\tau_1 \times \dots \times \chi_W \tau_r \rtimes \Theta_{V_{m_0}, W_{n_0}}(\pi_0)
\twoheadrightarrow \Theta_{V_m,W_n}(\pi),
\]
where $m_0=m-2\sum_{i=1}^r k_i$.
In particular, if $\Theta_{V_m,W_n}(\pi)$ is nonzero, then
$\Theta_{V_{m_0}, W_{n_0}}(\pi_0)$ is also nonzero.
\vskip 5pt

\begin{lem}\label{d=d}
Suppose that $m^\down(\pi)<m^\up(\pi)$.
Then the going-down tower $\VV^\down$ with respect to $\pi$
is also the going-down tower $\VV^\down$ with respect to $\pi_0$.
\end{lem}
\begin{proof}
Set $m=n+\epsilon_0+2-\kappa$.
Then $l=n-m+\epsilon_0=\kappa -2 \in \{0,-1\}$.
A tower $\VV$ is the going-down tower with respect to $\pi$
if and only if $\Theta_{V_m,W_n}(\pi)$ is nonzero for $V_m \in \VV$.
In this case, $\Theta_{V_{m_0}, W_{n_0}}(\pi_0)$ is also nonzero for $V_{m_0} \in \VV$.
This shows that $\VV$ is also the going-down tower with respect to $\pi_0$.
\end{proof}
\vskip 5pt

We determine the first occurrence index of $\pi_0$ in terms of the one of $\pi$.
\begin{prop}\label{pi0down}
Let notation be as above.
If $m^\down(\pi)=n+\epsilon_0-l$ with $l>0$, then
\[
m^\down(\pi_0)=\left\{
\begin{aligned}
&n_0+\epsilon_0-l	\iif \text{$m_\phi(\chi_V S_l)$ is odd},\\
&n_0+\epsilon_0-l+2	\iif \text{$m_\phi(\chi_V S_l)$ is even}.
\end{aligned}
\right.
\]
\end{prop}
\begin{proof}
Note that we have proven Theorem \ref{main1} (1) for the discrete series representation $\pi_0$.
By Corollary \ref{nus}, 
we see that $m_\phi(\chi_V S_{\kappa+2i})$ is odd 
for $0 \leq i < (l-\kappa)/2$, 
where we define $\kappa \in \{1,2\}$ by $\kappa \equiv l \bmod 2$.
If $m_\phi(\chi_V S_l)$ is even, then 
by applying Theorem \ref{main1} (1) to $\pi_0$, we have $m^\down(\pi_0)=n_0+\epsilon_0-l+2$.
\vskip 5pt

Suppose that $m_\phi(\chi_V S_l)$ is odd.
Note that $m^\up(\pi)=n+\epsilon_0+l+2$.
By Lemma \ref{d=d} and a remark before this lemma, 
we have $m^\up(\pi_0) \leq n_0+\epsilon_0+l+2$. 
Hence $m^\down(\pi_0) \geq n_0+\epsilon_0-l$.
On the other hand, by applying Theorem \ref{main1} (1) to $\pi_0$, 
we have $m^\down(\pi_0) \leq n_0+\epsilon_0-l$.
Therefore we have $m^\down(\pi_0) = n_0+\epsilon_0-l$.
\end{proof}

\begin{cor}\label{temp2}
Let $\pi \in \Irr_\temp(G(W_n))$ with $L$-parameter $(\phi,\eta)$.
Assume that $m^\down(\pi)=n+\epsilon_0-l$ with $l\geq0$, 
so that $m^\up(\pi)=n+\epsilon_0+l+2$.
Let $\sigma=\theta_{V'_{m^\up(\pi)},W_n}(\pi)$ be the first occurrence 
to the going-up tower $\VV^\up$.
\begin{enumerate}
\item
If $l=0$, then $\sigma$ is tempered.
\item
Suppose that $l>0$. 
Then $\sigma$ is tempered 
if and only if $m_\phi(\chi_V S_l)$ is odd.
\end{enumerate}
\end{cor}
\begin{proof}
We prove (2).
The proof of (1) is similar.
So we assume that $l>0$.
\vskip 5pt

If $\sigma$ is tempered, then we have proven that 
$m_\phi(\chi_V S_l)$ is odd in the proof of Proposition \ref{non-alt}.
\vskip 5pt

Conversely, suppose that $m_\phi(\chi_V S_l)$ is odd.
We may assume that
\[
\chi_V\tau_1 \times \dots \times \chi_V \tau_r \rtimes \pi_0 \twoheadrightarrow \pi
\]
for some $\tau_i \in \Irr_\disc(\GL_{k_i}(E))$ and $\pi_0 \in \Irr_\disc(G(W_{n_0}))$ with
$n_0=n-2\sum_{i=1}^r k_i$.
As we have seen before Lemma \ref{d=d}, 
we have
\[
\chi_W\tau_1 \times \dots \times \chi_W \tau_r \rtimes \Theta_{V_{m_0}, W_{n_0}}(\pi_0)
\twoheadrightarrow \Theta_{V_m,W_n}(\pi),
\]
where $m_0=m-2\sum_{i=1}^r k_i$.
Hence there exists an irreducible subquotient $\sigma_0$ of $\Theta_{V_{m_0}, W_{n_0}}(\pi_0)$
such that
\[
\chi_W\tau_1 \times \dots \times \chi_W \tau_r \rtimes \sigma_0
\twoheadrightarrow \sigma.
\]
Since $m_\phi(\chi_V S_l)$ is odd, 
by Proposition \ref{pi0down} together with the conservation relation (Proposition \ref{cons}), 
we see that $\Theta_{V_{m_0}, W_{n_0}}(\pi_0)$
is the first lift of a discrete series representation $\pi_0$ to going-up tower $\VV^\up$.
By Proposition \ref{temp} (2), an irreducible subquotient $\sigma_0$ of 
$\Theta_{V_{m_0}, W_{n_0}}(\pi_0)$ is tempered.
Therefore, $\sigma$ is also tempered.
\end{proof}

Corollary \ref{temp2} and Proposition \ref{pi0down}
imply that
\[
n-m^\down(\pi)+\epsilon_0+2=l+2 \not\in \TT
\]
for all tempered representations.
Hence we have $l(\pi)=\max\,\TT=n-m^\down(\pi)+\epsilon_0$.
This completes the proof of Theorem \ref{main1} (1).
Also, using Corollary \ref{temp2}, 
we obtain Theorem \ref{main3} (1) from Theorem \ref{main2} (1).

\subsection{Non-tempered first lifts}\label{first}
In this subsection, we prove Theorem \ref{main3} (2).
\vskip 5pt

Let $\pi \in \Irr_\temp(G(W_n))$ with $L$-parameter $(\phi,\eta)$.
Assume that $l = l(\pi) = n - m^\down(\pi) + \epsilon_0 > 0$.
Theorem \ref{main1} (1)
implies that 
\begin{itemize}
\item
$\phi$ contains $\chi_V S_l, \chi_V S_{l-2}, \dots, \chi_V S_\kappa$, 
where $\kappa \in \{1,2\}$ is defined by $\kappa \equiv l \bmod 2$;
\item
$m_\phi(\chi_V S_{\kappa+2i})$ is odd for $0 \leq i < (l - \kappa)/2$.
\end{itemize}
We put $m=m^\up(\pi)$.
Note that $m-n-\epsilon_0=l+2$.
Let $\sigma = \theta_{V_{m},W_{n}}(\pi)$
be the first occurrence of $\pi$ to the going-up tower $\VV^\up$.
By Corollary \ref{temp2}, we see that
$\sigma$ is non-tempered if and only if $m_\phi(\chi_V S_l)$ is even.
In this subsection, we assume these conditions.
\vskip 5pt

Suppose that $\sigma$ is the Langlands quotient of the standard module
\[
\tau_1 |\cdot|_E^{s_1} \times \dots \times \tau_r |\cdot|_E^{s_r} \rtimes \sigma_0,
\]
where 
$\tau_i \in \Irr_\disc(\GL_{k_i}(E))$, 
$\sigma_0 \in \Irr_\temp(H(V_{m_0}))$,
$2k_1+\dots+2k_r+m_0=m$, and
$s_1 \geq \dots \geq s_r >0$.
\vskip 5pt

First, we have the following: 
\begin{prop}
For any $i=1,\dots,r$, the exponent $s_i$ is in $(1/2)\Z$.
\end{prop}
\begin{proof}
Consider the Plancherel measure (see Appendix \ref{Pm}).
By Theorem \ref{Pmeas}, we have
\[
\mu(\chi_W \tau |\cdot|_E^s \otimes \sigma)=
\mu(\chi_V \tau |\cdot|_E^s \otimes \pi) \cdot \gamma(s-\half{l-1},\tau,\psi_E)^{-1}
\cdot \gamma(-s-\half{l-1},\tau^\vee,\psi_E^{-1})^{-1}
\]
for any $\tau \in \GL_k(E)$.
In particular, by Desideratum \ref{des} (\ref{P-hyp}), we have
\begin{align*}
&\gamma(s,\chi_W\phi_\tau \otimes \phi_\sigma^\vee, \psi_E) \cdot 
\gamma(-s,\chi_W^{-1}\phi_\tau^\vee \otimes \phi_\sigma, \psi_E^{-1})
\\&=
\gamma(s,\chi_V\phi_\tau \otimes \phi_\pi^\vee, \psi_E) \cdot 
\gamma(-s,\chi_V^{-1}\phi_\tau^\vee \otimes \phi_\pi, \psi_E^{-1})
\cdot \gamma(s-\half{l-1},\phi_\tau,\psi_E)^{-1}
\cdot \gamma(-s-\half{l-1},\phi_\tau^\vee,\psi_E^{-1})^{-1}.
\end{align*}
Let $\AA$ be the set of $s_0 \in \C$ such that 
the left hand side of the above equation has a pole at $s=s_0$
for some unitary supercuspidal representation $\tau$ of $\GL(k,E)$.
Looking at the right hand side, we see that
\[
\{\mathrm{Re}(s_0)\ |\ s_0 \in \AA\} \subset \half{1}\Z.
\]
\vskip5pt 

Let $\phi_{\tau_i}$ be the irreducible representation of $\WD_E$ corresponding to $\tau_i$.
We may decompose $\phi_{\tau_i} \cong \phi_i \boxtimes S_{d_i}$, 
where $\phi_i$ is an irreducible representation of $W_E$ and $d_i$ is a positive integer.
Since 
\[
\phi_\sigma = \phi_{\tau_1}|\cdot|_E^{s_1} \oplus \dots \oplus \phi_{\tau_r}|\cdot|_E^{s_r}
\oplus \phi_{\sigma_0} \oplus 
{}^c\phi_{\tau_r}^\vee|\cdot|_E^{-s_r} \oplus \dots \oplus {}^c\phi_{\tau_1}^\vee|\cdot|_E^{-s_1}, 
\]
we have
\begin{align*}
&\gamma(s,\chi_W\phi_\tau \otimes \phi_\sigma^\vee, \psi_E) \cdot 
\gamma(-s,\chi_W^{-1}\phi_\tau^\vee \otimes \phi_\sigma, \psi_E^{-1})
\\&=\Big[\prod_{i=1}^r
\gamma(s-s_i, \chi_W\phi_\tau \otimes \phi_{\tau_i}^\vee, \psi_E)
\gamma(s+s_i, \chi_W\phi_\tau \otimes {}^c\phi_{\tau_i}, \psi_E)
\\&\quad\times
\gamma(-s-s_i, \chi_W^{-1}\phi_\tau^\vee \otimes \phi_{\tau_i}^\vee, \psi_E^{-1})
\gamma(-s+s_i, \chi_W^{-1}\phi_\tau^\vee \otimes {}^c\phi_{\tau_i}, \psi_E^{-1})
\Big]\\
&\quad\times
\gamma(s,\chi_W\phi_\tau \otimes \phi_{\sigma_0}^\vee, \psi_E) \cdot 
\gamma(-s,\chi_W^{-1}\phi_\tau^\vee \otimes \phi_{\sigma_0}, \psi_E^{-1}).
\end{align*}
\vskip 5pt

Now suppose that some $s_j$ is not in $(1/2)\Z$.
We may assume that $s_i \not \in (1/2)\Z$ and $s_i$ satisfies that
\[
\max\left\{s_j+\half{d_j-1}\ |\ s_j \not \in \half{1}\Z \right\}
= s_i + \half{d_i-1}.
\]
Taking $\phi_\tau = \chi_W^{-1} \phi_i$, by above equation, 
we see that 
$\gamma(s,\chi_W\phi_\tau \otimes \phi_\sigma^\vee, \psi_E) \cdot 
\gamma(-s,\chi_W^{-1}\phi_\tau^\vee \otimes \phi_\sigma, \psi_E^{-1})$
has a pole at $s= 1+ s_i + (d_i-1)/2$
since $\gamma(s-s_i, \chi_W\phi_\tau \otimes \phi_{\tau_i}^\vee, \psi_E)$ has a pole at this point.
Hence $1+ s_i + (d_i-1)/2 \in \AA$ but $1+ s_i + (d_i-1)/2 \not\in(1/2)\Z$.
This is a contradiction.
\end{proof}

\begin{cor}\label{unif}
We have $s_i=1/2$ and $\tau_i=\chi_W \St_{l+1}$ for any $i=1,\dots,r$.
\end{cor}
\begin{proof}
By \cite[Proposition 3.1]{GT1}, we know that $s_1=1/2$ and $\tau_1=\chi_W \St_{l+1}$.
Hence we have $s_i=1/2$ for any $i=1,\dots,r$.
Since each $\tau_i$ is a discrete series representation of a general linear group, 
we can interchange $\tau_i$ with $\tau_1$ (see e.g., \cite{Z}).
Hence we have $\tau_i=\chi_W \St_{l+1}$ for any $i=1,\dots,r$.
\end{proof}

The following is the key result.
\begin{prop}\label{r=1}
We have $r=1$.
\end{prop}
\begin{proof}
By (the proof of) Proposition 3.1 in \cite{GT1}, 
we can find an irreducible representation $\sigma_1$ of $H(V_{m_1})$ such that
\[
\Ind_{Q(Y_{l+1})}^{H(V_m)}(\chi_W \St_{l+1}|\cdot|_E^{1/2} \otimes \sigma_1)
\twoheadrightarrow \sigma,
\]
and
\[
\Ind_{P(X_l)}^{G(W_n)}(\chi_V \St_{l} \otimes \Theta_{W_{n_1},V_{m_1}}(\sigma_1))
\twoheadrightarrow  \pi,
\]
where we put $m_1=m-2(l+1)$ and $n_1=n-2l$.
We have to show that $\sigma_1$ is tempered.
Suppose for the sake of contradiction that $\sigma_1$ is not tempered.
Then by Corollary \ref{unif}, there exists $\sigma_2 \in \Irr(H(V_{m_2}))$ such that
\[ 
\Ind_{Q(Y'_{l+1})}^{H(V_{m_1})}(\chi_W \St_{l+1}|\cdot|_E^{1/2} \otimes \sigma_2)
\twoheadrightarrow \sigma_1,
\]
where $m_2 = m_1 - 2(l+1)$ and $V_{m_1}= Y'_{l+1} \oplus V_{m_2} \oplus (Y'_{l+1})^*$.
Since $m_1-n_1-\epsilon_0=l$, by Corollary \ref{pi0}, we have
\[
\Ind_{P(X'_{l+1})}^{G(W_{n_{1}})}
(\chi_V \St_{l+1}|\cdot|_E^{1/2} \otimes \Theta_{W_{n_2},V_{m_2}}(\sigma_2)) 
\twoheadrightarrow \Theta_{W_{n_{1}},V_{m_{1}}}(\sigma_{1}),
\]
where $n_2=n_1-2(l+1)$ and $W_{n_1}= X'_{l+1} \oplus W_{n_2} \oplus (X'_{l+1})^*$.
Combining these maps, we have
\[
\chi_V \St_l \times \chi_V \St_{l+1}|\cdot|_E^{1/2} 
\rtimes \Theta_{W_{n_2},V_{m_2}}(\sigma_2)
\twoheadrightarrow \pi.
\]
This contradicts that $\pi$ is tempered by Casselman's criterion.
\end{proof}

Now we are ready to prove Theorem \ref{main3} (2).
More precisely, we prove the following theorem:
\begin{thm}
Let $\pi \in \Irr_\temp(G(W_n))$ with $L$-parameter $(\phi,\eta)$.
Assume that 
\begin{itemize}
\item
$l = l(\pi) = n - m^\down(\pi) + \epsilon_0 > 0$;
\item
$m_\phi(\chi_V S_l)=2h$ for some $h>0$.
\end{itemize}
We write $\phi = \phi_0 \oplus (\chi_V S_l)^{\oplus 2h}$.
Put $n_0=n-2hl$ and $m_0=m-2hl-2$.
Let $\pi_0 \in \Irr_\temp(G(W_{n_0}))$ 
such that
\[
\chi_V\St_l \times \dots \chi_V \St_l \rtimes \pi_0 \twoheadrightarrow \pi,
\]
so that
the $L$-parameter of $\pi_0$ is $(\phi_0,\eta|A_{\phi_0})$.
Here, $\chi_V \St_l$ appears $h$-times.
We set $m=m^\up(\pi)$ and 
let $\sigma=\theta_{V_m,W_n}(\pi)$ be the first occurrence of $\pi$ 
to the going-up tower $\VV^\up$.
Then we have
\[
\chi_W \St_{l+1}|\cdot|_E^{1/2} \times \chi_W\St_l \times \dots \times \chi_W\St_l
\rtimes \sigma_0
\twoheadrightarrow \sigma,
\]
where 
$\sigma_0=\theta_{V_{m_0},W_{n_0}}(\pi_0)$,
and $\chi_W\St_l$ appears $(h - 1)$-times.
In particular, if we denote the $L$-parameter for $\sigma$ (\resp $\sigma_0$) 
by $(\phi_\sigma, \eta_\sigma)$ (\resp $(\phi_{\sigma_0}, \eta_{\sigma_0})$), 
then we have
\[
\phi_{\sigma_0}=\phi\chi_V^{-1}\chi_W - (\chi_W S_l)^{\oplus (2h-1)}
\quad \text{and} \quad
\phi_\sigma
=\phi_{\sigma_0} +(\chi_W S_l)^{\oplus 2(h-1)} + \chi_W S_{l+1} \otimes (|\cdot|_E^{1/2}+|\cdot|_E^{-1/2}).
\]
Moreover the canonical injection $A_{\phi_{\sigma_0}} \hookrightarrow A_{\phi_\sigma}$
is in fact bijective, and we have $\eta_\sigma|A_{\phi_{\sigma_0}}=\eta_{\sigma_0}$.
\end{thm}
\begin{proof}
By \cite[Proposition 3.1]{GT1}, 
we can find $\sigma_1 \in \Irr(H(V_{m_1}))$ and $\pi_1 \in \Irr(G(W_{n_1}))$
with $m_1=m-2(l+1)$ and $n_1=n-2l$ such that
\begin{align*}
\Ind_{Q(Y_{l+1})}^{H(V_m)}(\chi_W \St_{l+1} |\cdot|_E^{1/2} \otimes \sigma_1)
\twoheadrightarrow \sigma, \quad
\Ind_{P(X_l)}^{G(W_n)}(\chi_V \St_l \otimes \pi_1)
\twoheadrightarrow \pi
\end{align*}
and $\pi_1$ is a subquotient of $\Theta_{W_{n_1},V_{m_1}}(\sigma_1)$.
Proposition \ref{r=1} says that $\sigma_1$ is tempered.
Hence $\theta_{W_{n_1},V_{m_1}}(\sigma_1)$ belongs to the same $L$-packet as $\pi_1$
by Proposition \ref{temp} and Lemma \ref{lemma L-packet}.
Therefore we have
\begin{align*}
\phi_\sigma&=\phi_{\sigma_1} + \chi_W S_{l+1} \otimes (|\cdot|_E^{1/2}+|\cdot|_E^{-1/2})\\
&=(\phi_{\pi_1}\chi_V^{-1}\chi_W + \chi_W S_l) 
+\chi_W S_{l+1} \otimes (|\cdot|_E^{1/2}+|\cdot|_E^{-1/2})\\
&=\phi\chi_V^{-1}\chi_W - \chi_W S_l +\chi_W S_{l+1} \otimes (|\cdot|_E^{1/2}+|\cdot|_E^{-1/2}),
\end{align*}
where we denote by $\phi_{\sigma_1}$ and $\phi_{\pi_1}$ 
the last names of the $L$-parameters for $\sigma_1$ and $\pi_1$, respectively.
\vskip 5pt

In particular, there exists $\sigma_0 \in \Irr_\temp(H(V_{m_0}))$ 
whose $L$-parameter is $(\phi_{\sigma_0},\eta_{\sigma_0})$
with
\[
\phi_{\sigma_0}=\phi\chi_V^{-1}\chi_W - (\chi_W S_l)^{\oplus (2h-1)}, 
\quad
\eta_{\sigma_0}=\eta_\sigma|A_{\phi_{\sigma_0}}
\] 
such that
\[
\chi_W \St_{l+1}|\cdot|_E^{1/2} \times \chi_W\St_l \times \dots \times \chi_W\St_l
\rtimes \sigma_0
\twoheadrightarrow \sigma.
\]
Note that 
$\chi_W\St_l \times \dots \times \chi_W\St_l \rtimes \sigma_0$ 
is irreducible since $\phi_{\sigma_0}$ contains $\chi_W S_l$, so that 
\[
\chi_W \St_{l+1}|\cdot|_E^{1/2} \times \chi_W\St_l \times \dots \times \chi_W\St_l
\rtimes \sigma_0
\]
is a standard module, which has a unique Langlands quotient.
We have to show that $\sigma_0=\theta_{V_{m_0}, W_{n_0}}(\pi_0)$.
Since $\chi_W \St_{l+1}|\cdot|_E^{1/2}$ and $\chi_W\St_l$ are not linked, we have
\[
\chi_W \St_{l+1}|\cdot|_E^{1/2} \times \chi_W\St_l \times \dots \times \chi_W\St_l
\rtimes \sigma_0
\cong
\chi_W\St_l \times \dots \times \chi_W\St_l \times 
\chi_W \St_{l+1}|\cdot|_E^{1/2} \rtimes \sigma_0.
\]
For the linked-ness and its properties, see \cite{Z}
(in particular, see \cite[Theorem 9.7]{Z}).
By Lemma \ref{sub-quot}, we have
\[
\sigma \hookrightarrow 
\chi_W\St_l \times \dots \times \chi_W\St_l \times 
\chi_W \St_{l+1}|\cdot|_E^{-1/2} \rtimes \sigma_0.
\]
Since $m-n-\epsilon_0=l+2$, by applying Corollary \ref{pi0} to 
$\chi_W\St_l \times \dots \times \chi_W\St_l \times 
\chi_W \St_{l+1}|\cdot|_E^{-1/2} \rtimes \sigma_0$, we have 
\begin{align*}
&\pi^\vee \hookrightarrow \Theta_{W_{n},V_m}(\sigma)^\vee
\cong
\Hom_{H(V_m)}(\omega_{V_m,W_n}, \sigma)_\infty\\
&\hookrightarrow 
\Hom_{H(V_m)}(\omega_{V_m,W_n}, 
\chi_W\St_l \times \dots \times \chi_W\St_l \times 
\chi_W \St_{l+1}|\cdot|_E^{-1/2} \rtimes \sigma_0)_\infty\\
&\cong
\chi_V\St_l \times \dots \times \chi_V\St_l 
\rtimes 
\Hom_{H(V_{m-2(h-1)l})}(\omega_{V_{m-2(h-1)l},W_{n-2(h-1)l}}, 
\Ind_{Q(Y_{l+1})}^{H(V_{m-2(h-1)l})}(
\chi_W \St_{l+1}|\cdot|_E^{-1/2} \otimes \sigma_0))_\infty.
\end{align*}
\vskip 5pt

To $\Hom_{H(V_{m-2(h-1)l})}(\omega_{V_{m-2(h-1)l},W_{n-2(h-1)l}}, 
\Ind_{Q(Y_{l+1})}^{H(V_{m-2(h-1)l})}(
\chi_W \St_{l+1}|\cdot|_E^{-1/2} \otimes \sigma_0))_\infty$, 
we cannot apply Corollary \ref{pi0}.
According to Proposition \ref{Ja}, 
$J^{l}$ and $J^{l+1}$ can contribute.
However, since 
\[
\Hom_{\GL(Y_{l+1}) \times H(V_{m_0})}(J^{l+1}, 
\chi_W \St_{l+1}|\cdot|_E^{-1/2} \otimes \sigma_0)_\infty
\cong
\big(
\Ind_{P(X_{l+1})}^{G(W_{n_0+2l})}(\chi_V \St_{l+1}|\cdot|_E^{1/2} \otimes 
\Theta_{W_{n_0-2},V_{m_0}}(\sigma_0))
\big)^\vee, 
\]
we have
\[
\Hom_{G(W_n)}(\pi^\vee, 
\chi_V\St_l \times \dots \times \chi_V\St_l 
\rtimes 
\Hom_{\GL(Y_{l+1}) \times H(V_{m_0})}(J^{l+1}, 
\chi_W \St_{l+1}|\cdot|_E^{-1/2} \otimes \sigma_0)_\infty
)=0
\]
by Casselman's temperedness criterion.
Hence we have
\begin{align*}
\pi^\vee &\hookrightarrow 
\chi_V\St_l \times \dots \times \chi_V\St_l 
\rtimes 
\Hom_{\GL(Y_{l+1}) \times H(V_{m_0})}(J^{l}, 
\chi_W \St_{l+1}|\cdot|_E^{-1/2} \otimes \sigma_0)_\infty\\
&\cong
\chi_V\St_l \times \dots \times \chi_V\St_l 
\times (\chi_V\St_l \rtimes \Theta_{W_{n-2hl},V_{m-2hl-2}}(\sigma_0)^\vee)
\end{align*}
by Proposition \ref{Ja}.
In particular, there exists an irreducible subquotient $\pi_0'$ of 
$\Theta_{W_{n_0},V_{m_0}}(\sigma_0)$ such that
\[
\chi_V\St_l \times \dots \times \chi_V\St_l \rtimes \pi'_0
\twoheadrightarrow \pi,
\]
where $\chi_V\St_l$ appears $h$-times.
This implies that the $L$-parameter for $\pi_0'$ is given by
$(\phi_0, \eta|A_{\phi_0})$, which is the same as the one for $\pi_0$.
Also, if $G(W_n)$ is an odd orthogonal group, 
the central character of $\pi_0'$ coincides with the one of $\pi_0$.
Hence we have $\pi_0' \cong \pi_0$.
Since $\phi_{\sigma_0}$ contains $\chi_W S_l$ with multiplicity one,
by Proposition \ref{irred}, we see that $\Theta_{W_{n_0},V_{m_0}}(\sigma_0)$ is irreducible, 
and so that 
$\Theta_{W_{n_0},V_{m_0}}(\sigma_0)=\theta_{W_{n_0},V_{m_0}}(\sigma_0)=\pi_0$.
In other words, we have $\sigma_0=\theta_{V_{m_0}, W_{n_0}}(\pi_0)$.
This completes the proof.
\end{proof}

\subsection{Higher lifts}\label{higher}
In this subsection, we prove Theorem \ref{main3} (3).
\vskip 5pt

Let $\pi \in \Irr_\temp(G(W_n))$ with $L$-parameter $(\phi,\eta)$, 
and $\sigma =\theta_{V_m,W_n}(\pi) \in \Irr(H(V_m))$ be the first occurrence 
to the going-up tower i.e., $m=m^\up(\pi)$.
Assume that $\sigma$ is non-tempered.
Then  $l(\pi)+2=m-n-\epsilon_0>2$.
Let $\sigma' = \theta_{V_{m'},W_n}(\pi)$ be a higher lift, i.e., $m'>m$.
The assertion of Theorem \ref{main3} (3) follows from \cite[Proposition 3.2]{GT1}
if we knew that this proposition can be applied to $\sigma$ and $\sigma'$.
So what we have to show is as follows:
\begin{prop}\label{GT32}
We can apply \cite[Proposition 3.2]{GT1} to $\sigma$ and $\sigma'$.
Namely, the same assertion as Proposition \ref{L-quot} is true for 
$\sigma =\theta_{V_m,W_n}(\pi)$ and $\sigma' = \theta_{V_{m'},W_n}(\pi)$.
\end{prop}
\begin{proof}
We freely use the notation of \cite{GT1}.
According to the proof of Proposition 3.2 in \cite{GT1}, 
it suffices show that 
only the $0$-th piece $R_0$ of the filtration of Lemma 2.2 in \cite{GT1} can contribute 
in the proof of Proposition 3.2 in \cite{GT1} for $\sigma$ and $\sigma'$.
\vskip 5pt

Suppose that $R_t$ contributes for some $t>0$.
Then we have a nonzero $\GL(Y_t)$-homomorphism
\[
\chi_W|{\det}_{Y_t}|^{s} \rightarrow R_{\overline{Q(Y_t)}}(\sigma), 
\]
where
\begin{itemize}
\item
$V_m=Y_t+V_{m_0}+Y_t^*$ with $m_0=m-2t$;
\item
$s=(m+r-n-\epsilon_0)/2+t/2>0$ for some $r\geq0$.
\end{itemize}
See also the argument after Lemma 2.2 in \cite{GT1}.
\vskip 5pt

Put
\[
\Sigma = \sum_{f\in \Hom_{\GL(Y_t)}(\chi_W|{\det}_{Y_t}|^s, R_{\overline{Q(Y_t)}}(\sigma))}
\mathrm{Im}(f).
\]
This is a $\GL(Y_t)\times H(V_{m_0})$-subrepresentation of $R_{\overline{Q(Y_t)}}(\sigma)$ 
of the form
\[
\Sigma = \chi_W|{\det}_{Y_t}|^s \boxtimes \Sigma_0,
\]
where $\Sigma_0$ is a nonzero smooth representation of $H(V_{m_0})$.
Since $R_{\overline{Q(Y_t)}}(\sigma)$ is finite length, so is $\Sigma_0$.
Hence we can find an irreducible subrepresentation $\sigma_0$ of $\Sigma_0$.
We obtain a nonzero $\GL(Y_t)\times H(V_{m_0})$-homomorphism
\[
\chi_W|{\det}_{Y_t}|^s \boxtimes \sigma_0 \rightarrow R_{\overline{Q(Y_t)}}(\sigma).
\]
By Bernstein's Frobenius reciprocity, we have a surjection
\[
\Ind_{Q(Y_t)}^{H(V_m)}(\chi_W|{\det}_{Y_t}|^s \boxtimes \sigma_0) \twoheadrightarrow \sigma.
\]
\vskip 5pt

By Lemma \ref{sub-quot}, this surjection gives an injection
\[
\sigma \hookrightarrow \Ind_{Q(Y_t)}^{H(V_m)}(\chi_W|{\det}_{Y_t}|^{-s} \boxtimes \sigma_0).
\]
Hence we have
\begin{align*}
\pi^* &\hookrightarrow \Hom_{H(V_m)}(\omega_{V_m,W_n}, \sigma)\\
&\hookrightarrow \Hom_{H(V_m)}(\omega_{V_m,W_n}, 
\Ind_{Q(Y_t)}^{H(V_m)}(\chi_W|{\det}_{Y_t}|^{-s} \boxtimes \sigma_0))\\
&\cong \Hom_{\GL(Y_t) \times H(V_{m_0})}(R_{Q(Y_t)}(\omega_{V_m,W_n}), 
\chi_W|{\det}_{Y_t}|^{-s} \boxtimes \sigma_0
).
\end{align*}
By Kudla's filtration (Lemma \ref{kudla}),
we see that there is a nonzero homomorphism
\[
\pi^\vee \rightarrow 
\Hom_{\GL(Y_t) \times H(V_{m_0})}(J^a, \chi_W|{\det}_{Y_t}|^{-s} \boxtimes \sigma_0)_\infty
\]
for some $0\leq a \leq t$.
\vskip 5pt

First, consider the case when $0 \leq a < t$.
By the definition of the normalized Jacquet module, we have
\[
R_{\overline{Q(Y_{t-a},Y_t)}}(\chi_W|{\det}_{Y_t}|^{-s})
= \chi_W |{\det}_{Y_{t-a}}|^{-s+a/2} \boxtimes \chi_W |{\det}_{Y'_a}|^{-s-(t-a)/2}.
\]
Note that $\GL(Y_{t-a})$ acts on $J^a$ by the character 
\[
\chi_W |{\det}_{Y_{t-a}}|^{(n-m+\epsilon_0+t-a)/2}.
\]
Since $t-a>0 \geq -r/2$, we have
\[
(n-m+\epsilon_0+t-a)/2 \not= -(m+r-n-\epsilon_0)/2-t/2+a/2.
\]
Hence we have
\[
\Hom_{G(W_n)}(\pi^\vee,
\Hom_{\GL(Y_t) \times H(V_{m_0})}(J^a, \chi_W|{\det}_{Y_t}|^{-s} \boxtimes \sigma_0)_\infty
)=0.
\]
\vskip 5pt

We conclude that there must be an injection
\[
\pi^\vee \hookrightarrow 
\Hom_{\GL(Y_t) \times H(V_{m_0})}(J^t, \chi_W|{\det}_{Y_t}|^{-s} \boxtimes \sigma_0)_\infty.
\]
However, 
\[
\Hom_{\GL(Y_t) \times H(V_{m_0})}(J^t, \chi_W|{\det}_{Y_t}|^{-s} \boxtimes \sigma_0)_\infty
\cong 
\big(\Ind_{P_t}^{G(W_{n})}
(\chi_V|{\det}_{X_t}|^s \boxtimes \Theta_{W_{n_0},V_{m_0}}(\sigma_0))\big)^\vee.
\]
Since $s>0$,
it has no irreducible tempered subrepresentations by Casselman's criterion.
\vskip 5pt

We obtain a contradiction, so that $R_t$ cannot contribute for any $t>0$.
\end{proof}

\subsection{Central characters of representations of odd orthogonal groups}\label{sec.central}
Recall that for an odd orthogonal group $\Oo(V_m)$, 
our local Langlands correspondence described in \S \ref{LLC} or Appendix \ref{app LLC}
parametrizes $\Irr(\Oo(V_m))$ by the triples $(\phi,\eta,\nu)$.
More precisely, a pair $(\phi,\eta)$ corresponds to the set
\[
\{\sigma,\ \sigma \otimes \det\}
\]
for some $\sigma \in \Irr(\Oo(V_m))$, 
and 
\[
\nu \colon \Irr(\Oo(V_m)) \rightarrow \{\pm1\}
\]
is given by the central character, i.e.,
$\sigma(-\1_{V_m})=\nu(\sigma) \cdot \id$ for $\sigma \in \Irr(\Oo(V_m))$.
\vskip 5pt

In this subsection, 
we consider the theta correspondence for $(\Mp(W_n),\Oo(V_m))$, 
i.e., $E=F$, $\epsilon=+1$, $m$ is odd and $n$ is even.
We prove Theorems \ref{main2} (5), \ref{main3} (4) and 
complete the proof of Theorem \ref{main1} (2).
Namely, we treat the following two issues:
\begin{enumerate}
\item
For $\pi \in \Irr_\temp(\Mp(W_n))$ with $\theta_{V_m,W_n}(\pi)\not=0$, 
determine $\nu(\theta_{V_m,W_n}(\pi))$.
\item
For $[\sigma] \in \Irr_\temp(\Oo(V_m))$, 
determine
which tower
$\{\Theta_{W_n, V_m}(\sigma)\}_n$ or $\{\Theta_{W_n, V_m}(\sigma \otimes \det)\}_n$
is the going-down tower.
\end{enumerate}
\vskip 5pt

First, we consider (1).
Let $\pi \in \Irr(\Mp(W_n))$ and assume that $\sigma=\theta_{V_m,W_n}(\pi)$ is nonzero so that
$\sigma \in \Irr(\Oo(V_m))$.
We define $\epsilon(V)\in\{\pm1\}$ by
\[
\epsilon(V)=\left\{
\begin{aligned}
&+1	\iif \text{$\Oo(V_m)$ is split},\\
&-1	\iif \text{$\Oo(V_m)$ is non-split}.
\end{aligned}
\right.
\]
Note that $\epsilon(V)=\eta_\sigma(z_{\phi_\sigma})$ 
by Desideratum \ref{des} (3),
where $(\phi_\sigma,\eta_\sigma)$ is the $L$-parameter for $\sigma$.
The following proposition is Theorem \ref{main2} (5) and Theorem \ref{main3} (4).
\begin{prop}\label{sign}
Let $\pi \in \Irr_\temp(\Mp(W_n))$ with $L$-parameter $(\phi_\pi,\eta_\pi)$.
Assume that $\sigma=\theta_{V_m,W_n}(\pi)$ is nonzero.
Then we have
\[
\nu(\sigma)=\eta_\pi(z_{\phi_\pi}) \cdot \ep(\phi_\pi) \cdot \chi_V(-1)^{\half{n}}.
\]
\end{prop}
\begin{proof}
By \cite[\S 5.2]{GI1}, we see that $\nu(\theta_{V_m,W_n}(\pi))$ does not depend on $m$.
Hence we may assume that $m \geq n+1$.
By applying \cite[Theorem 11.1]{GI1} to $\sigma$ in the theta correspondence for 
$(\Oo(V_m),\Mp(W_{m-1}))$, we have
\[
\nu(\sigma)=\ep(\phi_\sigma) \cdot \epsilon(V),
\]
where $\phi_\sigma$ is the $L$-parameter for $[\sigma]$.
By Theorems \ref{main2} and \ref{main3}, we see that
\[
\ep(\phi_\sigma)
=\left\{
\begin{aligned}
&\ep(\phi_\pi \otimes \chi_V) \iif \text{$\{V_m\}_m$ is the going-down tower},\\
&-\ep(\phi_\pi \otimes \chi_V) \iif \text{$\{V_m\}_m$ is the going-up tower}.
\end{aligned}
\right.
\]
On the other hand, by Theorem \ref{Mp}, we see that
\[
\epsilon(V)
=\left\{
\begin{aligned}
&\eta_\pi(z_{\phi_\pi}) \cdot \ep(\phi_\pi) \cdot \ep(\phi_\pi\otimes \chi_V) \cdot \chi_V(-1)^{\half{n}} 
\iif \text{$\{V_m\}_m$ is the going-down tower},\\
&-\eta_\pi(z_{\phi_\pi}) \cdot \ep(\phi_\pi) \cdot \ep(\phi_\pi\otimes \chi_V) \cdot \chi_V(-1)^{\half{n}}
\iif \text{$\{V_m\}_m$ is the going-up tower}.
\end{aligned}
\right.
\]
These equations imply the proposition.
\end{proof}

Next, we consider (2).
Let $\sigma \in \Irr(\Oo(V_m))$.
We compare two towers $\{\Theta_{W_n, V_m}(\sigma)\}_n$ and 
$\{\Theta_{W_n, V_m}(\sigma \otimes \det)\}_n$. 
\vskip5pt

\begin{prop}\label{sign2}
Let $\sigma \in \Irr_\temp(\Oo(V_m))$ with $L$-parameter $(\phi_\sigma, \eta_\sigma, \nu_\sigma)$.
Then $\{\Theta_{W_n, V_m}(\sigma)\}_n$ is the going-down tower with respect to $\sigma$, 
i.e., 
\[
\min\{n\ |\ \Theta_{W_n, V_m}(\sigma) \not=0\} 
\leq
\min\{n\ |\ \Theta_{W_n, V_m}(\sigma \otimes \det) \not=0\} 
\]
if and only if
\[
\nu_\sigma = \eta_\sigma(z_{\phi_\sigma}) \cdot \ep(\phi_\sigma).
\]
\end{prop}
\begin{proof}
Note that $\{\Theta_{W_n, V_m}(\sigma)\}_n$ is the going-down tower
if and only if $\Theta_{W_{m-1},V_m}(\sigma)$ is nonzero.
This is equivalent to
$\nu_\sigma = \epsilon(V) \cdot \ep(\phi_\sigma) 
= \eta_\sigma(z_{\phi_\sigma}) \cdot \ep(\phi_\sigma)$
by \cite[Theorem 11.1]{GI1}.
\end{proof}
Together with Proposition \ref{center1}, this completes the proof of Theorem \ref{main1} (2).

\appendix
\section{Preparations for the local Langlands correspondence}
In this appendix, we recall some basic results on standard gamma factors, 
Plancherel measures, and
local factors associated to representations of Weil--Deligne groups.

\subsection{Standard gamma factors}\label{std}
Fix a non-trivial additive character $\psi$ of $F$.
For $\pi\in\Irr(G(W_n))$ and a character $\chi$ of $E^\times$, 
let $\gamma(s,\pi,\chi,\psi)$ be the standard $\gamma$-factor
defined by Lapid--Rallis \cite{LR} using the doubling method.
For its properties, see \cite{LR}, \cite{G} and \cite[\S 10, \S 11]{GI1}.
The property which we need is as follows:
\begin{prop}[{\cite[Theorem 11.2]{GI1}}]\label{pole}
Let $\pi\in\Irr_\temp(G(W_{n}))$.
Assume that $\Theta_{V_m,W_n}(\pi)\not=0$ and 
$l=n-m+\epsilon_0>0$.
Then $\gamma(s,\pi,\chi_V^{-1},\psi)$ has a pole at $s=\frac{l+1}{2}$.
\end{prop}

\subsection{Plancherel measures}\label{Pm}
Let $G$ be a reductive group over $F$ and
$P=MU$ be a parabolic subgroup of $G$.
For $\pi\in \Irr(M)$, consider the normalized induced representation
\[
I_P^G(\pi)\coloneqq\Ind_P^G(\pi).
\]
We define an intertwining operator 
\[
J_{\overline{P}|P}(\pi)\colon I_P^G(\pi) \rightarrow I_{\overline{P}}^G(\pi)
\]
by
\[
J_{\overline{P}|P}(\pi)f(g)=\int_{\overline{U}}f(\overline{u}g)d\overline{u}
\quad\text{for $f \in I_P^G(\pi)$},
\]
where $\overline{P}=M\overline{U}$ is the parabolic subgroup of $G$ opposite to $P$.
Then there exists a rational function $\mu$ of $\pi$ such that
\[
J_{P|\overline{P}}(\pi) \circ J_{\overline{P}|P}(\pi) = \mu(\pi)^{-1}.
\]
The rational function $\mu$ is called the Plancherel measure associated to $I_P^G(\pi)$.
It is only well-defined up to a scalar 
since it depends on the choice of Haar measures on $U$ and $\overline{U}$.
We choose Haar measures as in \cite[\S B.2]{GI1}, 
which are determined by $\psi$.
We denote the corresponding Plancherel measure by $\mu_\psi$.
\vskip 5pt

Let $(V_m,W_n)$ be as in \S \ref{spaces},
and put $W_{n_1}=W_{n}+\H^k$ and $V_{m_1}=V_{m}+\H^k$
with $n_1=n+2k$ and $m_1=m+2k$.
We consider the maximal parabolic subgroups $P=M_PU_P$ and $Q=M_QU_Q$
of $G(W_{n_1})$ and $H(V_{m_1})$ with Levi components
\[
M_P=\GL_k(E)\times G(W_{n})
\quad\text{and}\quad
M_Q=\GL_k(E)\times H(V_{m}),
\]
respectively.

\begin{thm}[{\cite[Theorem 12.1]{GI1}}]\label{Pmeas}
Let $\pi\in\Irr(G(W_{n}))$ and put $\sigma=\theta_{V_{m},W_{n}}(\pi)$.
Assume that $\sigma\not=0$, so that $\sigma\in\Irr(H(V_{m}))$.
For $\tau\in\GL_k(E)$ and $s\in \C$, we put $\tau_s=\tau|\det|_E^s$.
Then we have
\begin{align*}
\frac{\mu_\psi(\tau_s\chi_V \otimes \pi)}{\mu_\psi(\tau_s\chi_W\otimes \sigma)}
&=\gamma(s-\frac{l-1}{2},\tau,\psi_E)\cdot\gamma(-s-\frac{l-1}{2},\tau^\vee,\psi_E^{-1}).
\end{align*}
\end{thm}

For metaplectic groups, we have to replace $\GL_k(E)$ with 
its double cover $\cl\GL_k(E)$.
More precisely, see \cite[\S 2.2--\S 2.5]{GS} and \cite[\S 2.5 and \S 2.6]{GI1}.

\subsection{Representations of Weil--Deligne groups}\label{repWD}
We denote by $W_E$ and $\WD_E=W_E \times \SL_2(\C)$
the Weil group and Weil--Deligne group of $E$, respectively.
Let $I_E$ be the inertia subgroup of $W_E$.
We fix a geometric Frobenius element $\Frob_E$ of $W_E$.
\vskip 5pt

If $E\not=F$, we regard $W_E$ as a subgroup $W_F$ such that $W_F/W_E \cong \Gal(E/F)$ 
and fix $s\in W_F \setminus W_E$.
If $E=F$, we put $s=1$.
\vskip 5pt

Let $M$ be a finite dimensional vector space over $\C$.
We say that a homomorphism $\phi \colon \WD_E \rightarrow \GL(M)$
is a representation of $\WD_E$ if
\begin{itemize}
\item
$\phi(\Frob_E)$ is semi-simple;
\item
the restriction of $\phi$ to $W_E$ is smooth;
\item
the restriction of $\phi$ to $\SL_2(\C)$ is algebraic.
\end{itemize}
We call $\phi$ tempered if the image of $W_E$ is bounded.
Let $\phi^\vee$ be the contragredient representation of $\phi$ defined by
$\phi^\vee(w)={}^t \phi(w)^{-1}$.
We define a representation ${}^c\phi$ of $\WD_E$ by
${}^c\phi(w)=\phi(sws^{-1})$.
Then the equivalence class of ${}^c\phi$ is independent of the choice of $s$.
\vskip 5pt

Fix $b\in\{\pm1\}$.
We say that $\phi$ is conjugate self-dual with sign $b$
if there exists a non-degenerate bilinear form 
$B\colon M \times M \rightarrow \C$ such that
\[
\left\{
\begin{aligned}
&B(\phi(w)x,\phi(sws^{-1})y) = B(x,y),\\
&B(y,x) = b \cdot B(x,\phi(s^2)y)
\end{aligned}
\right.
\]
for $x,y\in M$ and $w\in \WD_E$.
In this case, $\phi$ is equivalent to ${}^c\phi^\vee$.
If $E=F$, then $s=1$ and ${}^c\phi=\phi$.
In this case, we call $\phi$ self-dual with sign $b$.
We also say that $\phi$ is 
\[
\left\{
\begin{aligned}
&\text{orthogonal}\iif \text{$\phi$ is self-dual with sign $+1$},\\
&\text{symplectic}\iif \text{$\phi$ is self-dual with sign $-1$},\\
&\text{conjugate-orthogonal}\iif \text{$\phi$ is conjugate self-dual with sign $+1$},\\
&\text{conjugate-symplectic}\iif \text{$\phi$ is conjugate self-dual with sign $-1$}.
\end{aligned}
\right.
\]
More precisely, see \cite[\S 3]{GGP}.
\vskip 5pt

For each positive integer $k$, 
there exists a unique irreducible algebraic representation $S_k$ of $\SL_2(\C)$
with dimension $k$.
It is easy to see that
$S_k$ is (conjugate) self-dual with sign $(-1)^{k-1}$.
Moreover we have
\[
S_a \otimes S_b \cong \bigoplus_{k=1}^{\min\{a,b\}}S_{a+b+1-2k}
=S_{a+b-1} \oplus S_{a+b-3} \oplus \dots \oplus S_{|a-b|+1}
\]
for positive integers $a$ and $b$.

\subsection{Local factors}
We define local factors associated to representations of $\WD_E$.
Fix a non-trivial additive character $\psi'$ of $E$.
A representation $\phi$ of $\WD_E$ is written by
\[
\phi=\bigoplus_{n\geq1}\phi_n\boxtimes S_n,
\]
where $(\phi_n,M_n)$ is a representation of $W_E$.
Let $M_n^{I_E}$ be the subspace of $M_n$ consisting of $I_E$-fixed vectors.
Note that $M_n^{I_E}$ is a subrepresentation of $M_n$ and
$\phi_n(\Frob_E) \in \GL(M_n^{I_E})$ is independent of the choice of $\Frob_E$.
We define the local factors associated $\phi$ by
\begin{align*}
L(s,\phi)&=\prod_{n\geq1} \det(\1-q^{-(s+\frac{n-1}{2})}\phi_n(\Frob_E)|M_n^{I_E})^{-1}
=\prod_{n\geq1}L(s+\frac{n-1}{2},\phi_n),\\ 
\ep(s,\phi,\psi')
&=\prod_{n\geq1}\ep(s,\phi_n,\psi')^n\det(-q^{\frac{1}{2}-s}\phi_n(\Frob_E)|M_n^{I_E})^{n-1},\\
\gamma(s,\phi,\psi')&=\ep(s,\phi,\psi')\frac{L(1-s,\phi^\vee)}{L(s,\phi)}.
\end{align*}
For the definition of $\ep(s,\phi_n,\psi')$, see \cite[\S 3]{T}.
For $c \in E^\times$, we define the non-trivial additive character $\psi'_c$ of $E$ by 
$\psi'_c(x)=\psi'(cx)$.
It is known that
\[
\ep(s, \phi, \psi'_c)=\det(\phi)(c) \cdot |c|_E^{\dim(\phi)(s - \half{1})} \cdot \ep(s,\phi,\psi').
\]
The local functional equation asserts that
\[
\gamma(s,\phi,\psi') \cdot \gamma(1-s, \phi^\vee, \psi'^{-1})=1
\quad\text{or}\quad
\ep(s,\phi,\psi') \cdot \ep(1-s, \phi^\vee, \psi')=\det(\phi)(-1).
\]
In particular, if $\phi$ is self-dual with $\det(\phi)=\1$, 
then $\ep(1/2,\phi,\psi')$ is in $\{\pm1\}$ and independent of $\psi'$.
In this case, we write $\ep(\phi)\coloneqq\ep(1/2,\phi,\psi')$.
For $a\not\equiv b\bmod 2$, we have
\[
\ep(S_a \otimes S_b)=(-1)^{\min\{a,b\}}.
\]
If $E\not=F$ and $\phi$ is conjugate self-dual, 
then we write $\ep(\phi,\psi')\coloneqq\ep(1/2,\phi,\psi')$. 
By \cite[Propostition 5.1]{GGP}, if $E\not=F$ and ${}^c\psi'=\psi'^{-1}$, 
then $\ep(\phi,\psi')\in\{\pm1\}$.
Here, ${}^c\psi(x) = \psi({}^c x)$ for $x \in E$, where
${}^c x$ is the conjugate of $x$.
\vskip 5pt

We need some lemmas for local factors.
\begin{lem}\label{gamma}
Let $\phi$ be an irreducible representation of $W_E$ and $l$ be a positive integer.
Then we have
\[
\ep(s,\phi,\psi')^l\ep(-s,\phi^\vee,\psi'^{-1})^l
=\ep(s-\frac{l-1}{2},\phi,\psi')\ep(-s-\frac{l-1}{2},\phi^\vee,\psi'^{-1}),
\]
and
\[
\gamma(s,\phi\otimes S_l,\psi')\gamma(-s,\phi^\vee\otimes S_l,\psi'^{-1})
=\gamma(s-\frac{l-1}{2},\phi,\psi')\gamma(-s-\frac{l-1}{2},\phi^\vee,\psi'^{-1}).
\]
\end{lem}
\begin{proof}
Straightforward.
\end{proof}
\vskip 5pt

\begin{lem}\label{crit}
Let $\psi'$ be a non-trivial additive character of $E$, 
$\phi$ be a representation of $\WD_E$, and
$l$ be a positive integer.
Assume that 
\begin{itemize}
\item
$\psi'|F=\1$, i.e., ${}^c\psi'=\psi'^{-1}$ if $E\not=F$;
\item
$\phi$ is conjugate self-dual  with sign $(-1)^{l-1}$ if $E\not=F$;
\item
$\phi$ is self-dual  with sign $(-1)^{l-1}$ if $E=F$.
\end{itemize}
We define $\alpha_l(\phi)\in\{\pm1\}$ by
\[
\alpha_l(\phi)=\frac{\ep(\phi \otimes S_{l+1},\psi')}{\ep(\phi \otimes S_{l-1},\psi')}
\times \left\{
\begin{aligned}
&\det(\phi)(-1)	\iif E=F,\\
&1	\iif E\not=F.
\end{aligned}
\right.
\]
Here, if $l=1$, then we interpret $\ep(\phi \otimes S_{l-1},\psi') \coloneqq 1$. 
\begin{enumerate}
\item
Suppose that $\phi$ is irreducible. 
Then $\alpha_l(\phi)=-1$ if and only if $\phi=S_l$.
\item
If $\phi=\phi_0\oplus {}^c\phi_0^\vee$, then $\alpha_l(\phi)=1$.
\item
In general, $\alpha_l(\phi)=(-1)^{m_\phi(S_l)}$, 
where $m_\phi(S_l)$ is the multiplicity of $S_l$ in $\phi$.
\end{enumerate}
\end{lem}
\begin{proof}
Straightforward.
\end{proof}
\vskip 5pt

For a character $\chi$ of $E^\times$, we put
\[
\delta(\chi = \1) = \left\{
\begin{aligned}
&1\iif \chi = \1, \\
&-1\iif \chi \not= \1.
\end{aligned}
\right.
\]
\begin{lem}\label{delta}
Let $\chi$ be a quadratic character of $E^\times$, and $k$ be a positive integer. 
Then $\chi \otimes S_{2k}$ is a symplectic representation of $\WD_E$, and satisfies
\[
\ep(\chi \otimes S_{2k}) = - \delta(\chi=\1) \cdot \chi(-1)^k.
\]
\end{lem}
\begin{proof}
Since $\chi$ and $S_{2k}$ is self-dual representations with sign $+1$ and $-1$, respectively, 
we see that $\chi \otimes S_{2k}$ has sign $-1$.
By the definition of the $\ep$-factor, we have
\[
\ep(\chi \otimes S_{2k}) = \ep(\chi,\psi)^{2k} \cdot \det(-\chi(\Frob_E) | \C(\chi)^{I_E})^{2k-1}
=\chi(-1)^k \cdot \det(-\chi(\Frob_E) | \C(\chi)^{I_E})^{2k-1},
\]
where $\C(\chi)$ denotes the space of $\chi$.
If $\chi$ is ramified, then $\C(\chi)^{I_E}=0$ so that $\det(-\chi(\Frob_E) | \C(\chi)^{I_E})=1$.
If $\chi$ is unramified, then we have
\[
\det(-\chi(\Frob_E) | \C(\chi)^{I_E})=
\left\{
\begin{aligned}
&-1	\iif \chi=\1,\\
&1	\iif \text{$\chi$ is the unique non-trivial unramified quadratic character}.
\end{aligned}
\right.
\]
Hence for any quadratic character $\chi$, we have
$\det(-\chi(\Frob_E) | \C(\chi)^{I_E})=-\delta(\chi=\1)$.
\end{proof}
\vskip 5pt

The following lemma is \cite[Lemma 12.3]{GS} and \cite[Lemma A.6]{GI2}.
\begin{lem}\label{converse}
Let $\phi_1$, $\phi_2$ be a tempered representations of $\WD_E$
with same dimension $n$.
Assume that
\[
\gamma(s,\phi_1^\vee \otimes \phi_\rho, \psi') \cdot 
\gamma(-s, \phi_1 \otimes \phi_\rho^\vee, \psi'^{-1})
=\gamma(s,\phi_2^\vee \otimes \phi_\rho, \psi') \cdot 
\gamma(-s, \phi_2 \otimes \phi_\rho^\vee, \psi'^{-1})
\]
for every irreducible representation $\phi_\rho$ of $W_E$.
Then 
\[
\phi_1\cong \phi_2
\]
as representations of $\WD_E$.
\end{lem}

\section{Local Langlands correspondence}\label{app LLC}
In this paper, we assume the local Langlands correspondence for classical groups, 
which parametrizes irreducible representations.
For general linear groups, 
it was established by Harris--Taylor \cite{HT}, Henniart \cite{He}, and Scholze \cite{Sc}.
For other classical groups, 
it is known by Arthur \cite{Ar}, Mok \cite{Mo}, and Kaletha--M\'{i}nguez--Shin--White \cite{KMSW},
under some assumption on the stabilization of twisted trace formulas.
For this assumption, see also the series of papers
\cite{Stab1}, \cite{Stab2}, \cite{Stab3}, \cite{Stab4}, \cite{Stab5}, 
\cite{Stab6}, \cite{Stab7}, \cite{Stab8}, \cite{Stab9} and \cite{Stab10}
of Waldspurger and M{\oe}glin--Waldspurger, 
and papers of Chaudouard--Laumon \cite{CL1} and \cite{CL2}.
For metaplectic groups,
it was established by the second author and Savin \cite{GS}.
In this appendix, we summarize some of its properties which are used in this paper.

\subsection{Parameters and its component groups}
In this subsection, we define parameters and its component groups 
for (possibly disconnected) classical groups.
More precisely, see \cite{Ar} and \cite{GGP}.
\vskip 5pt

Fix $\epsilon\in\{\pm1\}$.
Let $V_m$ be an $\epsilon$-Hermitian space of dimension $m$ and
$G=H(V_m)$ be the isometry group of $V_m$.
Let $\Phi(H(V_m))$ is the set of equivalence classes of representations $\phi$ of $\WD_E$ 
of dimension $m-\epsilon_0$
which are
\[
\left\{
\begin{aligned}
&\text{conjugate self-dual with sign $(-1)^{m-1}$}, 
\iif E\not=F,\\
&\text{self-dual with sign $+1$ such that $\det(\phi)=\chi_V$}, 
\iif E=F,\ \epsilon=+1\text{ and }m\in2\Z,\\
&\text{self-dual with sign $-\epsilon$ such that $\det(\phi)=\1$}, \other.
\end{aligned}
\right.
\]
In particular, if $E=F$, $\epsilon=+1$ and $m=1$, then
$\Phi(H(V_1))=\{\text{the zero representation of $\WD_E$}\}$.
We call an element in $\Phi(H(V_m))$ a parameter for $H(V_m)$.
We denote by $\Phi_\temp(H(V_{m}))$ the subset of equivalence classes of
tempered representations.
\vskip 5pt

If $E=F$ and $G=H(V_m)$, we denote by $\widehat{G}$ the Langlands dual group of $G$.
It is given by
\[
\widehat{G}=
\left\{
\begin{aligned}
&\Sp_{m-1}(\C) \iif \text{$E=F$, $\epsilon=+1$ and $m$ is odd},\\
&\SO_{m+1}(\C) \iif \text{$E=F$, $\epsilon=-1$},\\
&\SO_m(\C)	\iif \text{$E=F$, $\epsilon=+1$ and $m$ is even}.
\end{aligned}
\right.
\]
Let $\phi \in \Phi(H(V_m))$. 
We denote the space of $\phi$ by $M$ and 
the $\WD_E$-invariant bilinear form on $M$ by $B$.
Let 
\[
C_\phi=\{g \in \GL(M)\ |\ 
\text{$B(gx,gy)=B(x,y)$ for any $x, y \in M$,
and 
$g \phi(w) = \phi(w) g$ for any $w \in \WD_E$}
\}
\]
be the centralizer of $\mathrm{Im}(\phi)$ in $\mathrm{Aut}(M,B)$.
Also we put
\[
C_\phi^+ =
\left\{
\begin{aligned}
&C_\phi \cap \SL(M) \iif \text{$E=F$ and $m$ is even},  \\
&C_\phi \other.
\end{aligned}
\right.
\]
Finally, we define the component groups $A_\phi$ and $A_\phi^+$ of $\phi$ by
\[
A_\phi = \pi_0(C_\phi)
\quad\text{and}\quad
A_\phi^+ = \pi_0(C_\phi^+),
\]
respectively.
\vskip 5pt

Let $\phi\in\Phi(H(V_m))$.
For an irreducible representation $\phi_0$ of $\WD_E$, 
we denote the multiplicity of $\phi_0$ in $\phi$ by $m_\phi(\phi_0)$.
We can decompose 
\[
\phi=m_1\phi_1+\dots+m_r\phi_r+\phi'+{}^c\phi'^\vee,
\]
where $\phi_1,\dots, \phi_r$ are distinct irreducible (conjugate) self-dual representations
of $\WD_E$ with the same type as $\phi$, $m_i=m_\phi(\phi_i)$, 
and $\phi'$ is a sum of irreducible representations of $\WD_E$
which are not (conjugate) self-dual with the same type as $\phi$.
Then by \cite[\S 4]{GGP}, $A_\phi$ is described as follows:
\[
A_\phi=\bigoplus_{i=1}^{r}(\Z/2\Z) a_i \cong (\Z/2\Z)^r.
\]
Namely, $A_\phi$ is a free $\Z/2\Z$-module of rank $r$ with a canonical basis $\{a_i\}$
indexed by the summands $\phi_i$ of $\phi$.
For $a=a_{i_1}+\dots+a_{i_k} \in A_\phi$ with $1\leq i_1 < \dots < i_k \leq r$, 
we put
\[
\phi^{a}=\phi_{i_1} \oplus \dots \oplus \phi_{i_k}.
\]
Also, we denote 
\[
z_\phi \coloneqq \sum_{i=1}^r m_\phi(\phi_i)\cdot a_i 
=\sum_{i=1}^s m_i \cdot a_i \in A_\phi.
\]
This is the image of $-\1$ in $C_\phi$.
We call $z_\phi$ the central element in $A_\phi$.
The determinant map $\det \colon \GL(M) \rightarrow \C^\times$ gives a homomorphism
\begin{align*}
\det \colon A_\phi \rightarrow \Z/2\Z, \quad
\sum_{i=1}^r \ep_i a_i \mapsto \sum_{i=1}^r \ep_i \cdot \dim(\phi_i),
\end{align*}
where $\ep_i\in\{0,1\} = \Z/2\Z$.
Then the group $A_\phi^+$ can be described as follows (\cite[Theorem 8.1]{GGP}):
\[
A_\phi^+=
\left\{
\begin{aligned}
&\ker(\det)\iif \text{$E=F$ and $m$ is even},\\
&A_\phi	\other.
\end{aligned}
\right.
\]
\vskip 5pt

We say that a parameter $\phi$ is discrete if
$m_i=1$ for any $i=1,\dots, r$ and $\phi'=0$, i.e.,
$\phi$ is a multiplicity-free sum of irreducible (conjugate) self-dual representations of $\WD_E$ 
with the same type as $\phi$.
We denote by $\Phi_\disc(H(V_{m}))$ the subset of equivalence classes of
discrete parameters.
Then we have a sequence
\[
\Phi_\disc(H(V_{m})) \subset \Phi_\temp(H(V_{m})) \subset \Phi(H(V_{m})).
\]

\subsection{Local Langlands correspondence for connected classical groups}
In this subsection, we introduce $\Pi(H(V_m))$ and 
state some properties of the local Langlands correspondence which we need.
\vskip 5pt

First, we consider orthogonal groups.
So we assume that $E=F$ and $\epsilon=+1$,
and we write $H(V_m)=\Oo(V_m)$.
We define equivalence relations $\sim_{\det}$ on $\Irr(\Oo(V_m))$ and 
$\sim_\ep$ on $\Irr(\SO(V_m))$ by 
\[
\sigma \sim_{\det} \sigma \otimes \det
\quad\text{and}\quad
\sigma_0 \sim_{\ep} \sigma_0^\ep
\]
for $\sigma \in \Irr(\Oo(V_m))$ and $\sigma_0 \in \Irr(\SO(V_m))$. 
Here, we fix an element $\ep \in \Oo(V_m) \setminus \SO(V_m)$
and define $\sigma_0^\ep$ by $\sigma_0^\ep(h)=\sigma_0(\ep^{-1}h\ep)$ 
for $\sigma_0 \in \Irr(\SO(V_m))$ and $h \in \SO(V_m)$.
Note that $\sigma|\SO(V_m) \cong (\sigma \otimes \det)|\SO(V_m)$ for $\sigma \in \Irr(\Oo(V_m))$,
and 
$\Ind_{\SO(V_m)}^{\Oo(V_m)}(\sigma_0) \cong \Ind_{\SO(V_m)}^{\Oo(V_m)}(\sigma_0^\ep)$
for $\sigma_0 \in \Irr(\SO(V_m))$. 
The restriction and the induction give a canonical bijection 
\[
\Irr(\Oo(V_m))/\sim_{\det} \longleftrightarrow \Irr(\SO(V_m))/\sim_\ep.
\]
In \cite{Ar}, one has parametrized not $\Irr(\SO(V_m))$ but $\Irr(\SO(V_m))/\sim_\ep$.
Via the above bijection, we translate the parametrization for $\Irr(\Oo(V_m))/\sim_{\det}$.
\vskip 5pt

We return the general setting.
Let $E$ be either $F$ or a quadratic extension of $F$, 
$V_m$ be an $\epsilon$-Hermitian space of dimension $m$ for fixed $\epsilon \in \{\pm1\}$,
and $H(V_m)$ be the isometry group of $V_m$.
We define $\Pi(H(V_m))$ by
\[
\Pi(H(V_m))=
\left\{
\begin{aligned}
&\Irr(H(V_m))/\sim_{\det} \iif \text{$E=F$ and $\epsilon=+1$},\\
&\Irr(H(V_m))	\other.
\end{aligned}
\right.
\]
For $\pi \in \Irr(H(V_m))$, we denote the image of $\pi$ under the canonical map 
$\Irr(H(V_m)) \rightarrow \Pi(H(V_m))$ by $[\pi]$.
Also, we denote the image of $\Irr_*(H(V_m))$ in $\Pi(H(V_m))$ by $\Pi_*(H(V_m))$
for $*=\disc$ or $\temp$.
\vskip 5pt

If $E\not=F$ or $\epsilon=+1$, 
then there exist exactly two Witt towers $\VV$ and $\VV'$
such that $V_m \in \VV$ and
\[
\left\{
\begin{aligned}
&\dim(V_m) \equiv \dim(V'_{m'}) \bmod 2 \iif E\not=F, \\
&\disc(V_{m}) = \disc(V'_{m'}) \iif \text{$E=F$ and $\epsilon=+1$}
\end{aligned}
\right.
\]
for $V'_{m'} \in \VV'$.
Let $\VV^+$ be the Witt tower whose anisotropic space is
\[
\left\{
\begin{aligned}
&0	\iif \text{$E\not=F$ and $m$ is even},\\
&(E,1)	\iif \text{$E\not=F$, $m$ is odd and $\epsilon=+1$},\\
&(E,\delta)	\iif \text{$E\not=F$, $m$ is odd and $\epsilon=-1$},\\
&0	\iif \text{$E=F$, $m$ is even and $\disc(V_m)=1$},\\
&(F(\sqrt{d}), \tr_{F(\sqrt{d})/F})	
\iif \text{$E=F$, $m$ is even and $d \coloneqq \disc(V_m)\not=1$ in $F^\times/F^{\times2}$},\\
&(F, 2\disc(V_m))	\iif \text{$E=F$ and $m$ is odd}.
\end{aligned}
\right.
\]
We denote the other Witt tower by $\VV^-$.
A pure inner form of $H(V_m)$ is $H(V^+_m)$ or $H(V^-_m)$, where $V^\pm_m\in\VV^\pm$.
If $E=F$ and $\epsilon=-1$, a pure inner form of $H(V_m)$ is $H(V_m)$ itself only.
\vskip 5pt

Now we are ready to describe the desiderata for the Langlands correspondence.
\begin{des}\label{des}
\begin{enumerate}
\item
There exists a canonical surjection
\[
\bigsqcup_{V_m^\bullet} \Pi(H(V_m^\bullet))\rightarrow \Phi(H(V_m)),
\]
where $V_m^\bullet$ runs over the spaces such that 
$H(V_m^\bullet)$ is a pure inner form of $H(V_m)$.
For $\phi \in \Phi(H(V_m))$, we denote by $\Pi_\phi^0$
the inverse image of $\phi$ under this map,
and call $\Pi_\phi^0$ the $L$-packet of $\phi$.
\item
There exists a bijection
\[
\iota \colon \Pi_\phi^0 \rightarrow \widehat{A_\phi^+},
\]
which satisfies certain character identities.
Here, we denote by $\widehat{A_\phi^+}$ the Pontryagin dual of $A_\phi^+$.
\item\label{center}
Let $[\pi] \in \Pi_\phi^0$ with $\iota([\pi])=\eta$.
Then $[\pi] \in \Pi(H(V^-_m))$ if and only if $z_\phi \in A_\phi^+$ and $\eta(z_\phi)=-1$.
\item
We have 
\[
\bigsqcup_{V_m^\bullet} \Pi_*(H(V^\bullet_m)) = \bigsqcup_{\phi \in \Phi_*(H(V_m))}\Pi_\phi^0
\]
for $* \in \{\disc, \temp\}$.
\item
Assume that $\phi=\phi_\tau+\phi_0+{}^c\phi_\tau^\vee$, 
where $\phi_0$ is an element in $\Phi_\temp(H(V_{m_0}))$ and 
$\phi_\tau$ is an irreducible tempered representation of $\WD_E$
which corresponds to $\tau \in \Irr_\temp(\GL_k(E))$.
Then the induced representation 
\[
\Ind_Q^{H(V_m)}(\tau \otimes \pi_0)
\]
is a direct sum of tempered representations of $H(V_m)$,
where $Q$ is a parabolic subgroup of $H(V_m)$ with
Levi subgroup $L_Q=\GL_k(E) \times H(V_{m_0})$ and
$\pi_0$ is (a representative of) an element in $\Pi_{\phi_0}^0$.
The $L$-packet $\Pi_\phi^0$ is given by
\[
\Pi_\phi^0 = \{[\pi]\ |\ \pi \subset \Ind_Q^{H(V_m)}(\tau \otimes \pi_0), [\pi_0] \in \Pi_{\phi_0}^0\}.
\]
Moreover if $\pi \subset \Ind_Q^{H(V_m)}(\tau \otimes \pi_0)$, 
then $\iota([\pi])|A_{\phi_0}^+ = \iota([\pi_0])$.
\item
Assume that 
\[
\phi=\phi_{\tau_1}|\cdot|^{s_1} + \dots +\phi_{\tau_r}|\cdot|^{s_r} + \phi_0
+{}^c(\phi_{\tau_1}|\cdot|^{s_1} + \dots +\phi_{\tau_r}|\cdot|^{s_r})^\vee,
\]
where $\phi_0$ is an element in $\Phi_\temp(H(V_{m_0}))$, 
$\phi_{\tau_i}$ is an irreducible tempered representation of $\WD_E$
which corresponds to $\tau_i \in \Irr_\temp(\GL_{k_i}(E))$,
and $s_1 \geq \dots \geq s_r>0$.
Then the $L$-packet $\Pi_\phi^0$ consists of (the equivalence classes of)
the unique irreducible quotients $\pi$ of 
the standard modules
\[
\tau_1|\det|_F^{s_1} \times \dots \times \tau_r|\det|_F^{s_r} \rtimes \pi_0,
\]
where $\pi_0$ runs over (representatives of) elements of $\Pi_{\phi_0}^0$.
Moreover if $\pi$ is the unique irreducible quotient of 
$\tau_1|\det|_F^{s_1} \times \dots \times \tau_r|\det|_F^{s_r} \rtimes \pi_0$, 
then $\iota([\pi])|A_{\phi_0}^+ = \iota([\pi_0])$.
\item\label{G-hyp}
The local Langlands correspondence respects the standard $\gamma$-factor.
Namely, we have
\[
\gamma(s,\pi,\chi,\psi) = \gamma(s, \phi \otimes \chi, \psi_E)
\]
for $\pi \in \Irr(H(V_m))$ whose parameter is $\phi$, and any character $\chi$ of $E^\times$.
Here, we put $\psi_E=\psi \circ \tr_{E/F}$.
\item\label{P-hyp}
The Plancherel measures are invariants of an $L$-packet.
Namely, if $\pi_1,\pi_2$ have the same parameter $\phi$, then we have
\[
\mu_\psi(\tau_s\otimes \pi_1)=\mu_\psi(\tau_s\otimes \pi_2)
\]
for any $\tau\in\Irr(\GL_k(E))$.
In particular, by a result of Shahidi \cite{Sh}, we have
\[
\mu_\psi(\tau_s\otimes\pi)
=\gamma(s,\phi_\tau\otimes\phi^\vee,\psi_E)\cdot \gamma(-s,\phi_\tau^\vee\otimes\phi,\psi_E^{-1})
\cdot \gamma(2s,R\circ\phi_\tau,\psi)\cdot \gamma(-2s,R\circ\phi_\tau^\vee,\psi^{-1})
\]
for any $\pi$ whose parameter is $\phi\in\Phi(H(V_m))$,
where 
\[
R=
\left\{
\begin{aligned}
&\Asai^+	\iif \text{$E\not=F$ and $m$ is even},\\
&\Asai^-	\iif \text{$E\not=F$ and $m$ is odd},\\
&\Sym^2		\iif \text{$E=F$, $\epsilon=+1$ and $m$ is odd},\\
&\wedge^2	\other.
\end{aligned}
\right.
\]
\end{enumerate}
\end{des}

The desiderata \ref{des} (\ref{G-hyp}) and (\ref{P-hyp}), 
at least for quasi-split classical groups, should follow from \cite{Ar} and \cite{Mo}, 
supplemented by some results of many others.
For non-quasi-split unitary groups, see also \cite{KMSW} and \cite[\S 1.4, Theorem 1.4.1]{Moe2}.
\vskip 5pt

\begin{rem}\label{non-can}
The bijection $\iota \colon \Pi_\phi^0 \rightarrow \widehat{A_\phi^+}$ may not be canonical.
It is determined by a choice of a Whittaker datum of a quasi-split pure inner form $H(V_m^\bullet)$.
If $m$ is odd, then $H(V_m^\bullet)$ has a unique Whittaker datum, so that $\iota$ is canonical.
Otherwise, we choose the Whittaker datum such that
\[
\iota=
\left\{
\begin{aligned}
&J_{\psi^E}	&\text{in \cite{GI2}}	\iif \text{$E\not=F$ and $\epsilon = +1$},\\
&J_\psi	&\text{in \cite{GI2}}	\iif \text{$E\not=F$ and $\epsilon = -1$},\\
&\iota_{\w_1}	&\text{in \cite{At}}	\iif \text{$E=F$ and $\epsilon = +1$},\\
&\iota_{\w'_1}	&\text{in \cite{At}}	\iif \text{$E=F$ and $\epsilon = -1$}.
\end{aligned}
\right.
\]
Here, in the first case, we fix $\delta \in E$ such that $\tr_{E/F}(\delta)=0$ and put
$\psi^E(x) = \psi(\half{1}\tr_{E/F}(\delta x))$ for $x \in E$.
\end{rem}
\vskip 5pt

The $L$-parameter for the contragredient representation $\pi^\vee$ of $\pi$ 
is described by Kaletha \cite{Ka}.
\begin{prop}[{\cite[Theorem 4.9]{Ka}}]\label{contra}
Let $\pi \in \Irr(H(V_m))$ with $L$-parameter $(\phi_\pi,\eta_\pi)$.
We denote the $L$-parameter for $\pi^\vee$ by $(\phi_{\pi^\vee},\eta_{\pi^\vee})$.
Then we have $\phi_{\pi^\vee} = \phi_\pi^\vee$. 
In particular, the component groups $A_{\phi_\pi}^+$ and $A_{\phi_{\pi}^\vee}^+$ 
are canonically identified.
Moreover, we have $\eta_{\pi^\vee}=\eta_\pi \cdot \eta_0$,
where $\eta_0$ is given by
\[
\eta_0(a)=
\left\{
\begin{aligned}
&\omega_{E/F}(-1)^{\dim(\phi_\pi^{a})}	\iif \text{$E\not=F$ and $m$ is even},\\
&\det(\phi_\pi^{a})(-1)	\iif \text{$E=F$ and $\epsilon=-1$},\\
&1	\other
\end{aligned}
\right.
\]
for $a \in A_{\phi_\pi}^+$.
\end{prop}
\vskip 5pt

\begin{rem}\label{Sp}
If $H(V_m) = \Sp(V_m)$ is a symplectic group, 
then $z_\phi \not \in A_\phi^+$ so that 
\[
A_\phi = A_\phi^+ \oplus (\Z/2\Z)z_\phi
\]
for each $\phi \in \Phi(\Sp(V_m))$.
Hence we may identify $\widehat{A_\phi^+}$ with
\[
\{\eta \in \widehat{A_\phi}\ |\ \eta(z_\phi) = 1\} \subset \widehat{A_\phi}.
\]
\end{rem}

If $H(V_m)$ is not an orthogonal group, we have $\Pi(H(V_m)) = \Irr(H(V_m))$.
In this case, we set $\Pi_\phi = \Pi_\phi^0$ for $\phi \in \Phi(H(V_m))$.
Using Remark \ref{Sp}, unless $H(V_m)$ is an orthogonal group, 
we may regard $\iota$ as an injection
\[
\iota \colon \Pi_\phi \hookrightarrow \widehat{A_\phi}.
\]
If $\pi \in \Pi_{\phi}$ and $\iota(\pi)=\eta \in \widehat{A_\phi}$, 
we call $(\phi,\eta)$ the $L$-parameter for $\pi$.

\subsection{Local Langlands correspondence for full orthogonal groups}
In this subsection, we explain the parametrization of $\Irr(\Oo(V_m))$.
Through this subsection, we assume $E=F$ and $\epsilon=+1$,
so that $H(V_{m})=\Oo(V_{m})$.
For $\phi \in \Phi(\Oo(V_m))$, we define the $L$-packet $\Pi_\phi$ of $\Oo(V_m)$, which 
is a subset of $\sqcup_{V_m^\bullet}\Irr(\Oo(V_m^\bullet))$ by 
the inverse image of $\Pi_\phi^0$ under the canonical map
\[
\bigsqcup_{V_m^\bullet}\Irr(\Oo(V_m^\bullet)) \rightarrow 
\bigsqcup_{V_m^\bullet}\Pi(\Oo(V_m^\bullet)) 
= \bigsqcup_{V_m^\bullet}\Irr(\Oo(V_m^\bullet))/\sim_{\det}.
\]
In the rest of this subsection, we parametrize $\Pi_\phi$.
\vskip 5pt

First, we assume that $m$ is odd.
Then $\Oo(V_{m}) = \SO(V_m) \times \{\pm\1_{V_m}\}$.
Any representation $\pi \in \Irr(\Oo(V_m))$ is determined by 
its image $[\pi]$ in $\Irr(\Oo(V_{m}))/\sim_{\det}$
and its central character $\omega_\pi \colon \{\pm\1_{V_m}\} \rightarrow \C^\times$.
Hence we have a bijection
\[
\Pi_\phi \rightarrow \widehat{A_\phi} \times \{\pm1\},\ 
\pi \mapsto (\iota([\pi]), \omega_\pi(-\1_{V_m^\bullet})).
\]
If $\pi \in \Pi_\phi$ corresponds to $(\eta, \nu) \in \widehat{A_\phi} \times \{\pm1\}$, 
we call the triple $(\phi, \eta, \nu)$ the $L$-parameter for $\pi$.
\vskip 5pt

Next, we assume that $m$ is even.
For $\phi \in \Phi(\Oo(V_m))$, we have an inclusion $A_\phi^+ \subset A_\phi$, so that
we obtain a canonical surjection
\[
\widehat{A_\phi} \twoheadrightarrow \widehat{A_\phi^+}.
\]
\begin{prop}
For $\phi \in \Pi_\phi$, we have $\#\Pi_\phi = \# \widehat{A_\phi}$.
Moreover, the following are equivalent:
\begin{itemize}
\item
$[A_\phi : A_\phi^+]=2$; 
\item
$\pi \otimes \det \not\cong \pi$ for some $\pi \in \Pi_\phi$;
\item
$\pi \otimes \det \not\cong \pi$ for any $\pi \in \Pi_\phi$.
\end{itemize}
\end{prop}
\begin{proof}
This follows from \cite[Proposition 3.2]{At}.
\end{proof}
\vskip 5pt

We fix $\epsilon \in \Oo(V_{m}) \setminus \SO(V_m)$ as in \cite{AG}, 
which depends on the choice of Whittaker datum.
Then \cite[Theorem 2.2.4]{Ar} gives a bijection
\[
\iota \colon \Pi_\phi \rightarrow \widehat{A_\phi}
\]
which satisfies a similar condition of Desiderata \ref{des} (2) -- (8), and
such that the diagram
\[
\begin{CD}
\Pi_\phi @>\iota>> \widehat{A_\phi}\\
@VVV @VVV\\
\Pi_\phi^0 @>\iota>> \widehat{A_\phi^+}
\end{CD}
\]
is commutative.
More precisely, see \cite{AG}.
If $\pi \in \Pi_\phi$ and $\iota(\pi)=\eta$, 
we call $(\phi, \eta)$ the $L$-parameter for $\pi$.

\subsection{Local Langlands correspondence for metaplectic groups}
In this subsection, we explain the parametrization of $\Irr(\Mp(W_{2n}))$.
Let $(W_{2n},V_m)$ be as in \S \ref{spaces}.
Through this subsection, we assume $E=F$, $\epsilon=+1$ and $m=2n+1$,
so that $G(W_{2n})=\Mp(W_{2n})$ and $H(V_{m})=\Oo(V_{2n+1})$.
\vskip 5pt

First, we recall a result of Gan--Savin.
\begin{thm}[{\cite[Theorem 1.1 and Corollary 1.2]{GS}}]\label{Mp-O}
Let $c \in F^\times / F^{\times2}$.
The theta correspondence gives a natural bijection (depending on the choice of $\psi$)
\[
\Irr(\Mp(W_{2n})) \rightarrow \bigsqcup_{V_{2n+1}^\bullet}\Irr(\Oo(V_{2n+1}^\bullet))/\sim_{\det}
=\bigsqcup_{V_{2n+1}^\bullet}\Pi(\Oo(V_{2n+1}^\bullet)),
\]
where the union is taken over all the isomorphism classes of orthogonal spaces
$V_{2n+1}^\bullet$ over $F$ with $\dim(V_{2n+1}^\bullet)=2n+1$ and $\disc(V_{2n+1}^\bullet)=c$.
\end{thm}

We describe this map more precisely.
There are exactly two isomorphism classes $V_{2n+1}$ and $V_{2n+1}'$ 
such that $\dim(V_{2n+1}) = \dim(V_{2n+1}') = 2n+1$ and 
$\disc(V_{2n+1}) = \disc(V_{2n+1}')=c$.
For $\pi \in \Irr(\Mp(W_{2n}))$, exactly one of two theta lifts 
$\Theta_{V_{2n+1},W_{2n}}(\pi)$ and $\Theta_{V_{2n+1}',W_{2n}}(\pi)$
is nonzero. 
If $\Theta_{V_{2n+1}^\bullet,W_{2n}}(\pi)$ is nonzero, then
the image of $\pi$ under this map is $[\theta_{V_{2n+1}^\bullet,W_{2n}}(\pi)]$.
Also, the inverse image can be described as follows:
For $\sigma \in \Irr(\Oo(V_{2n+1}^\bullet))$, 
exactly one of two theta lifts $\Theta_{W_{2n},V_{2n+1}^\bullet}(\sigma)$ and
$\Theta_{W_{2n},V_{2n+1}^\bullet}(\sigma \otimes \det)$ is nonzero,
and the image of $[\sigma] \in \Pi(\Oo(V_{2n+1}^\bullet))$
under the inverse map is the nonzero small theta lift $\theta_{W_{2n}, V_{2n+1}^\bullet}(\sigma)$ 
or $\theta_{W_{2n}, V_{2n+1}^\bullet}(\sigma \otimes \det)$.

\begin{cor}
The theta correspondence for $(\Mp(W_{2n},\Oo(V_{2n+1}^\bullet)))$ 
with $\disc(V_{2n+1}^\bullet)=1$
and the local Langlands correspondence for $\Oo(V_{2n+1}^\bullet)$
gives a surjection (depending on $\psi$) 
\[
\Irr(\Mp(W_{2n})) \rightarrow  \Phi(\Oo(V_{2n+1})).
\]
For $\phi \in \Phi(\Oo(V_{2n+1}))$, 
we denote by $\Pi_\phi$ the inverse image of $\phi$ under this map, 
and call $\Pi_\phi$ the $L$-packet of $\phi$. 
Moreover, the composition of $\iota$ for $\Oo(V_{2n+1})$ and theta lifts 
gives a bijection (depending on $\psi$) 
\[
\iota \colon \Pi_\phi \rightarrow \widehat{A_\phi}.
\]
\end{cor}
We define $\Phi(\Mp(W_{2n})) \coloneqq \Phi(\Oo(V_{2n+1}))$.
For $*=\disc$ or $\temp$, we put 
$\Phi_*(\Mp(W_{2n})) \coloneqq \Phi_*(\Oo(V_{2n+1}))$.
Then by \cite[Theorem 1.3]{GS},
we see that Desideratum \ref{des} (1), (2), (4), (5), (6), (7) and 
(8) for $R=\Sym^2$ are satisfied.
\vskip 5pt

We also need to know the theta correspondence for $(\Mp(W_{2n}, \Oo(V_{2n+1})))$
with $\disc(V_{2n+1})=c$.
Then $\chi_{V}=\chi_c$, where $\chi_c$ is the quadratic character of $F^\times$ 
associated to $c \in F^\times / F^{\times2}$.
\begin{thm}\label{Mp}
We write $c = \disc(V_{2n+1})$.
Let $\pi \in \Irr(\Mp(W_{2n}))$ and $\sigma \in \Irr(\Oo(V_{2n+1}))$
with $L$-parameters $(\phi_\pi,\eta_\pi)$ and $(\phi_\sigma, \eta_\sigma)$, respectively.
Assume that $\sigma=\theta_{V_{2n+1},W_{2n}}(\pi)$.
Then we have the following:
\begin{enumerate}
\item
We have
\[
\phi_\sigma = \phi_\pi \otimes \chi_c.
\]
In particular, we have a canonical identification $A_{\phi_\pi} = A_{\phi_\sigma}$.
\item
The characters $\eta_\pi$ and $\eta_{\sigma}$ are related by
\[
\eta_\sigma(a) / \eta_\pi(a) = 
\ep(\phi_\pi^{a}) \cdot \ep(\phi_\pi^{a} \otimes \chi_c) 
\cdot \chi_c(-1)^{\half{1}\dim(\phi_\pi^{a})} \in \{\pm1\}
\]
for any $a \in A_{\phi_\pi} = A_{\phi_\sigma}$.
\item
Let $(\phi_{\pi^\vee},\eta_{\pi^\vee})$ be the $L$-parameter for $\pi^\vee \in \Irr(\Mp(W_{2n}))$.
Then we have
\[
\phi_{\pi^\vee} = \phi_\pi \otimes \chi_{-1}
\quad\text{and}\quad
\eta_{\pi^\vee}(a) / \eta_\pi(a) = 
\ep(\phi_\pi^{a}) \cdot \ep(\phi_\pi^{a} \otimes \chi_{-1}) 
\cdot \chi_{-1}(-1)^{\half{1}\dim(\phi_\pi^{a})} \in \{\pm1\}
\]
for any $a \in A_{\phi_\pi} = A_{\phi_{\pi^\vee}}$.
\end{enumerate}
\end{thm}
\begin{proof}
This follows from {\cite[Theorem 1.5]{GS}}.
See also \cite[\S 3.6]{At}.
\end{proof}

\section{Gross--Prasad conjecture}\label{GPconj}
To prove main theorems, we used two highly non-trivial results.
The one is the Gross--Prasad conjecture, 
which gives an answer for restriction problems.
The other is Prasad's conjectures, which describe the local theta correspondence 
for (almost) equal rank cases.
In this appendix, we state the Gross--Prasad conjecture (GP).
\vskip 5pt

The Gross--Prasad conjecture consists of four cases; 
orthogonal, hermitian, symplectic-metaplectic, and skew-hermitian cases.
For each cases, the statements are slightly different.
So we states each cases separately.
We refer the reader to \cite[\S 6 and \S 18]{GGP} for 
a discussion of the various subtleties in the definition of the characters which appear in the statements of conjecture.
\vskip 5pt

First, we state the GP conjecture for the orthogonal cases.
\begin{thm}[GP conjecture for the orthogonal cases]\label{GGP-O}
For an orthogonal space $V_m^\bullet$, 
we put $V_{m+1}^\bullet=V_m^\bullet\oplus L_{(-1)^{m+1}}$, 
where $L_{(-1)^{m+1}}$ is the orthogonal space of dimension $1$ and discriminant $(-1)^{m+1}$.
We set $V_{\mathrm{even}}$ and $V_{\mathrm{odd}}$ so that 
\[
\{V_{\mathrm{even}},\ V_{\mathrm{odd}}\} = \{V_m,\ V_{m+1}\}
\quad\text{and}\quad
\dim(V_{\mathrm{even}}) \in 2\Z.
\]
For $\phi \in \Phi_\temp(\Oo(V_{\mathrm{even}}))$, 
$\phi' \in \Phi_\temp(\Oo(V_{\mathrm{odd}}))$
and $\nu \in \{\pm1\}$, 
there exists a unique pair $(\sigma,\sigma') \in \Pi_\phi \times \Pi_{\phi'}$ 
such that
\begin{itemize}
\item
$\sigma \otimes \sigma'$ is a representation of 
$\Oo(V_m^\bullet) \times \Oo(V_{m+1}^\bullet)$ for some $V_m^\bullet$;
\item
the central character of $\sigma'$ corresponds to $\nu$;
\item
$\Hom_{\Oo(V_m^\bullet)}(\sigma \otimes \sigma', \C)\not=0$.
\end{itemize}
Moreover, $\iota(\sigma)$ and $\iota(\sigma')$ are given by
\begin{align*}
&\left\{
\begin{aligned}
\iota(\sigma)(a) &= \ep(\phi^{a} \otimes \phi') 
\cdot \det(\phi^{a})(-1)^{\half{1}\dim(\phi')} \cdot \nu^{\dim(\phi^{a})},\\
\iota(\sigma')(a') &= \ep(\phi \otimes \phi'^{a'}) 
\cdot \det(\phi)(-1)^{\half{1}\dim(\phi'^{a'})} 
\end{aligned}
\right.
\end{align*}
for $a \in A_\phi$ and $a' \in A_{\phi'}$.
\end{thm}
The GP conjecture for the special orthogonal cases 
was proven by \cite{W2}, \cite{W3}, \cite{W4}, \cite{W5}.
In \cite{AG}, the authors extended this result to the full orthogonal cases
under an assumption on LLC for $\Oo(V_{2n})$.
\vskip 5pt

Secondly, we state the GP conjecture for the hermitian cases.
\begin{thm}[GP conjecture for the hermitian cases]\label{GGP-H}
Suppose that $E \not=F$.
For a hermitian space $V_m^\bullet$, 
we put $V_{m+1}^\bullet=V_m^\bullet\oplus L_{(-1)^{m}}$, 
where $L_{(-1)^{m}}$ is the hermitian space of dimension $1$ and discriminant $(-1)^{m}$.
For $\phi \in \Phi_\temp(\U(V_{m}))$ and $\phi' \in \Phi_\temp(\U(V_{m+1}))$, 
there exists a unique pair $(\sigma,\sigma') \in \Pi_\phi \times \Pi_{\phi'}$ 
such that
$\sigma \otimes \sigma'$ is a representation of 
$\U(V_m^\bullet) \times \U(V_{m+1}^\bullet)$ for some $V_m^\bullet$, and
\[
\Hom_{\U(V_m^\bullet)}(\sigma \otimes \sigma', \C)\not=0.
\]
Moreover, $\iota(\sigma)$ and $\iota(\sigma')$ are given by
\begin{align*}
&\left\{
\begin{aligned}
\iota(\sigma)(a) &= \omega_{E/F}(-1)^{(m+1)\dim(\phi^a)} \cdot \ep(\phi^{a} \otimes \phi', \psi^E_{2}),\\
\iota(\sigma')(a') &= \omega_{E/F}(-1)^{m\dim(\phi'^{a'})} \cdot \ep(\phi \otimes \phi'^{a'}, \psi^E_{2}) 
\end{aligned}
\right.
\end{align*}
for $a \in A_\phi$ and $a' \in A_{\phi'}$.
\end{thm}
The GP conjecture for the hermitian cases 
was proven by \cite{BP1}, \cite{BP2}, \cite{BP3}.
\vskip 5pt

Thirdly, we state the GP conjecture for the symplectic-metaplectic cases.
\begin{thm}[GP conjecture for the symplectic-metaplectic cases]\label{GGP-SM}
Let $W_n$ be a symplectic space.
For $c\in F^\times$, 
we denote by $\omega_{\psi_c}$ be the Weil representation of $\Mp(W_n \otimes L_1)$ 
associated to the additive character $\psi_c(x)\coloneqq\psi(cx)$ of $F$, 
where $L_1$ is the orthogonal space of dimension $1$ and discriminant $1$.
For $\phi \in \Phi_\temp(\Sp(W_n))$ and $\phi' \in \Phi_\temp(\Mp(W_{n}))$, 
there exists a unique pair $(\pi,\pi') \in \Pi_\phi \times \Pi_{\phi'}$ 
such that
\[
\Hom_{\Mp(W_n)}(\pi \otimes \pi', \omega_{\psi_c})\not=0.
\]
Moreover, $\iota(\pi)$ and $\iota(\pi')$ are given by
\begin{align*}
&\left\{
\begin{aligned}
\iota(\pi)(a) &= \ep(\phi^{a}\chi_c \otimes \phi') \cdot \ep(\phi\chi_c \otimes \phi')^{\det(a)}
\cdot \det(\phi^{a})(-1)^{\half{1}\dim(\phi')} \cdot \det(\phi^a)(c),\\
\iota(\pi')(a') &= \ep(\phi\chi_c \otimes \phi'^{a'}) \cdot \ep(\phi'^{a'}) 
\cdot \chi_c(-1)^{\half{1}\dim(\phi'^{a'})}
\end{aligned}
\right.
\end{align*}
for $a \in A_\phi$ and $a' \in A_{\phi'}$.
\end{thm}
The GP conjecture for the symplectic-metaplectic cases 
was proven by \cite{At} when $c=1$.
For general $c$, it follows from \cite[Proposition 18.1]{GGP} and the case when $c=1$. 
\vskip 5pt

Finally, we state the GP conjecture for the skew-hermitian cases.
\begin{thm}[GP conjecture for the skew-hermitian cases]\label{GGP-SH}
Suppose that $E \not=F$.
Let $W_n$ be a skew-hermitian space.
For a character $\chi$ of $E^\times$ such that $\chi|F^\times=\omega_{E/F}$, 
we denote by $\omega_{\psi,\chi}$ the Weil representation of $\U(W_n)$ 
associated to $\psi$ and $\chi$.
For $\phi, \phi' \in \Phi_\temp(\U(W_n))$, 
there exists a unique pair $(\pi,\pi') \in \Pi_\phi \times \Pi_{\phi'}$ 
such that
$\pi$ and $\pi'$ are representations of the same group $\U(W_n^\bullet)$ and
\[
\Hom_{\U(W_n^\bullet)}(\pi \otimes \pi', \omega_{\psi,\chi})\not=0.
\]
Moreover, $\iota(\pi)$ and $\iota(\pi')$ are by
\begin{align*}
&\left\{
\begin{aligned}
\iota(\pi)(a) &=
\ep(\phi^{a} \otimes \phi' \otimes \chi^{-1}, \psi^E_{2}),\\
\iota(\pi')(a') &
=\ep(\phi \otimes \phi'^{a'} \otimes \chi^{-1}, \psi^E_{2}) 
\end{aligned}
\right.
\end{align*}
for $a \in A_\phi$ and $a' \in A_{\phi'}$.
\end{thm}
The GP conjecture for the skew-hermitian cases 
was proven by \cite{GI2}.
We also use the following form.
\begin{cor}\label{GGP-HS}
Let notation be as above.
For $\phi, \phi' \in \Phi_\temp(\U(W_n))$, 
there exists a unique pair $(\pi,\pi') \in \Pi_\phi \times \Pi_{\phi'}$ 
such that
$\pi$ and $\pi'$ are representations of the same group $\U(W_n^\bullet)$ and
\[
\Hom_{\U(W_n^\bullet)}(\pi \otimes \pi', \overline{\omega_{\psi,\chi}})\not=0.
\]
Moreover, $\iota(\pi)$ and $\iota(\pi')$ are given as follows:
\begin{align*}
\left\{
\begin{aligned}
\iota(\pi)(a) &=
\omega_{E/F}(-1)^{\dim(\phi^a)} \cdot \ep(\phi^{a} \otimes \phi' \otimes \chi, \psi^E_{2}),\\
\iota(\pi')(a') &= \omega_{E/F}(-1)^{\dim(\phi'^{a'})} \cdot 
\ep(\phi \otimes \phi'^{a'} \otimes \chi, \psi^E_{2}) 
\end{aligned}
\right.
\end{align*}
for $a \in A_\phi$ and $a' \in A_{\phi'}$.
\end{cor}
\begin{proof}
Since $\pi$ and $\pi'$ are tempered,
we have $\pi^\vee=\overline{\pi}$ and $\pi'^\vee=\overline{\pi'}$.
The assertion follows from Theorem \ref{GGP-SH} and Proposition \ref{contra}.
\end{proof}
\vskip 5pt

We also need the following lemma.
\begin{lem}\label{Lemma 12.5}
Let $V_m$ be a Hermitian space of dimension $m$ and 
$W_n$ be a skew-Hermitian space of dimension $n$.
Put $V_{m+1}=V_m \oplus L$ for some line $L$.
If $E=F$, we set $G(W_n)$ and $G'(W_n)$ to be $\Sp(W_n)$ or $\Mp(W_n')$ such that
$\{G(W_n), G(W_n')\}=\{\Sp(W_n), \Mp(W_n)\}$.
Let $\omega=\omega_{\psi_c}$ or $\omega_{\psi,\chi}$.
\begin{enumerate}
\item
For $\sigma \in \Irr_\temp(H(V_{m+1}))$, there exists $\sigma' \in \Irr_\temp(H(V_m))$ such that
$\Hom_{H(V_m)}(\sigma \otimes \sigma', \C)\not=0$.
\item
For $\pi \in \Irr_\temp(G(W_n))$, there exists $\pi' \in \Irr_\temp(G'(W_n))$ such that
$\Hom_{G(W_n)}(\pi \otimes \pi', \omega)\not=0$.
\end{enumerate}
\end{lem}
\begin{proof}
The proof is similar to that of Lemma 12.5 in \cite{GS}.
The absolutely convergence of double integrals which we need 
are proven in \cite{II} for orthogonal cases, \cite{Ha} for hermitian cases, 
\cite{X2} for symplectic-metaplectic cases,
and \cite{X1} for skew-hermitian cases.
\end{proof}

\section{Prasad's conjectures}\label{Pconj}
In this appendix, we state Prasad's conjectures \cite{P}, which are the other highly non-trivial result.
\vskip 5pt

Let $(V_m,W_n)$ be as in \S \ref{spaces}.
We have fixed a non-trivial additive character $\psi$ of $F$, and
$\delta \in E^\times$ such that $\tr_{E/F}(\delta)=0$ if $E\not=F$.
Recall that we put
\[
\psi^E_c(x)=\psi(\half{c}\tr_{E/F}(\delta x))
\]
for $x\in E$ and $c \in F^\times$.
If $c=1$, we simply write $\psi^E=\psi^E_1$.
For a representation $\phi$ of $\WD_E$, we write
$\ep(\phi,\psi_c^E)=\ep(1/2, \phi,\psi_c^E)$.
\vskip 5pt

First, we state Prasad's conjecture for the equal rank case:
\begin{thm}[Prasad's conjecture for the equal rank case]\label{PE}
Assume that $E\not=F$ and $m=n$.
Hence $G(W_n)=\U(W_n)$ and $H(V_m^\pm)=\U(V_n^\pm)$.
Let $\pi \in \Irr(\U(W_n))$ with $L$-parameter $(\phi,\eta)$.
Then we have the following:
\begin{enumerate}
\item
There is a unique pure inner form $U(V_n^\bullet)$ of $U(V_n)$ such that 
$\Theta_{V_n^\bullet,W_n}(\pi)$ is nonzero.
\item
For given $U(V_n^\bullet)$, the theta lift $\Theta_{V_n^\bullet,W_n}(\pi)$ is nonzero 
if and only if
\[
\ep(\phi \otimes \chi_V^{-1}, \psi^E_2) = 
\omega_{E/F}(\delta^{-n} \cdot \disc(V_n^\bullet) \cdot \disc(W_n)).
\]
\item
Suppose $\Theta_{V_n^\bullet,W_n}(\pi)$ is nonzero.
Let $(\theta(\phi),\theta(\eta))$ be the $L$-parameter of $\theta_{V_n^\bullet,W_n}(\pi)$. 
Then $\theta(\phi)=\phi \otimes \chi_V^{-1}\chi_W$. 
In particular, we have a canonical identification $A_\phi = A_{\theta(\phi)}$.
Moreover, we have
\[
\theta(\eta)(a)/\eta(a)=\ep(\phi^{a}\otimes\chi_V^{-1},\psi^E_2)
\]
for $a \in A_\phi=A_{\theta(\phi)}$.
\end{enumerate}
\end{thm}
\vskip 5pt

Next, we state Prasad's conjecture for the almost equal rank case.
If $E=F$ and $\epsilon=-1$, then
$G(W_n)=\Oo(W_n)$ and $H(V_m)=\Sp(V_m)$.
Recall that for $\pi \in \Irr(\Oo(W_n))$, 
we may consider the two theta lifts 
$\Theta_{V_m,W_n}(\pi)$ and $\Theta_{V_m,W_n}(\pi \otimes \det)$.
\begin{thm}[Prasad's conjecture for the almost equal rank case]\label{PA}
Assume that $l=n-m+\epsilon_0=-1$.
Let $\pi \in \Irr(G(W_n))$ with $L$-parameter $(\phi,\eta)$.
Then we have the following:
\begin{enumerate}
\item[($\mathrm{i}$)]
Suppose that $\phi$ does not contain $\chi_V$.
\begin{enumerate}
\item
For any pure inner form $H(V_m^\bullet)$ of $H(V_m)$, 
the theta lift $\Theta_{V_m^\bullet,W_n}(\pi)$ is nonzero. 
\item
Let $(\theta(\phi),\theta(\eta))$ be the $L$-parameter of $\theta_{V_m^\bullet,W_n}(\pi)$.
Then $\theta(\phi) = (\phi \otimes \chi_V^{-1}\chi_W) \oplus \chi_W$. 
Hence there is a canonical injection $A_\phi \hookrightarrow A_{\theta(\phi)}$.
\item
We have $[A_{\theta(\phi)}:A_\phi]=2$. 
\item
The character $\theta(\eta)$ of $A_{\theta(\phi)}$ satisfies
\[
\theta(\eta)|A_\phi = \eta.
\]
\end{enumerate}
\item[($\mathrm{ii}$)]
Suppose that $\phi$ contains $\chi_V$.
\begin{enumerate}
\item
Exactly one of two theta lifts $\Theta_{V_m,W_n}(\pi)$ and $\Theta_{V'_m,W_n}(\pi)$
(or $\Theta_{V_m,W_n}(\pi)$ and $\Theta_{V_m,W_n}(\pi \otimes \det)$)
is nonzero.
\item
$\Theta_{V_m^\bullet,W_n}(\pi)$ is nonzero if and only if
\[
\left\{
\begin{aligned}
&\eta(z_\phi + e_1) = 1 \iif \text{$G(W_n)=\Oo(W_n)$ and $H(V_m)=\Sp(V_m)$}, \\
&V_m^\bullet \in \VV^{\eta(z_\phi)} \other.
\end{aligned}
\right.
\]
Here, $e_1$ is the element in $A_\phi$ corresponding to $\chi_V$.
\item
Suppose that $\Theta_{V_m^\bullet,W_n}(\pi)$ is nonzero.
Let $(\theta(\phi),\theta(\eta))$ be the $L$-parameter of $\theta_{V_m^\bullet,W_n}(\pi)$.
Then $\theta(\phi) = (\phi \otimes \chi_V^{-1}\chi_W) \oplus \chi_W$.
Hence there is a canonical injection $A_\phi \hookrightarrow A_{\theta(\phi)}$.
\item
We have $[A_{\theta(\phi)}:A_\phi]=1$.
\item
The character $\theta(\eta)$ of $A_{\theta(\phi)}$ satisfies
\[
\theta(\eta)|A_\phi = \eta.
\]
\end{enumerate}
\end{enumerate}
\end{thm}

Prasad's conjectures (Theorems \ref{PE} and \ref{PA}) 
are established by \cite{GI2} when $E\not=F$. 
When $E=F$, 
Theorem \ref{PA} is proven by \cite{At} and \cite{AG}.
\vskip 5pt

By the conservation relation (Proposition \ref{cons}), 
for any $\pi \in \Irr(G(W_n))$, we have
\[
m^\down(\pi) \leq n+\epsilon_0+1.
\]
If $m^\down(\pi)= n+ \epsilon_0 +1$, then $m^\up(\pi)=m^\down(\pi)=n+ \epsilon_0 +1$.
Namely, both of two theta lifts $\Theta_{V_m^\bullet,W_n}(\pi)$ with $m=n+ \epsilon_0 +1$ 
are nonzero.
In this case, $\phi$ does not contain $\chi_V$ by Theorem \ref{PA}.


\end{document}